\documentclass[11pt,reqno]{amsart}
\usepackage{amssymb,mathrsfs,graphicx,enumerate}
\usepackage{amsmath,amsfonts,amssymb,amscd,amsthm,bbm}
\allowdisplaybreaks
\usepackage[retainorgcmds]{IEEEtrantools}
\usepackage{colortbl}
\allowdisplaybreaks
\DeclareMathOperator*{\bs}{\bigstar}
\title[An algebraic approach for weak coupling of LT models]{An algebraic approach for the weak coupling of multiple Lohe tensor models}

\author[S.-Y. Ha]{Seung-Yeal Ha}
\address[Seung-Yeal Ha]{\newline Department of Mathematical Sciences and Research Institute of Mathematics, \newline Seoul National University, Seoul 08826, Republic of Korea}
\email{syha@snu.ac.kr}

\author[D. Kim]{Dohyun Kim}
\address[Dohyun Kim]{\newline School of Mathematics, Statistics and Data Science, \newline Sungshin Women's University, Seoul 02844, Republic of Korea}
\email{dohyunkim@sungshin.ac.kr}

\author[H. Park]{Hansol Park}
\address[Hansol Park]{\newline Department of Mathematics,\newline
 Simon Fraser University,  8888 University Dr, Burnaby, BC V5A 1S6, Canada}
\email{hansol960612@snu.ac.kr}

\newtheorem{theorem}{Theorem}[section]
\newtheorem{lemma}{Lemma}[section]
\newtheorem{corollary}{Corollary}[section]
\newtheorem{proposition}{Proposition}[section]
\newtheorem{remark}{Remark}[section]

\newtheorem{definition}{Definition}[section]

\newcommand{\bbr}{\mathbb R}

\newcommand{\bbs}{\mathbb S}
\newcommand{\bbz}{\mathbb Z}

\newcommand{\bbn}{\mathbb N}
\newcommand{\mi}{\mathrm{i}}
\newcommand{\kp}{\kappa}
\newcommand{\tF}{{\textup{F}}}
\newcommand{\SOn}{{\mathbf{SO}}(n) }
\newcommand{\SOd}{{\mathbf{SO}}(d) }

\renewcommand{\d}{{\textup{d}}}
\newcommand{\dt}{{\textup{d}t}}
\newcommand{\bfd}{\mathbf d}
\newcommand{\dg}{\dagger}
\newcommand{\veps}{\varepsilon}

\newcommand{\bd}{\mathbf{d}}

\newcommand{\bbc}{\mathbb C}

\begin{document}

\date{\today}


\thanks{\textbf{Acknowledgment.} The work of S.-Y. Ha is supported by National Research Foundation of Korea (NRF-2020R1A2C3A01003881), and the work of D. Kim was supported by the National Research Foundation of Korea (NRF) grant funded by the Korea government (MSIT) (No.2021R1F1A1055929).}

\begin{abstract}
We present a systematic algebraic approach for the weak coupling of Cauchy problems to multiple Lohe tensor models. For this, we identify an admissible Cauchy problem to the Lohe tensor (LT) model with a characteristic symbol  consisting of four tuples in terms of a size vector, a natural frequency tensor, a  coupling strength tensor and admissible initial configuration. In this way, the collection of all admissible Cauchy problems to the LT models is equivalent to the space of characteristic symbols. On the other hand, we introduce a binary operation, namely ``{\it fusion operation}" as a binary operation between the characteristic symbols. It turns out that the fusion operation satisfies the associativity and admits the identity element in the space of characteristic symbols which naturally forms a monoid. By virtue of the fusion operation, the weakly coupled system of multiple LT models can be obtained by applying the fusion operation of multiple characteristic symbols corresponding to the LT models. As a concrete example, we consider a weak coupling of the swarm sphere model and the Lohe matrix model, and provide a sufficient framework leading to emergent dynamics to the proposed weakly coupled model. 
\end{abstract}

\maketitle \centerline{\date}


\section{Introduction} \label{sec:1}
\setcounter{equation}{0}
Collective behaviors of interacting many-body systems have received a lot of  attention from various scientific disciplines in engineering and sciences, e.g., synchronization and swarming behaviors in biology \cite{B-B, Pe, T-B-L, T-B, Wi1, Wi2}, decentralized control in engineering \cite{L-P-L-S, P-L-S, Re}, non-convex stochastic optimization algorithms in machine learning community \cite{C-C-T-T, C-J-L-Z, F-H-P-S1, F-H-P-S2, K-T, K-E, P-T-T-M}, etc. However, despite  their ubiquitous appearance in nature, its systematic model-based studies were begun only   a half-century ago  by Winfree \cite{Wi2} and Kuramoto \cite{Ku2} whose seminal works have attracted several researchers and serves as milestones of collective dynamics as a major research subject in applied mathematics, control theory, statistical physics, etc. Since then, several phenomenological models have been proposed for analytical and practical treatments for diverse applications,  to name a few, the Cucker-Smale model \cite{C-S}, the swarm sphere model \cite{C-C-H, C-H, H-K-L-N, M2, M1, M20, Ol, Zhu1, Zhu2}, matrix aggregation models \cite{B-C-S, D-F-M-T, D,Lo-09}, etc. Among them, we are mainly interested in the LT model which was recently introduced by Ha and Park \cite{H-P3}. To set up the stage, we first begin with several jargons and notation related to tensors and their algebraic operations following the presentation from \cite{H-P3}. 
 
 A complex rank-$m$ tensor can be visualized as a  multi-dimensional array of complex numbers with multi-indices. The {\it ``rank}'' (or  {\it ``order}'') of a tensor is the number of indices,  for example, scalars, vectors and matrices corresponding to rank-0, 1 and 2 tensors, respectively.  For notational simplicity, we set 
\[
[k] := \{1, \cdots, k \} \subset \bbn, \quad \mbox{for $k \in \bbn$}.
\]
Let $T$ be a rank-$m$ tensor with a size $d_1 \times \cdots \times d_m$.  Then, we denote $(\alpha_1, \cdots, \alpha_m)$-th component of  $T$ by $[T]_{\alpha_1 \cdots \alpha_m}$, and we also denote $\overline{T}$ by the rank-$m$ tensor whose components are simply the complex conjugate of the corresponding elements in $T$:
\[ [\overline{T}]_{\alpha_1 \cdots \alpha_m} :=\overline{[T]_{\alpha_1 \cdots \alpha_m}},  \quad \alpha_i \in[d_i], \quad i \in [m]. \]
In other words, each component of $\overline T$ is defined as the complex conjugate of the corresponding element of $T$.  Let ${\mathcal T}_m(\bbc; d_1 \times\cdots\times d_m)$ be the collection of all rank-$m$ tensors with size $d_1 \times\cdots\times d_m$ and complex entries. Then, it is easy to see that $ {\mathbb C}^{d_1 \times \cdots \times d_m}$ is a complex vector space. Throughout the paper,  for notational convenience, we use a handy notation:
\[  {\mathbb C}^{d_1 \times \cdots \times d_m}:= {\mathcal T}_m(\bbc; d_1 \times\cdots\times d_m). \]
Furthermore, we also introduce the following handy notation:~for $T \in \bbc^{d_1 \times \cdots \times d_m}$ and $A \in \bbc^{d_1 \times \cdots \times d_m \times d_1 \times \cdots \times d_m}$, we set
\begin{align*}
\begin{aligned}
& [T]_{\alpha_{*}}:=[T]_{\alpha_{1}\alpha_{2}\cdots\alpha_{m}}, \quad [T]_{\alpha_{*0}}:=[T]_{\alpha_{10}\alpha_{20}\cdots\alpha_{m0}},  \quad  [T]_{\alpha_{*1}}:=[T]_{\alpha_{11}\alpha_{21}\cdots\alpha_{m1}}, \\
&  [T]_{\alpha_{*i_*}}:=[T]_{\alpha_{1i_1}\alpha_{2i_2}\cdots\alpha_{mi_m}}, \quad [T]_{\alpha_{*(1-i_*)}}:=[T]_{\alpha_{1(1-i_1)}\alpha_{2(1-i_2)}\cdots\alpha_{m(1-i_m)}}, \\
&  [A]_{\alpha_*\beta_*}:=[A]_{\alpha_{1}\alpha_{2}\cdots\alpha_{m}\beta_1\beta_2\cdots\beta_{m}}.
\end{aligned}
\end{align*}
Moreover, we can endow an inner product $\langle \cdot, \cdot \rangle_\tF$, namely {\it ``Frobenius inner product''} and its induced norm $\| \cdot \|_\tF$ on $ \bbc^{d_1 \times \cdots \times d_m}$:~for $T, S \in  \bbc^{d_1 \times \cdots \times d_m}$, 
\[ \langle T, S \rangle_\tF := \sum_{\alpha_* \in \prod_{i=1}^{m} \{1, \cdots, d_i\} } [\overline{T}]_{\alpha_*} [S]_{\alpha_*}, \quad \|T \|_\tF^2 := \langle T, T \rangle_\tF. \]
Let  $A_j$ be a {\it block skew-Hermitian} rank-$2m$ tensor with size $(d_1 \times\cdots\times d_m) \times (d_1 \times \cdots\times d_m)$ satisfying the relation:
\begin{equation*}
[A_j]_{\alpha_{*0} \alpha_{*1}} = -[\overline {A}_j]_{\alpha_{*1} \alpha_{*0}}.
\end{equation*}
In other words, if two blocks with the first $m$-indices are interchanged with the rest $m$-indices, then it results in the negation of the complex conjugate of itself. \newline

Now, we are ready to introduce the LT model on the finite ensemble $\{T_j \}_{j=1}^{N} \subset \bbc^{d_1 \times \cdots \times d_m}$ of rank-$m$ tensors: 
\begin{align}\label{LT}
\begin{cases}
[\dot{T}_j]_{\alpha_{*0}}=[A_j]_{\alpha_{*0}\alpha_{*1}}[T_j]_{\alpha_{*1}}\\
\displaystyle + \sum_{i_*\in\{0, 1\}^m}\frac{\kappa_{i_*}}{N}\sum_{k=1}^N\left([T_k]_{\alpha_{*i_*}}[\bar{T}_j]_{\alpha_{*1}}[T_j]_{\alpha_{*(1-i_*)}}-[T_j]_{\alpha_{*i_*}}[\bar{T}_k]_{\alpha_{*1}}[T_j]_{\alpha_{*(1-i_*)}}\right)~~ t>0,\\
T_j(0)=T_j^0\in \bbc^{d_1\times d_2\times\cdots \times d_m},\quad \|T_j^0\|_\tF=1,
\end{cases}
\end{align}
where $A_j\in\bbc^{d_1\times \cdots\times d_m\times d_1\times \cdots\times d_m}$ is a rank-$2m$ natural frequency tensor satisfying the block skew-hermitian property:
\begin{equation*}
[A_j]_{\alpha_{*0}\alpha_{*1}}=-[\overline{A}_j]_{\alpha_{*1}\alpha_{*0}},
\end{equation*}
and  we used the Einstein summation convention for repeated indices in the right-hand side of \eqref{LT}.

Although the LT model \eqref{LT} looks pretty complicated, coupling terms inside the parenthesis of \eqref{LT} are designed to incorporate the coupling terms for the Kuramoto model, the swarm sphere model and the Lohe matrix model, and they certainly have ``gain cubic term$-$loss cubic term''  structure. These careful designs of couplings (or interactions) yield  the existence of a constant of motion:
\[  \|T_j(t) \|_\tF = \|T_j(0) \|_\tF, \quad t \geq 0,\quad j\in [N].\]
In this paper, we are interested in the weak coupling of the LT models. For instance, the following question naturally arises:
 \begin{quote}
``{\it How to couple two swarm sphere models or two Lohe matrix models {\it weakly}? } 
 \end{quote}
 Here, the adverb ``weakly" means that  when we coupled two aggregation models, we do not change the underlying interaction structures, and the effect of other states is reflected on parameters such as coupling strengths or natural frequency tensors. As briefly discussed in Section \ref{sec:2}, weak couplings of two swarm sphere models and two Lohe matrix models are recently investigated in \cite{Lo-20} and \cite{H-K-P2} under the names of  ``{\it double sphere model''} and ``{\it double  matrix model''}, respectively. In fact, the aforementioned weak couplings are for the LT models with the same rank structures. Then, how about the weak coupling of the swarm sphere model and the Lohe matrix model? We will see the affirmative answer for this question at the end of this paper. \newline
 
In this paper, we address the following two simple questions: 
\vspace{0.1cm}
\begin{itemize}
\item
(Q1):~For a given Cauchy problem to  the LT model \eqref{LT},  which set of parameters characterize the Cauchy problem? 
\vspace{0.2cm}
\item
(Q2):~For a collection of Cauchy problems to the LT model, can we define a suitable weak coupling of two Cauchy problems to the LT models for possibly different rank tensors?
\end{itemize}

Next, we briefly outline   the main results dealing with affirmative answers for (Q1) and (Q2) by developing an algebraic approach for the weak coupling of two Cauchy problems. First, we introduce {\it the characteristic symbol} denoted by $\mathfrak{C}$ for each admissible Cauchy problem to the LT model \eqref{LT} consisting of {\it four components}: a size vector $\mathbf{d}$, a coupling strength tensor $\mathfrak{K}$, set of  natural frequency tensors $\{A_j\}$ and initial configuration $\{T_j^0\}$:
\[
\Big \{ \mbox{Cauchy problems \eqref{LT} for the LT model} \Big \} \quad \Longleftrightarrow \quad  \mathfrak{C}:=(\mathbf{d}, \mathfrak{K}, \{A_j\}, \{T_j^0\}).
\]
If a solution $T_j$ to \eqref{LT} belongs to $\bbc^{d_1\times d_2\times \cdots \times d_m}$, then one can associate a vector in $\bbn^m :=\underbrace{\bbn \times \cdots \times \bbn}_{m~\mbox{times}}$:
\[
\mathbf{d} := (d_1,d_2,\cdots,d_m) \in \bbn^m,
\]
and we call it as the {\it size vector} of \eqref{LT}. For the LT model \eqref{LT}, there are $2^m$ coupling strengths denoted $\{\kp_{i_*}\}_{i_* \in \{0,1\}^m}$. Thus, we collect all the  coupling strengths and named it as the {\it coupling strength tensor}:
\[
\mathfrak{K} :=(\kappa_{i_*})_{i_*\in\{0,1\}^m}.
\]
Then, $\mathfrak{K}$ is a rank-$m$ real tensor. Lastly, the sets of natural frequency tensors and initial configuration are included in the characteristic symbol (see Definition \ref{D3.2}). Thus, once the Cauchy problem to the  LT model is given, then we can uniquely assign the associated characteristic symbol up to some equivalence relations. On the contrary, if the characteristic symbol is given, we can find the corresponding Cauchy problem to the LT model. In this way, we can naturally consider  a bijective map  between the collection of all admissible Cauchy problems to the LT models and the collection of characteristic symbols. 

Second, we introduce an algebraic operation denoted by $\star$ which we call the {\it fusion operation} on the set of characteristic symbols, and we show that the proposed fusion operation forms a monoid on its domain, i.e., it is associative and admits the identity element.  However, an inverse element does not exist. Thus, the set of characteristic symbols equipped with $\star$ forms a monoid. Thanks to this algebraic machinery on the space of characteristic symbols, we can weakly couple two Cauchy problems to the LT models for rank-$m_1$ tensors $\{T_j^1\}$ and rank-$m_2$ tensors $\{T_j^2\}$. We call the resulting coupled system as {\it the double tensor model} where the interaction between $T_j^1$ and $T_j^2$ is involved in  a classical way (see \eqref{C-1}). After tedious manipulation, we can also find a governing equation for $T_j:= T_j^1 \otimes T_j^2$ (see \eqref{C-3}). Thus, we would say that the equation for $T_j$ is obtained by {\it coupling} two LT models. Since the LT model corresponds to the unique characteristic symbol up to an equivalence relation, coupling two LT models allow us to define the fusion operation. For a detailed discussion, we refer the reader to Section \ref{sec:4.1}.  \newline

The rest of the paper is organized as follows. In Section \ref{sec:2}, we review two weakly coupled models, namely the double sphere model and the double matrix model, and study their emergent dynamics and relation with the Cauchy problem to the LT model.  In Section \ref{sec:3}, we provide a characterization of the Cauchy problem to the LT model. For this, we introduce a characteristic symbol for the given Cauchy problem, and as concrete examples, we discuss characteristic symbols  for the LT model on the space of low-rank tensors. In Section \ref{sec:4}, we provide the double tensor model which can be obtained as a weak coupling of two LT models. To motivate this weak coupling, we recall a coupled system between the Kuramoto model and the swarm sphere model from suitable parametrization of the LT model on the unitary group of degree 2 $\mathbf{U}(2)$. In Section \ref{sec:5}, we present the multiple tensor model which couples arbitrary numbers of  LT models. For this explicit construction, we introduce a fusion operation between characteristic symbols, and study the reductions of the multiple tensor model to the low-rank tensors and a framework for the gradient flow formulation. In Section \ref{sec:6}, we study weak couplings for the LT models on the ensemble of low-rank tensors with different ranks, and study several emergent dynamics of the proposed coupled model. Finally, Section \ref{sec:7} is devoted to a brief summary of the paper and future direction.

%
%
%
%
%
%
%

\section{Preliminaries}\label{sec:2}
In this section, we discuss two concrete examples for the weak couplings of first-order aggregation models  for low-rank tensors to motivate the subsequent sections. More precisely, we deal with weak couplings for the following two models:
\begin{itemize}
\item
Weak coupling of two swarm sphere models with distinct sets of free flows for rank-one real tensors with the same size.
\vspace{0.1cm}
\item
Weak coupling of two Lohe matrix models with distinct sets of free flows for rank-two unitary tensors with the same size.
\end{itemize}
For each case, we present the corresponding weakly coupled systems and review their  emergent dynamics and relation to the  LT model \eqref{LT}. 

\subsection{Weak coupling of two swarm sphere models} \label{sec:2.1} The swarm sphere model describes aggregation phenomenon on the unit sphere embedded in Euclidean space. 
Consider two ensembles $\{u_j\}$ and $\{v_j\}$ with the same cardinalities whose dynamics are governed by the swarm sphere models with different set of free flows $\{\Omega_j \}$ and $\{\Lambda_j \}$. In the absence of the intra-coupling between two ensembles,  the dynamics of each ensemble is governed by two independent swarm sphere models: 
\begin{equation}
\begin{cases} \label{B-0}
\displaystyle \dot{u}_j =\Omega_j u_j+\frac{1}{N}\sum_{k=1}^N  \kp (u_k-\langle u_k, u_j\rangle u_j),  \\
\displaystyle \dot{v}_j =\Lambda_jv_j+\frac{1}{N}\sum_{k=1}^N  \kp (v_k-\langle v_k, v_j\rangle v_j), \quad j\in [N], 
\end{cases}
\end{equation}
where $\Omega_j$ and $\Lambda_j$ are $d_1\times d_1$ and $d_2\times d_2$ skew-symmetric matrices with sizes $d_1 \times d_1$ and $d_2 \times d_2$, respectively: 
\begin{equation*} \label{B-1}
\Omega_j^\top = -\Omega_j, \quad \Lambda_j^\top = -\Lambda_j, \quad  j \in [N], 
\end{equation*}
and $\kp$ denotes the nonnegative (uniform) coupling strength.  The emergent dynamics of \eqref{B-0} has been extensively in a series of literature \cite{C-C-H, C-H, C-H2, H-K-L-N, Lo-09, Ol, M2, M1, M20, Zhu1, Zhu2}. \newline

Now, we consider weak couplings between two ensembles $\{ u_j \}$ and $\{v_j \}$. In this case,  a natural question to arise will be 
\begin{quote}
``{\it How to couple the dynamics for $u_j$ and  $v_j$ in a suitable way so that the resulting weakly coupled model can exhibit emergent collective dynamics}?"
\end{quote}
Of course, there might be numerous ways to couple the dynamics of $u_j$ and $v_j$. However, we require that the resulting coupled model needs to be consistent with previously well-known cases. In this direction,   Lohe in \cite{Lo-20} first proposed the double sphere model which associates the dynamics of $u_j$ and $v_j$ in \eqref{B-0}. To be specific, the coupled dynamics is governed by the Cauchy problem to the double sphere model on the product of two spheres $\bbs^{d_1-1} \times \bbs^{d_2-1}$:
\begin{align}\label{B-1}
\begin{cases}
\displaystyle\dot{u}_j=\Omega_j u_j+\frac{1}{N}\sum_{k=1}^N \kp \langle v_j, v_k\rangle(u_k-\langle u_k, u_j\rangle u_j), \quad t > 0, \\
\displaystyle\dot{v}_j=\Lambda_jv_j+\frac{1}{N}\sum_{k=1}^N\kp \langle u_j, u_k\rangle(v_k-\langle v_k, v_j\rangle v_j),\\
(u_j, v_j)(0)=(u_j^0,v_j^0)\in\bbs^{d_1-1}\times\bbs^{d_2-1}, \quad j \in [N].
\end{cases}
\end{align}
Compared to the decoupled system \eqref{B-0}, weak coupling is made via the change of  constant coupling strengths into state-dependent couplings:
\begin{equation*} \label{B-1-1}
(\kappa, \kappa) \quad  \longrightarrow \quad  (\kappa \langle v_j, v_k\rangle, ~ \kappa \langle u_j, u_k\rangle).
\end{equation*}
In other words, the effect of other ensemble is reflected on the coupling strength, not on the interaction terms. Thus, we call \eqref{B-1} as a \textit{weakly coupled} model. Complete aggregation of \eqref{B-1} can be achieved under suitable conditions on system parameters and initial data as follows.
\begin{proposition}\label{P2.1}
\emph{\cite{H-K-P1}} 
Suppose system parameters and initial data $\{ (u_i^0, v_i^0) \}$ satisfy
\[ \kappa > 0, \quad \Omega_i = \Omega, \quad \Lambda_i = \Lambda, \quad i \in [N], \quad  \min_{1\leq i, j\leq N }\langle u_i^0, u_j^0\rangle>0,\quad  \min_{1\leq i, j\leq N}\langle v_i^0, v_j^0\rangle>0,
\]
where $\Omega$ and $\Lambda$ are $d_1\times d_1$ and $d_2\times d_2$ skew-symmetric matrices respectively, and let $\{(u_i,v_i)\}$ be a solution to system \eqref{B-1}.  Then, we have
\[
\lim_{t\to\infty} \max_{1\leq i,j\leq N}  |u_i(t)-u_j(t) |=0 \quad \textup{and} \quad \lim_{t\to\infty}  \max_{1\leq i,j\leq N}  |v_i(t)-v_j(t) |=0.
\]
Moreover, convergence rate is exponential. 
\end{proposition}

\vspace{0.5cm}

Next, we explain how the double sphere model \eqref{B-1} can be related to the  LT model \eqref{LT} as a special case. For this, we consider the following special set-up with low-rank tensors:
\begin{equation*}
m=2,\quad \kp_{00} = \kp_{11} =0,\quad \kp_{01} = \kp_{10} = \kp.
\end{equation*}
Under this set-up, the LT model in a mean-field form becomes the generalized Lohe matrix model in \cite{H-P3}:
\begin{equation}\label{B-2} 
\begin{cases}
\displaystyle {\dot T}_j =A_j T_j +\kappa(T_cT_j^\dagger T_j -T_j T_c^\dagger T_j)+\kappa(T_j T_j^\dagger T_c-T_j T_c^\dagger T_j), \quad t > 0, \\
\displaystyle T_j(0)=T_j^0,\quad  j \in [N],
\end{cases}
\end{equation}
where $T_c : = \frac{1}{N} \sum_{k=1}^{N} T_k$ is the average state of $\{T_k\}$. In the following proposition, we see that how \eqref{B-1} and \eqref{B-2} can be regarded as equivalent systems under {\it well-prepared} natural frequency tensors and initial data. 
\begin{proposition} \cite{H-K-P1}  \label{P2.2} 
Systems \eqref{B-1} and \eqref{B-2} are equivalent in the following sense.
 \begin{enumerate}
 \item
 Suppose $\{(u_j, v_j) \}$ is a solution to \eqref{B-1}. Then, $\{ T_j := u_j \otimes v_j  \}$ is  a solution to \eqref{B-2}  with the following system parameters and  initial data:
 \begin{equation} \label{B-3}
 A_j T_j:= \Omega_j T_j + T_j \Lambda_j^\top, \quad T_j^0 = u_j^0 \otimes v_j^0, \quad j \in [N].
 \end{equation}
  \item
 Suppose $\{ T_j \}$ is a solution to \eqref{B-2}--\eqref{B-3} with completely separable initial data:
 \begin{equation*}  
 T_j^0 =: u_j^0 \otimes v_j^0,  \quad j \in [N],
 \end{equation*}
 for rank-1 real tensors $u_j^0 \in \bbr^{d_1}$ and $v_j^0 \in \bbr^{d_2}$ with unit norms. Then, there exist pair of unit vectors $(u_j(t), v_j(t))$ such that 
 \[  T_j(t) = u_j(t) \otimes v_j(t), \quad t>0,  \quad j \in [N],
 \]
 where $\{ (u_j,v_j) \}$ is a solution to \eqref{B-1} with  the initial data $(u_j, v_j)(0) = (u_j^0,v_j^0)$.
 \end{enumerate}
\end{proposition}
\begin{remark}
The result of Proposition \ref{P2.2} roughly says that the double sphere model can be identified as a subclass of LT model  for rank-2 tensors. 
\end{remark}
From now on, as long as there is no confusion, we sometimes omit the expression ``$j \in [N]$", when we state a dynamical system.

\subsection{Weak coupling of two Lohe matrix models} \label{sec:2.2} Consider two independent Lohe matrix ensembles $\{ U_j \}$ and $\{ V_j \}$ on the   unitary groups $\mathbf{U}(d_1)$ and $\mathbf{U}(d_2)$ whose dynamics are governed by the following system:
\begin{equation} 
\begin{cases} \label{B-3-0}
 \dot{U}_j= -{\mathrm i} H_j U_j+\displaystyle\frac{1}{N}\sum_{k=1}^N \kappa \left( U_kU_j^\dagger U_j -U_jU_k^\dagger U_j\right),  \\
 \dot{V}_j = -{\mathrm i} G_j V_j +\displaystyle\frac{1}{N}\sum_{k=1}^N \kappa \left(V_kV_j^\dagger V_j -V_jV_k^\dagger V_j\right),
\end{cases}
\end{equation}
where $\kappa$ is a nonnegative constant, and $H_j\in\bbc^{d_1\times d_1}$ and $G_j\in \bbc^{d_2 \times d_2}$ are Hermitian matrices. 

As in previous subsection, we consider weak coupling of two subsystems in \eqref{B-3-0}.  Motivated by the relation between \eqref{B-0} and \eqref{B-1}, the double matrix model for $(U_j, V_j)$ has been proposed in \cite{H-K-P2}:
\begin{equation} 
\begin{cases} \label{B-3-1}
 \dot{U}_j= -{\mathrm i} H_j U_j+\displaystyle\frac{\kappa}{N}\sum_{k=1}^N\left(
\langle V_j, V_k\rangle_\tF~U_kU_j^\dagger U_j
-\langle V_k, V_j\rangle_\tF~U_jU_k^\dagger U_j\right), \\
 \dot{V}_j = -{\mathrm i} G_j V_j+\displaystyle\frac{\kappa}{N}\sum_{k=1}^N\left(
\langle U_j, U_k\rangle_\tF~V_kV_j^\dagger V_j
-\langle U_k, U_j\rangle_\tF~V_jV_k^\dagger V_j\right),\\
\end{cases}
\end{equation}
where $\langle A,B\rangle_\tF = \textup{tr} (A^\dagger B)$ is the Frobenius inner product. For the emergent dynamics of \eqref{B-3-1}, we introduce some functionals:
\begin{align*}
\begin{aligned} 
& \mathcal D(\mathcal U) := \max_{1\leq i,j\leq N} \|U_i - U_j \|_\tF,\quad  \mathcal  S(\mathcal U):= \max_{1\leq i,j\leq N} |d_1- d_{ij}|, \quad d_{ij} : = \langle U_i,U_j\rangle_\tF, \\
&\mathcal D(\mathcal V) := \max_{1\leq i,j\leq N} \|V_i - V_j\|_\tF, \quad \mathcal S(\mathcal V):= \max_{1\leq i,j\leq N} |d_2 - c_{ij}|,  \quad  c_{ij} := \langle V_i,V_j\rangle_\tF, \\
& {\mathcal L}({\mathcal U}, {\mathcal V}):=  \mathcal D(\mathcal U)  +\mathcal  D(\mathcal V) + \mathcal S(\mathcal U ) +\mathcal  S(\mathcal V), \\
&\mathcal D(\mathcal H) := \max_{1\leq i,j\leq N } \|H_i - H_j\|_\infty, \quad \mathcal D(\mathcal G) := \max_{1\leq i,j\leq N } \|H_i - H_j\|_\infty,
\end{aligned}
\end{align*}
where $\|A\|_\infty:= \max_{1\leq i,j\leq d} | [A]_{ij}|$ for  $A \in \bbc^{d\times d}$. Then,  emergent dynamics of \eqref{B-3-1} can be achieved under well prepared system parameters and initial data, and summarized in the following two propositions without proofs. 
\begin{proposition} \label{P2.3} \cite{H-K-P2} \emph{(Homogeneous ensemble)} 
Suppose system parameters and initial data satisfy 
\begin{align*} 
& H_j = H, \quad G_j = G, \quad  d_1 \geq d_2 >4\sqrt{d_1}, \\
&{\mathcal L}^0 = {\mathcal L}({\mathcal U}^0, {\mathcal V}^0) <  \frac{-(12 d_1 +27) + \sqrt{ (12d_1 +27)^2 + 48(d_2-4\sqrt{d_1})(3d_1+4)  } }{4(3d_1+4)},
\end{align*}
and let $\{(U_j,V_j)\}$ be a solution to \eqref{B-3-1}. Then, we have 
 \begin{equation*}
 \lim_{t\to\infty}  {\mathcal L}({\mathcal U}(t), {\mathcal V}(t)) = 0.
 \end{equation*}
 \end{proposition}
Next, we present emergence of state-locking and orbital stability of \eqref{B-3-1}.
\begin{proposition} \label{P2.4} \cite{H-K-P2} \emph{(Heterogeneous ensemble)}
Suppose system parameters and initial data satisfy 
\[ d_1 \geq d_2,\quad \mathcal D(\mathcal H)\geq \mathcal D(\mathcal G) > 0, \quad \kp_1 \gg 1,  \quad  \max\{ \mathcal L^0, \tilde{\mathcal L}^0\} \ll 1,  \]
and let $\{(U_i,V_i)\}$ and $\{(\tilde U_i, \tilde V_i)\}$ be any two solutions to \eqref{B-3-1}.  Then, the following assertions hold. 
\begin{enumerate}
\item
(Asymptotic state-locking): there exist state-locked states $\{X_i^\infty\}$ and $\{Y_i^\infty\}$, and unitary matrices $P \in \mathbf{U}(d_1)$ and $Q \in  \mathbf{U}(d_2)$  such that
\[ \lim_{t\to\infty} \| U_i(t) - X_i^\infty P\|_\tF = 0  \quad \textup{and} \quad  \lim_{t\to \infty} \|V_i(t) - Y_i^\infty Q\|_\tF = 0. \]
\item
(Orbital stability): there exist unitary matrices $X^\infty \in {\mathbf U}(d_1)$ and $Y^\infty \in {\mathbf U}(d_2)$  such that
\[  \lim_{t\to\infty} \|\tilde U_i(t) - U_i(t) X^\infty\|_\tF =0 \quad \textup{and} \quad  \lim_{t\to\infty} \|\tilde V_i(t) - V_i(t) Y^\infty\|_\tF =0. \]
\end{enumerate}
\end{proposition}

\vspace{0.2cm}

Similar to Proposition \ref{P2.2}, we finally state the relation between the double matrix model \eqref{B-3-1} and the LT model in the following proposition.
\begin{proposition} \label{P2.5} 
\emph{\cite{H-K-P2}} Systems \eqref{LT} and \eqref{B-3-1} are equivalent in the following sense:
 \begin{enumerate}
 \item
 Suppose $\{(U_i,V_i) \}$ is a solution to \eqref{B-3-1} with the initial data  $\{(U^0_i,V^0_i) \}$. Then, a rank-4 tensor  $T_i$ defined by $T_i  :=U_i \otimes V_i$ is a solution to \eqref{LT} with a well-prepared free flow tensor $A_i$ and the initial data $T_j^0 =  U^0_i \otimes V^0_i$.
 \vspace{0.1cm}
\item
 Suppose a rank-4 tensor $T_i$ is a solution to \eqref{LT}  with a suitable $A_j$ and quadratically separable initial data:
 \begin{equation*} \label{C-3}
 T_i^0 =: U_i^0 \otimes V_i^0, \quad  j \in [N],
 \end{equation*}
 for  rank-2 tensors $U_i^0 \in \bbc^{d_1\times d_2}$ and $V_i^0 \in \bbc^{d_3 \times d_4}$ with unit norms. Then, there exist two matrices $U_i=U_i(t)$ and $V=V_i(t)$ with unit norms such that 
 \[  T_i(t) = U_i(t) \otimes V_i(t), \quad t>0, \]
 where $(U_i,V_i)$ is a solution to \eqref{B-3-1} with the initial data  $(U_i,V_i)(0) = (U_i^0,V_i^0)$.
 \end{enumerate}
  \end{proposition}

\section{Characteristic symbol for the Cauchy problem}\label{sec:3}
\setcounter{equation}{0}
In this section, we provide a concept of characteristic symbol for the Cauchy problem to the LT model \eqref{LT} via system parameters  and initial data. \newline

Recall the Cauchy problem to the LT model for rank-$m$ tensors:
\begin{align}\label{C-1}
\begin{cases}
[\dot{T}_j]_{\alpha_{*0}}=[A_j]_{\alpha_{*0}\alpha_{*1}}[T_j]_{\alpha_{*1}}\\
+\displaystyle\sum_{i_*\in\{0, 1\}^m}\frac{\kappa_{i_*}}{N}\sum_{k=1}^N\left([T_k]_{\alpha_{*i_*}}[\bar{T}_j]_{\alpha_{*1}}[T_j]_{\alpha_{*(1-i_*)}}-[T_j]_{\alpha_{*i_*}}[\bar{T}_k]_{\alpha_{*1}}[T_j]_{\alpha_{*(1-i_*)}}\right),~
t > 0 ,\vspace{0.2cm}\\
T_j(0)=T_j^0\in \bbc^{d_1\times d_2\times\cdots \times d_m},\quad \|T_j^0\|_\tF=1,
\end{cases}
\end{align}
where $A_j\in\bbc^{d_1\times \cdots\times d_m\times d_1\times \cdots\times d_m}$ is a rank-$2m$ natural frequency tensor satisfying block skew-hermitian property:
\begin{equation}\label{C-2}
[A_j]_{\alpha_{*0}\alpha_{*1}}=-[\bar{A}_j]_{\alpha_{*1}\alpha_{*0}},\quad \alpha_{*0} = (\alpha_{10}, \cdots, \alpha_{m0}),\quad \alpha_{*1} = (\alpha_{11}, \cdots, \alpha_{m1}) \in \prod_{i=1}^m[d_i].
\end{equation}

\subsection{Four components of characteristic symbol} \label{sec:3.1} 
In this subsection, we introduce four components consisting of characteristic symbol for  the Cauchy problem \eqref{C-1}--\eqref{C-2}. In what follows, we fix the number of particles $N$. \newline

\noindent$\bullet$ (Size vector and rank):~Since  a solution $T_j$ to \eqref{C-1} belongs to  $\bbc^{d_1\times d_2\times\cdots \times d_m}$, we can naturally associate \textit{a size vector} and \textit{a rank} to \eqref{C-1}:
\begin{align*}
&\mbox{(Size vector)}~~\mathbf{d}:=(d_1, d_2, \cdots, d_m)\in  \bbn^m; \\
&\mbox{(Rank)}~~m:=\textup{number of components in the size vector}.
\end{align*}
For simplicity, we use the notation $m:=\textup{dim}(\mathbf d)$. Here, we assume that $d_i \geq 2$ for $i\in [m]$. Among $d_1,d_2,\cdots,d_m$, if there is one index $k \in [m]$ such that $d_k=1$, for instance, $d_m=1$, then we regard  a rank $m$-tensor $T_i$ as a rank $m-1$-tensor $\tilde T_i$ by the natural reduction:
\[
[\tilde{T}]_{\alpha_1\alpha_2\cdots\alpha_{m-1}}=[T]_{\alpha_1\alpha_2\cdots\alpha_m}.
\]
This basically holds since $\bbc^{d_1\times \cdots d_{m-1}\times 1} =\bbc^{d_1\times \cdots d_{m-1}}$. Thus, there is no information loss from the reduction above. 
In particular, when scalar is given, then we naturally set the rank to be zero, i.e., $m=0$. Note that this is reminiscent of the relation $0!=1$.

\vspace{0.2cm}

\noindent$\bullet$ (Coupling strength tensor): since $\kappa_{i_*}$ is defined for all $i_*\in\{0, 1\}^m$, we call $\mathfrak{K}:=(\kappa_{i_*})_{i_*\in\{0,1\}^m}$  as \textit{a coupling strength tensor} associated to \eqref{C-1}. Note that $\mathfrak{K}$ is a rank-$m$ tensor with $2^m$ elements. 

\vspace{0.2cm}

\noindent$\bullet$ (Set of natural frequency tensors): we denote  $\{A_j\}_{j=1}^N$ by \textit{a set of natural frequency rank-$2m$ tensors}. 

\vspace{0.2cm}

\noindent$\bullet$ (Initial configuration): we set the \textit{initial configuration} to be  $\{T_j^0\}_{j=1}^N$.\\


For the simplicity of presentation, we define the index set and the set of natural frequency tensors. 
\begin{definition} \label{D3.1}
Let $(\mathbf{d} = (d_1,d_2,\cdots,d_m),~m)$ be the pair of a size vector and a rank associated to \eqref{C-1}.
\begin{enumerate}
\item
The index set of  $\bd$ is given as follows:
\[
\mathfrak{I}(\mathbf{d}):=\{\mathbf{e}=(e_1, e_2, \cdots, e_m)\in\bbn^m: 1\leq e_i\leq d_i,\quad   1\leq i\leq m\}.
\]
\item
The set of natural frequency tensors defined on $\bbc^{d_1\times \cdots\times d_m\times d_1\times \cdots\times d_m}$ is as follows:
\[
\mathfrak{F}(\mathbf{d}):=\{A\in \bbc^{d_1\times \cdots\times d_m\times d_1\times \cdots\times d_m}: [A]_{\alpha_{*0}\alpha_{*1}}=-[\overline{A}]_{\alpha_{*1}\alpha_{*0}},\quad \alpha_{*0},~\alpha_{*1}\in \mathfrak{I}(\mathbf{d})\}.
\]
 Lastly, we denote the set of all admissible natural frequency tensors by
\[
\mathfrak{F}:=\bigcup_{\mathbf{d}\in\bigcup_{k=0}^\infty \bbn^k}\mathfrak{F}(\mathbf{d}).
\]
\end{enumerate}
\end{definition}

\vspace{0.2cm}

Now, we are ready to define the characteristic symbol \eqref{C-1} for the Cauchy problem of the LT model consisting  of a  size vector,  a coupling strength tensor, a  set of natural frequency tensors and initial configuration.

\begin{definition} \label{D3.2} 
Let $\mathbf{d}, \mathfrak{K}, \{A_j\}$ and $\{T_j^0\}$ be the size vector, coupling strength tensor, set of natural frequency tensors and initial configuration associated to the Cauchy problem \eqref{C-1}, respectively. Then, we call ${\mathfrak C}:=(\mathbf{d}, \mathfrak{K}, \{A_j\}, \{T_j^0\}) $ as the characteristic symbol for the Cauchy problem \eqref{C-1}. 
\end{definition}
Since the characteristic symbol contains all  the information of a given Cauchy problem for the LT model, Cauchy problem is uniquely determined (under some equivalent relation) by a  characteristic symbol. 

\begin{definition} \label{D3.3}
We define a set of all characteristic symbols with a size vector $\bd$ and a rank $m=\textup{dim}(\mathbf d)$ as follows: for $\mathbf{d}\in\bbn^m$, 
\[
\mathfrak{D}_m^\mathbf{d}:= \left \{(\mathbf{d}, \mathfrak{K}, \{A_j\}, \{T_j^0\}): \mathfrak{K} \in\bbr^{2^m},~ \{A_j\}_{j=1}^N\in \mathfrak{F}(\mathbf{d})^N,~ \{T_j^0\}\in \bbc^{\mathbf{d}} \right \}.
\]
Then, it becomes a set of all characteristic symbols to the LT models with the size vector $\mathbf{d}$. We also denote
\[
\mathfrak{D}_m:=\bigcup_{\mathbf{d}\in\bbn^m} \mathfrak{D}_m^{\mathbf{d}}\quad\textup{and}\quad {\mathfrak D}:=\bigcup_{m=1}^\infty \mathfrak{D}_m.
\]
Note that  $\mathfrak{D}_m$ is a set of all characteristic symbols with rank $m$, and ${\mathfrak D}$ becomes the space of  all characteristic symbols.
\end{definition}

%

\subsection{Explicit examples} \label{sec:3.2} 
In this subsection, we provide explicit examples for the well-known low-rank models that can be reduced from the LT model such as the Kuramoto model, swarm sphere model, etc. It should be mentioned that the low-rank models are {\it not} the LT model as it is, and they are merely the variants of the LT model. Thus, their characteristic symbols would not be unique. 

\vspace{0.5cm}

\noindent $\bullet$~{\bf The Lohe hermitian sphere  model}: First, we begin with the Lohe hermitian sphere model \cite{H-P1} on $\bbc^d$, which is the LT model with rank-1 tensors, i.e., $m=1$:
\begin{align}\label{C-5}
\begin{cases}
\displaystyle\dot{z}_j=\Omega_jz_j+\frac{\kappa_0}{N}\sum_{k=1}^N(z_k-\langle z_k, z_j\rangle z_j)+\frac{\kappa_1}{N}\sum_{k=1}^N(\langle z_j, z_k\rangle-\langle z_k, z_j\rangle)z_j,\quad t>0,\\
z_j(0)=z_j^0\in\bbc^{d},\quad \|z_j^0\|=1,\quad j \in [N],
\end{cases}
\end{align}
where $\Omega_j$ is a $d \times d$   skew-Hermitian matrix, and $\langle w, z\rangle$ is a complex inner product between two complex vectors:
\[
\langle w, z\rangle :=\sum_{\alpha=1}^{d} [\overline{w}]_\alpha [z]_\alpha, \quad \|\omega \| := \sqrt{\langle \omega,\omega\rangle}.
\]
Then, the characteristic symbol for \eqref{C-5} is given as follows: 
\begin{equation*} \label{C-5-1}
\mathfrak{C} =(d, (\kappa_0, \kappa_1), \{\Omega_j\}, \{z_j^0\}).
\end{equation*} 
Note that since $z_j$ is a rank-1 tensor with $d$ components, the size vector becomes $d$, and the coupling strength tensor is also a rank-1 tensor with $2^1$ elements. 

\vspace{0.5cm}

\noindent $\bullet$~{\bf The Kuramoto model}:~Recall the Kuramoto model which is an aggregation model on the space of rank-$0$ tensors:
\begin{equation} \label{C-3}
\begin{cases}
\displaystyle\dot{\theta}_j=\nu_j+\frac{\kappa}{N}\sum_{k=1}^N\sin(\theta_k-\theta_j), \quad t > 0, \\
\theta_j(0)=\theta_j^0, \quad j \in [N].
\end{cases}
\end{equation}
In order to relate the Kuramoto model \eqref{C-3} with the LT model, we set in \eqref{C-5}
\begin{equation} \label{C-3-1}
 z_j = e^{\mi \theta_j},\quad \theta_j \in \bbr,\quad \Omega_j = \mi \nu_j.
\end{equation} 
Then, one can check that \eqref{C-5} with \eqref{C-3-1} reduces to \eqref{C-3}. Since $\theta_j$ is scalar, the characteristic symbol of \eqref{C-3} is 
\[
\mathfrak{C} = \left(\emptyset, \frac{\kappa}{2}  ,\{\mathrm{i}\nu_j\}, \{e^{\mathrm{i}\theta_j^0}\} \right).
\]

\vspace{0.5cm}

\noindent $\bullet$~{\bf The swarm sphere model}: If we restrict our discussion to $\bbr^{d}$ instead of $\bbc^{d}$, then the coupling term involving with $\kappa_1$ disappears due to the symmetry of inner product so that system \eqref{C-5} reduces to the swarm sphere model on $\bbs^{d-1}$ \cite{Ol}: 
\begin{align}\label{C-6}
\begin{cases}
\displaystyle\dot{x}_j=\Omega_j x_j+\frac{\kappa}{N}\sum_{j=1}^N(x_k-\langle x_k, x_j\rangle x_j), \quad t > 0, \\
x_j(0)=x_j^0, \quad \|x_j^0 \| = 1, \quad j \in [N].
\end{cases}
\end{align}
We mention that there are many ways how \eqref{C-6} is derived from \eqref{C-5}, and hence the characteristic symbol of \eqref{C-6} would not be unique. Here we introduce one simple reduction of \eqref{C-5} to \eqref{C-6}. Although the initial data $\{z_j^0\}$ belong to $\bbc^d$ so that $z_j=z_j(t)$ also lies in $\bbc^d$, we assume that
\begin{equation} \label{C-6-1}
z_j^0 \in \bbr^d,\quad \textup{$\Omega_j$ is a $d\times d$ skew-symmetric matrix with real entries},\quad j\in [N].
\end{equation}
Then, it follows from the uniqueness of solutions to \eqref{C-5} that 
\[
z_j (t) \in \bbr,\quad t>0,\quad j\in [N]. 
\]
Thus, one of the characteristic symbol for \eqref{C-6} is 
\[
\mathfrak{C} =(d, (\kappa_0, \kappa_1), \{\Omega_j\}, \{z_j^0\}),
\]
together with \eqref{C-6-1}.

\vspace{0.5cm}

\noindent $\bullet$~{\bf The generalized Lohe matrix  model}:~Consider the LT model \eqref{C-1} for rank-2 tensors:
\begin{align} \label{C-7-1} 
\begin{cases}
\vspace{0.2cm}
\displaystyle\dot{T}_j=A_jT_j+\frac{\kp_{00}}{N} \sum_{k=1}^N ( \textup{tr}(T_j^*T_j) T_k - \textup{tr}(T_k^*T_j)T_j)  + \frac{\kappa_{01}}{N}\sum_{k=1}^N(T_kT_j^\dagger T_j-T_j T_k^\dagger T_j) \\
\hspace{2cm}+\displaystyle \frac{\kappa_{10}}{N}\sum_{k=1}^N(T_jT_j^\dagger T_k-T_j T_k^\dagger T_j)+\frac{\kp_{11}}{N}\sum_{k=1}^N \textup{tr}( T_j^*T_k - T_k^*T_j   )T_j, ~t > 0,\\
T_j(0)=T_j^0\in \bbc^{d_1\times d_2},\quad \|T_j^0\|_\tF=1,\quad j\in [N].
\end{cases}
\end{align}
If we turn off the diagonal elements of the coupling strength tensors, i.e., 
\[ \kappa_{00}=\kappa_{11}=0, \]
then \eqref{C-7-1} reduces to the generalized Lohe matrix model \cite{H-P0}:
\begin{align}\label{C-7}
\begin{cases}
\displaystyle\dot{T}_j=A_jT_j+\frac{\kappa_{01}}{N}\sum_{k=1}^N(T_kT_j^\dagger T_j-T_j T_k^\dagger T_j)+\frac{\kappa_{10}}{N}\sum_{k=1}^N(T_jT_j^\dagger T_k-T_j T_k^\dagger T_j),~t > 0,\vspace{0.2cm}\\
T_j(0)=T_j^0\in \bbc^{d_1\times d_2},\quad \|T_j^0\|_\tF=1,
\end{cases}
\end{align}
where $A_j$ is  a rank-4 tensor, and the term $A_jT_j$ can be understood as a rank contraction:
\[[A_jT_j]_{\alpha\beta}:=[A_j]_{\alpha\beta\gamma\delta}[T_j]_{\gamma\delta}. \]
Then, again it follows from the previous argument that the characteristic symbol for \eqref{C-7} becomes 
\begin{equation} \label{C-7-2}
\mathfrak{C}=\left((d_1, d_2), 
\begin{bmatrix}
0&\kappa_{01}\\
\kappa_{10}&0
\end{bmatrix}
, \{A_j\}, \{T_j^0\}\right).
\end{equation} 

\vspace{0.5cm}

\noindent $\bullet$~{\bf The Lohe matrix  model}: We recall the Lohe matrix model \cite{Lo-09}: 
\begin{equation} \label{C-8}
\begin{cases}
\displaystyle \dot U_j = - \mi H_j U_j + \frac{\kp}{2N} \sum_{k=1}^N \Big( U_k U_j^{\dagger} U_j - U_j U_k^\dagger U_j\Big), \\
U_j(0) = U_j^0 \in \mathbf{U}(d),\quad j \in [N].
\end{cases}
\end{equation}
Note that there exists a $d\times d$ Hermitian matrix $H_j$ such that 
\begin{equation}  \label{C-8-1}
-\mi H_j T_j = A_j T_j.
\end{equation}
If we assume the following relation in \eqref{C-7}:
\begin{equation} \label{C-7-2}
d:=d_1=d_2, \quad 2\kappa:=\kappa_{01}+\kappa_{10}, \quad T_i^0 \in \mathbf{U}(d),
\end{equation}
then \eqref{C-7} with  \eqref{C-8-1}-\eqref{C-7-2} becomes \eqref{C-8}. Thus, as in the swarm sphere model, the characteristic symbol for \eqref{C-8} would not be unique. One parameter family of  characteristic symbols corresponding to \eqref{C-8} can be given as follows: for any $\tilde \kp \in \bbr$, 
\[
\mathfrak{C}=\left((d, d), 
\begin{bmatrix}
0& \frac\kappa2-\tilde{\kappa}\\
\tilde{\kappa}&0
\end{bmatrix}
, \{H_j\}, \{U_j^0\}\right).
\]

%
%
%
%

\section{Weak Coupling of two Lohe tensor type models}\label{sec:4}
\setcounter{equation}{0}
In this section, we present a weak coupling of two Lohe tensor models for {\it different} rank tensors. Recall that two examples given in Section \ref{sec:2} dealt  with the weak coupling  of the LT-type models with the same rank tensors, say weak couplings between two ensembles of  rank-1 tensors, or between two ensembles of rank-2 tensors, respectively. Now, we consider a weak coupling of two Lohe tensor type models for {\it different} rank tensors. To obtain an insight for the weak coupling, we review the weakly coupled model in \cite{Lo-09} consisting of the Kuramoto model and the swarm sphere model, and then  present {\it the double tensor model} for two Lohe tensor models. Finally, we introduce a binary operation between two characteristic symbols for the Cauchy problem of the LT model.

\subsection{Coupling of the Kuramoto model and the swarm sphere model} \label{sec:4.1} 
In this subsection, we discuss a weak coupling between the Kuramoto model and the swarm sphere model in $\bbs^3$ . For this, we begin with the Lohe matrix model on $\mathbf{U}(2)$: 
\begin{equation} \label{D-1}
{\mathrm i} {\dot U}_i U_i^\dg = H_i + \frac{{\mathrm i} \kappa}{2N} \sum_{j=1}^{N} \Big(U_j U_i^\dg - U_i U_j^\dg    \Big), \quad i \in [N].
\end{equation}
In this case, we can use the parametrization of $2\times 2$ unitary matrices $U_i$ and $2\times 2$ Hermitian matrices $H_i$ via Pauli's matrices $\{ \sigma_k \}_{k=1}^{3}$:
\[ U_i := e^{-\mathrm{i} \theta_i} \Big( \mathrm{i} \sum_{k=1}^{3} x_i^k \sigma_k + x_i^4 I_2  \Big), \quad H_i: = \sum_{k=1}^{3} \omega_i^k \sigma_k + \nu_i I, \]
where $I$ and $\sigma_i$ are the  Pauli matrices and identity matrix defined by
\[
 I := \left( \begin{array}{cc}
  1 & 0 \\
  0 & 1 \\
  \end{array} \right), \quad 
  \sigma_1 := \left( \begin{array}{cc}
  0 & 1 \\
  1 & 0 \\
  \end{array} \right), \quad \sigma_2 :=  \left( \begin{array}{cc}
  0 & -i \\
  i & 0 \\
  \end{array} \right), \quad \sigma_3 := \left( \begin{array}{cc}
  1 & 0 \\
  0 & -1 \\
  \end{array} \right). \]
If we denote
\[
x_i = (x_i^1,x_i^2,x_i^3,x_i^4),\quad \omega_i = (\omega_i^1, \omega_i^2, \omega_i^3),
\]
then after straightforward   manipulations, one can find the dynamics for $(\theta_i,x_i)$ so that the Lohe matrix model \eqref{D-1} reduces to $5N$ equations for $\theta_i$ and $x_i$:
\begin{equation}
\begin{cases} \label{D-2}
\displaystyle {\dot \theta}_i = \nu_i + \frac{\kappa}{N} \sum_{k=1}^{N}   \langle x_i,x_k \rangle \sin (\theta_k - \theta_i),  \\
\displaystyle {\dot x}_i = \Omega_i x_i + \frac{\kappa}{N} \sum_{k=1}^{N} \cos(\theta_k - \theta_i) (x_k -\langle x_i,x_k \rangle x_i), \quad i\in [N],
\end{cases}
\end{equation}
where $\Omega_i$ is a real $4 \times 4$ skew-symmetric matrix given by
\[ \Omega_i := \left( \begin{array}{cccc}
0 & -\omega_i^3 & \omega_i^2 & -\omega_i^1 \\
\omega_i^3 & 0 & -\omega_i^1 & -\omega_i^2 \\
-\omega_i^2 & \omega_i^1 & 0 & -\omega_i^3 \\
\omega_i^1 & \omega_i^2 &\omega_i^3 & 0 \\
\end{array}
\right),\quad i\in [N]. \]
For the emergent dynamics of \eqref{D-2}, we refer the reader to \cite{C-H2}. 

\subsection{The double tensor model} \label{sec:4.2} 
In this subsection, we present a weak coupling of two LT models for possibly different rank tensors. As discussed in previous section, we can associate a unique characteristic symbol for each  Cauchy problem to the   LT model.  To fix the idea, we consider two characteristic symbols:  
\begin{equation*} \label{D-3}
\mathfrak{C}_1=(\bfd^1, \mathfrak{K}^1,\{A^1_j\},\{(T_j^1)^0\})\quad \textup{and} \quad \mathfrak{C}_2=(\bfd^2, \mathfrak{K}^2,\{A^2_j\},\{(T_j^2)^0\}).
\end{equation*}
Now, let $\{T_j^1\}$ and $\{T_j^2\}$ be two solutions to the Cauchy problems of the LT models corresponding to $\mathfrak{C}_1$ and  $\mathfrak{C}_2$, respectively: 
\begin{align*}\label{D-4}
\begin{cases}
\displaystyle[\dot{T}_j^1]_{\beta_{*0}}=[A_j^1]_{\beta_{*0}\beta_{*1}}[T_j^1]_{\beta_{*1}} +\sum_{j_*\in\{0, 1\}^{m_1}} \frac{\kappa_{j_*}^1}{N}   \\
\displaystyle \hspace{1.7cm}  \times\sum_{k=1}^N \left([T_k^1]_{\beta_{*j_*}} [\overline{{T}_j^1}]_{\beta_{*1}}[T_j^1]_{\beta_{*(1-j_*)}}-
[T_j^1]_{\beta_{*j_*}} [\overline{{T}_k^1}]_{\beta_{*1}}[T_j^1]_{\beta_{*(1-j_*)}}\right),\quad t >0,\\
T_j^1(0)=(T_j^1)^0, \quad \|(T_j^1)^0\|_\tF = 1, \quad j \in [N],
\end{cases}
\end{align*}
and
\begin{equation}\label{D-5}
\begin{cases}
\displaystyle[\dot{T}_j^2]_{\gamma_{*0}}=[A_j^2]_{\gamma_{*0}\gamma_{*1}}[T_j^1]_{\gamma_{*1}} +\sum_{k_*\in\{0, 1\}^{m_2}}\frac{\kappa_{k_*}^2}{N}   \\
\displaystyle \hspace{1.6cm} \times \sum_{k=1}^N\left([T_k^2]_{\gamma_{*k_*}} [\overline{{T}_j^2}]_{\gamma_{*1}}[T_j^2]_{\gamma_{*(1-k_*)}}-
[T_j^2]_{\gamma_{*k_*}} [\overline{{T}_k^2}]_{\gamma_{*1}}[T_j^2]_{\gamma_{*(1-k_*)}}\right),\quad t > 0, \\
T_j^2(0)=(T_j^2)^0, \quad \|(T_j^2)^0\|_\tF = 1, \quad j \in [N].
\end{cases}
\end{equation}
Motivated by weakly coupled models \eqref{B-1} and \eqref{B-3-1} in Section \ref{sec:2}, we propose the double tensor model as follows: for $t>0$, 
\begin{align}\label{D-6}
\begin{cases}
\displaystyle[\dot{T}_j^1]_{\beta_{*0}}=[A_j^1]_{\beta_{*0}\beta_{*1}}[T_j^1]_{\beta_{*1}}+\sum_{j_*\in\{0, 1\}^{m_1}}\frac{\kappa_{j_*}^1}{N} \\
\displaystyle \hspace{0.2cm} \times \sum_{k=1}^N  \Big(\langle T_j^2, T_k^2\rangle_\tF[T_k^1]_{\beta_{*j_*}} [\overline{{T}_j^1}]_{\beta_{*1}}[T_j^1]_{\beta_{*(1-j_*)}} -\langle T_k^2, T_j^2\rangle_\tF[T_j^1]_{\beta_{*j_*}} [\overline{{T}_k^1}]_{\beta_{*1}}[T_j^1]_{\beta_{*(1-j_*)}}\Big),\\
\displaystyle[\dot{T}_j^2]_{\gamma_{*0}}=[A_j^2]_{\gamma_{*0}\gamma_{*1}}[T_j^1]_{\gamma_{*1}}
+\sum_{k_*\in\{0, 1\}^{m_2}} \frac{\kappa_{k_*}^2}{N}  \\
\displaystyle \hspace{0.2cm} \times \sum_{k=1}^N \Big(\langle T_j^1, T_k^1\rangle_\tF[T_k^2]_{\gamma_{*k_*}} [\overline{{T}_j^2}]_{\gamma_{*1}}[T_j^2]_{\gamma_{*(1-k_*)}} -\langle T_k^1, T_j^1\rangle_\tF[T_j^2]_{\gamma_{*k_*}} [\overline{{T}_k^2}]_{\gamma_{*1}}[T_j^2]_{\gamma_{*(1-k_*)}}\Big), \\
(T_j^1, T_j^2)(0)=((T_j^1)^0, (T_j^2)^0),\quad  \|(T_j^1)^0\|_\tF = 1, \quad \|(T_j^2)^0\|_\tF = 1,  \quad j \in [N]. 
\end{cases}
\end{align}
Next, we consider a larger tensor $T_j$ given by the tensor product of $T_j^1$ and $T_j^2$:
\[
T_j:=T_j^1\otimes T_j^2\in\bbc^{ d_1^1\times \cdots \times d_{m_1}^1 \times d_1^2 \times \cdots \times d_{m_2}^2  }.
\]
In what follows, we find a   Cauchy problem for  $T_j$.  For this, note that 
\begin{equation}  \label{D-7} 
\dot{T}_j=\dot{T}_j^1\otimes T_j^2+T_j^1\otimes \dot{T}_j^2\quad\Longleftrightarrow\quad [\dot{T}_j]_{\beta_{*0}\gamma_{*0}}=[\dot{T}_j^1]_{\beta_{*0}} [T_j^2]_{\gamma_{*0}}+[T_j^1]_{\beta_{*0}}[ \dot{T}_j^2]_{\gamma_{*0}}.
\end{equation}
We substitute \eqref{D-6} for $T_j^1$ and $T_j^2$ into the right-hand side of \eqref{D-7} to find 
\begin{align}
\begin{aligned}\label{D-8}
&[\dot{T}_j]_{\beta_{*0}\gamma_{*0}}
=[\dot{T}_j^1]_{\beta_{*0}} [T_j^2]_{\gamma_{*0}}+[T_j^1]_{\beta_{*0}}[ \dot{T}_j^2]_{\gamma_{*0}}\\
&\hspace{0.5cm}=[A_j^1]_{\beta_{*0}\beta_{*1}}[T_j^1]_{\beta_{*1}}[T_j^2]_{\gamma_{*0}} +\sum_{j_*\in\{0, 1\}^{m_1}}\frac{\kappa_{j_*}^1}{N} \\
&\hspace{0.7cm}  \times\sum_{k=1}^N \Big(\langle T_j^2, T_k^2\rangle_\tF [T_k^1]_{\beta_{*j_*}} \overline{[{T}_j^1]}_{\beta_{*1}}[T_j^1]_{\beta_{*(1-j_*)}} -\langle T_k^2, T_j^2\rangle_\tF [T_j^1]_{\beta_{*j_*}}\overline{[{T}_k^1]}_{\beta_{*1}}[T_j^1]_{\beta_{*(1-j_*)}}\Big)[T_j^2]_{\gamma_{*0}}\\
&\hspace{0.7cm}+[T_j^1]_{\beta_{*0}}[A_j^2]_{\gamma_{*0}\gamma_{*1}}[T_j^1]_{\gamma_{*1}} +\sum_{k_*\in\{0, 1\}^{m_2}}\frac{\kappa_{k_*}^2}{N}    \\
&\hspace{0.7cm} \times \sum_{k=1}^N[T_j^1]_{\beta_{*0}} \Big(\langle T_j^1, T_k^1\rangle_\tF [T_k^2]_{\gamma_{*k_*}} \overline{[{T}_j^2]}_{\gamma_{*1}}[T_j^2]_{\gamma_{*(1-k_*)}} -\langle T_k^1, T_j^1\rangle_\tF [T_j^2]_{\gamma_{*k_*}} \overline{[{T}_k^2]}_{\gamma_{*1}}[T_j^2]_{\gamma_{*(1-k_*)}}\Big)\\
&\hspace{0.5cm}=:\mathcal{I}_{11}+\mathcal{I}_{12}+\mathcal{I}_{13}+\mathcal{I}_{14}.
\end{aligned}
\end{align}
Below, our goal is to  simplify the terms $\mathcal I_{11}+\mathcal I_{13}$ for free flows, and $\mathcal I_2+\mathcal I_4$ for interaction terms, respectively. \newline

\subsubsection{Simplification of free flows} \label{sec:4.2.1} 
For natural frequency tensors $A_j^1$ and $A_j^2$ with rank-$2m_1$ and rank-$2m_2$, respectively, we define  a rank-$2(m_1+m_2)$ tensor $A_j$ as  
\begin{align*} \label{D-9}
[A_j]_{\beta_{*0}\gamma_{*0}\beta_{*1}\gamma_{*1}}:=[A_j^1]_{\beta_{*0}\beta_{*1}}\delta_{\gamma_{*0}\gamma_{*1}}+\delta_{\beta_{*0}\beta_{*1}}[A_j^2]_{\gamma_{*0}\gamma_{*1}},\quad j\in [N].
\end{align*}
Note that one can easily verify block  skew-hermitian property of $A_j$:
\[
[A_j]_{\beta_{*0}\gamma_{*0}\beta_{*1}\gamma_{*1}}=-[\overline{A}_j]_{\beta_{*1}\gamma_{*1}\beta_{*0}\gamma_{*0}}.
\]
Then, we can represent $\mathcal{I}_{11}+\mathcal{I}_{13}$ in terms of $A_j$ and $T_j$:
\begin{align} \label{D-10}
\begin{aligned}
\mathcal{I}_{11}+\mathcal{I}_{13}&=[A_j^1]_{\beta_{*0}\beta_{*1}}[T_j^1]_{\beta_{*1}}[T_j^2]_{\gamma_{*0}}+[T_j^1]_{\beta_{*0}}[A_j^2]_{\gamma_{*0}\gamma_{*1}}[T_j^1]_{\gamma_{*1}}\\
&=\left([A_j^1]_{\beta_{*0}\beta_{*1}}\delta_{\gamma_{*0}\gamma_{*1}}+\delta_{\beta_{*0}\beta_{*1}}[A_j^2]_{\gamma_{*0}\gamma_{*1}}\right)[T_j^1]_{\beta_{*1}}[T_j^2]_{\gamma_{*1}}\\
&=[A_j]_{\beta_{*0}\gamma_{*0}\beta_{*1}\gamma_{*1}}[T_j^1\otimes T_j^2]_{\beta_{*1}\gamma_{*1}}\\
&=[A_j]_{\beta_{*0}\gamma_{*0}\beta_{*1}\gamma_{*1}}[T_j]_{\beta_{*1}\gamma_{*1}}.
\end{aligned}
\end{align}
In other words, we define a new natural frequency tensor $A_j$ from $A_j^1$ and $A_j^2$ so that \eqref{D-8} becomes a free flow with $A_j$ when all coupling strengths are absent. This is the way we will introduce a binary operation denoted by $\star_\tF $ between $A_j^1$ and $A_j^2$ (see Definition \ref{D4.1} below):
\[
A_j := A_j^1 \star_\tF  A_j^2. 
\]


\subsubsection{Simplification of interaction terms} \label{sec:4.2.2} In this part, we simplify the interaction terms $\mathcal I_{12}$  and $\mathcal I_{14}$ appearing in \eqref{D-8} as follows. First, we set 
\begin{align*}
\begin{aligned}
& \mathfrak{T}_1:=  \left(\langle T_j^2, T_k^2\rangle_\tF [T_k^1]_{\beta_{*j_*}} \overline{[{T}_j^1]}_{\beta_{*1}}[T_j^1]_{\beta_{*(1-j_*)}}-
\langle T_k^2, T_j^2\rangle_\tF [T_j^1]_{\beta_{*j_*}}\overline{[{T}_k^1]}_{\beta_{*1}}[T_j^1]_{\beta_{*(1-j_*)}}\right)[T_j^2]_{\gamma_{*0}},   \\
&   \mathfrak{T}_2:=  [T_j^1]_{\beta_{*0}}\left(\langle T_j^1, T_k^1\rangle_\tF [T_k^2]_{\gamma_{*k_*}}\overline{[{T}_j^2]}_{\gamma_{*1}}[T_j^2]_{\gamma_{*(1-k_*)}}-
\langle T_k^1, T_j^1\rangle_\tF [T_j^2]_{\gamma_{*k_*}}\overline{[{T}_k^2]}_{\gamma_{*1}}[T_j^2]_{\gamma_{*(1-k_*)}}\right ).  
\end{aligned}
\end{align*}
Then, we have 
\[
\mathcal I_{12} = \sum_{j\in \{0,1\}^{m_1}} \frac{\kp_{j_*}^1}{N} \sum_{k=1}^N \mathfrak{T}_1, \quad \mathcal I_{12} = \sum_{j\in \{0,1\}^{m_2}} \frac{\kp_{j_*}^2}{N} \sum_{k=1}^N \mathfrak{T}_2.
\]
By straightforward calculation, we can express $\mathfrak{T}_1$ in terms of $T_j$:
\begin{align}
\begin{aligned}\label{D-12}
\mathfrak{T}_1&=[T_k^1]_{\beta_{*j_*}}\overline{[{T}_j^1]}_{\beta_{*1}}[T_j^1]_{\beta_{*(1-j_*)}}[T_k^2]_{\gamma_{*1}}\overline{[{T}_j^2]}_{\gamma_{*1}}[T_j^2]_{\gamma_{*0}}\\
&\hspace{3cm}-
[T_j^1]_{\beta_{*j_*}}\overline{[{T}_k^1]}_{\beta_{*1}}[T_j^1]_{\beta_{*(1-j_*)}}[T_j^2]_{\gamma_{*1}}\overline{[{T}_k^2]}_{\gamma_{*1}}[T_j^2]_{\gamma_{*0}}\\
&=[T_k^1\otimes T_k^2]_{\beta_{*j_*}\gamma_{*1}}[\overline{T_j^1\otimes T_j^2}]_{\beta_{*1}\gamma_{*1}}[T_j^1\otimes T_j^2]_{\beta_{*(1-j_*)}\gamma_{*0}}\\
&\hspace{3cm}-
[T_j^1\otimes T_j^2]_{\beta_{*j_*}\gamma_{*1}}[\overline{T_k^1\otimes T_k^2}]_{\beta_{*1}\gamma_{*1}}[T_j^1\otimes T_j^2]_{\beta_{*(1-j_*)}\gamma_{*0}}\\
&=[T_k]_{\beta_{*j_*}\gamma_{*1}}\overline{[{T}_j]}_{\beta_{*1}\gamma_{*1}}[T_j]_{\beta_{*(1-j_*)}\gamma_{*0}}-
[T_j]_{\beta_{*j_*}\gamma_{*1}}\overline{[{T}_k]}_{\beta_{*1}\gamma_{*1}}[T_j]_{\beta_{*(1-j_*)}\gamma_{*0}},
\end{aligned}
\end{align}
where we used the Einstein convention for the repeated indices. Similarly, one has 
\begin{align}
\begin{aligned}\label{D-13}
 \mathfrak{T}_2 =[T_k]_{\beta_{*1}\gamma_{*k_*}}\overline{[{T}_j]}_{\beta_{*1}\gamma_{*1}}[T_j]_{\beta_{*0}\gamma_{\gamma_{*(1-k_*)}}}-[T_j]_{\beta_{*1}\gamma_{*k_*}}\overline{[{T}_k]}_{\beta_{*1}\gamma_{*1}}[T_j]_{\beta_{*0}\gamma_{*(1-k_*)}}.
\end{aligned}
\end{align}
To avoid cluttered mathematical expression, we introduce simplified notation: for $n \in \bbn$, 
\begin{equation}\label{D-14}
\mathbf{1}_n:=\underbrace{(1, 1, \cdots, 1)}_{n-\text{times}}, \quad \alpha_{*(j_*, k_*)}:=(\beta_{*j_*}, \gamma_{*k_*}), \quad j_*\in\{0, 1\}^{m_1} ~~\textup{and}~~  k_*\in\{0, 1\}^{m_2}. 
\end{equation}
In \eqref{D-8}, we use \eqref{D-12}, \eqref{D-13} and \eqref{D-14} to represent $\mathcal I_{12}$ and $\mathcal I_{14}$ only in terms of $T_j$:
\begin{align*}
\mathcal{I}_{12}&=\sum_{j_*\in\{0, 1\}^{m_1}}\frac{\kappa_{j_*}^1}{N}\sum_{k=1}^N \Big(\langle T_j^2, T_k^2\rangle_\tF [T_k^1]_{\beta_{*j_*}}\overline{[{T}_j^1]}_{\beta_{*1}}[T_j^1]_{\beta_{*(1-j_*)}}\\
&\hspace{5cm}-\langle T_k^2, T_j^2\rangle_\tF [T_j^1]_{\beta_{*j_*}}\overline{[{T}_k^1]}_{\beta_{*1}}[T_j^1]_{\beta_{*(1-j_*)}}\Big)[T_j^2]_{\gamma_{*0}}\\
&=\sum_{j_*\in\{0, 1\}^{m_1}}\frac{\kappa_{j_*}^1}{N}\sum_{k=1}^N \Big([T_k]_{\beta_{*j_*}\gamma_{*1}}\overline{[{T}_j]}_{\beta_{*1}\gamma_{*1}}[T_j]_{\beta_{*(1-j_*)}\gamma_{*0}}\\
&\hspace{5cm}-[T_j]_{\beta_{*j_*}\gamma_{*1}}\overline{[{T}_k]}_{\beta_{*1}\gamma_{*1}}[T_j]_{\beta_{*(1-j_*)}\gamma_{*0}}\Big)\\
&=\sum_{\substack{i_*=(j_*,\mathbf{1}_{m_2})\\j_*\in\{0, 1\}^{m_1}}}
\frac{\kappa_{j_*}^1}{N}\sum_{k=1}^N\left([T_k]_{\alpha_{*i_*}}\overline{[{T}_j]}_{\alpha_{*1}}[T_j]_{\alpha_{*(1-i_*)}}-
[T_j]_{\alpha_{*i_*}}\overline{[{T}_k]}_{\alpha_{*1}}[T_j]_{\alpha_{*(1-i_*)}}\right),
\end{align*}
and
\begin{align*}
\mathcal{I}_{14}&=\sum_{k_*\in\{0, 1\}^{m_2}}\frac{\kappa_{k_*}^2}{N}\sum_{k=1}^N[T_j^1]_{\beta_{*0}}\Big(\langle T_j^1, T_k^1\rangle_\tF [T_k^2]_{\gamma_{*k_*}}\overline{[{T}_j^2]}_{\gamma_{*1}}[T_j^2]_{\gamma_{*(1-k_*)}}\\
&\hspace{5cm}-\langle T_k^1, T_j^1\rangle_\tF [T_j^2]_{\gamma_{*k_*}}\overline{[{T}_k^2]}_{\gamma_{*1}}[T_j^2]_{\gamma_{*(1-k_*)}}\Big)\\
&=\sum_{k_*\in\{0, 1\}^{m_2}}\frac{\kappa_{k_*}^2}{N}\sum_{k=1}^N\Big([T_k]_{\beta_{*1}\gamma_{*k_*}}\overline{[T_j]}_{\beta_{*1}\gamma_{*1}}[T_j]_{\beta_{*0}\gamma_{\gamma_{*(1-k_*)}}}\\
&\hspace{5cm}-[T_j]_{\beta_{*1}\gamma_{*k_*}}\overline{[{T}_k]}_{\beta_{*1}\gamma_{*1}}[T_j]_{\beta_{*0}\gamma_{*(1-k_*)}}\Big)\\
&=\sum_{\substack{i_*=(\mathbf{1}_{m_1}, k_*)\\k_*\in\{0, 1\}^{m_2}}}\frac{\kappa_{k_*}^2}{N}\sum_{k=1}^N \left([T_k]_{\alpha_{*i_*}}\overline{[{T}_j]}_{\alpha_{*1}}[T_j]_{\alpha_{*(1-i_*)}}-[T_j]_{\alpha_{*i_*}}\overline{[{T}_k]}_{\alpha_{*1}}[T_j]_{\alpha_{*(1-i_*)}}\right).
\end{align*}
Note that $\mathcal I_{12}$ and $\mathcal I_{14}$ are represented in terms of $\{T_j\}$. Now, we consider $\mathcal I_{12} + \mathcal I_{14}$: 
\begin{align}\label{D-16-1}
\begin{aligned} 
&\mathcal{I}_{12}+\mathcal{I}_{14} \\
&\hspace{0.5cm} =\sum_{\substack{i_*=(j_*,\mathbf{1}_{m_2})\\j_*\in\{0, 1\}^{m_1}}}
\frac{\kappa_{j_*}^1}{N}\sum_{k=1}^N\left([T_k]_{\alpha_{*i_*}}\overline{[{T}_j]}_{\alpha_{*1}}[T_j]_{\alpha_{*(1-i_*)}}-
[T_j]_{\alpha_{*i_*}}[\overline{T}_k]_{\alpha_{*1}}[T_j]_{\alpha_{*(1-i_*)}}\right)\\
&\hspace{0.7cm} +\sum_{\substack{i_*=(\mathbf{1}_{m_1}, k_*)\\k_*\in\{0, 1\}^{m_2}}}\frac{\kappa_{k_*}^2}{N}\sum_{k=1}^N \left([T_k]_{\alpha_{*i_*}}\overline{[{T}_j]}_{\alpha_{*1}}[T_j]_{\alpha_{*(1-i_*)}}-[T_j]_{\alpha_{*i_*}}\overline{[{T}_k]}_{\alpha_{*1}}[T_j]_{\alpha_{*(1-i_*)}}\right).
\end{aligned}
\end{align}
Since we want the desired model for $T_j$ to be the LT model, the interaction term \eqref{D-16-1} should be further reduced to a simpler form. For this, we will introduce a binary operation, denoted by $\star_c$ (see Definition \ref{D4.1} below) between two coupling strength $\kp_{j_*}^1$ and $\kp_{k_*}^2$ in \eqref{D-16-1} so that $\mathcal I_{12} + \mathcal I_{14}$ becomes
\begin{align} \label{D-17}
\begin{aligned}
&\mathcal I_{12} + \mathcal I_{14} \\
&=\sum_{i_*\in\{0, 1\}^{m_1+m_2}}\frac{(\kappa^1\star\kappa^2)_{i_*}}{N}\sum_{k=1}^N\left([T_k]_{\alpha_{*i_*}}\overline{[{T}_j]}_{\alpha_{*1}}[T_j]_{\alpha_{*(1-i_*)}}-[T_j]_{\alpha_{*i_*}}\overline{[{T}_k]}_{\alpha_{*1}}[T_j]_{\alpha_{*(1-i_*)}}\right).
\end{aligned}
\end{align}
Finally in \eqref{D-8}, we combine \eqref{D-10} for free flow terms and \eqref{D-17} for interaction terms to find the Cauchy problem for $T_j=T_j^1\otimes T_j^2$: 
\begin{equation}\label{D-18}
\begin{cases}
\displaystyle[\dot{T}_j]_{\alpha_{*0}}=[A_j^1\star_\tF  A_j^2]_{\alpha_{*0}\alpha_{*1}}[T_j]_{\alpha_{*1}} +\sum_{i_*\in\{0, 1\}^{m_1+m_2}}\sum_{k=1}^N\frac{(\mathfrak{K}^1\star_c \mathfrak{K}^2)_{i_*}}{N} \\\displaystyle \hspace{1cm} \times \Big([T_k]_{\alpha_{*i_*}}[\overline{T}_j]_{\alpha_{*1}}[T_j]_{\alpha_{*(1-i_*)}} -[T_j]_{\alpha_{*i_*}}[\overline{T}_k]_{\alpha_{*1}}[T_j]_{\alpha_{*(1-i_*)}}\Big),\\
\displaystyle T_j(0)=T_j^0=(T_j^1)^0\otimes (T_j^2)^0.
\end{cases}
\end{equation}

\vspace{0.2cm}

So far, we have derived the Cauchy problem for $T_j=T_j^1\otimes T_j^2$ from the weakly coupled model \eqref{D-6} for $T_j^1$ and $T_j^2$. Motivated by the previous arguments,   we provide the definition of fusion operations $\star_s$, $\star_\tF $, $\star_c$ and $\star_i$.

\begin{definition} \label{D4.1} 
\textup{(i) (Fusion of size vectors)}: Let $\mathbf{d}^1=(d^1_{1},\cdots,d^1_{m_1}) \in \bbn^{m_1}$ and $\mathbf{d}^2=(d^2_{1},\cdots,d^2_{m_2})\in \bbn^{m_2}$ be two size vectors with
\[ m_1: =\mathrm{dim}(\mathbf{d}^1) \quad \textup{and} \quad m_2:=\mathrm{dim}(\mathbf{d}^2). \]
Then, we define the fusion operation $\star_s$ of two size vectors:
\begin{equation*}
\begin{cases}
\vspace{0.1cm} \displaystyle \star_s:\bbn^{m_1} \times \bbn^{m_2} \to \bbn^{m_1+ m_2},\\
 \displaystyle (\mathbf{d}^1,\mathbf{d}^2)\mapsto\mathbf{d}^1\star_s \mathbf{d}^2 := (d^1_{1},\cdots,d^1_{m_1},d^2_{1},\cdots,d^2_{m_2}) \in \bbn^{m_1+m_2}.
\end{cases}
\end{equation*}
\vspace{0.2cm}

\noindent \textup{(ii) (Fusion of natural frequency tensors)}: Let $A^1\in\mathfrak{F}(\mathbf{d}_1)$ and $A^2\in\mathfrak{F}(\mathbf{d}_2)$ be two natural frequency tensors in $\mathfrak{F}$.  Then, we define the fusion operation $\star_\tF $ between two natural frequency tensors as follows:
\begin{equation} \label{D-18-0}
\begin{cases}
\vspace{0.2cm}\displaystyle \star_\tF : \mathfrak{F} \times\mathfrak{F} \to\mathfrak{F},\quad (A^1, A^2)\mapsto A^1\star_\tF  A^2, \\
\displaystyle  [A^1\star_\tF  A^2]_{\beta_{*0}\gamma_{*0}\beta_{*1}\gamma_{*1}} :=[A^1]_{\beta_{*0}\beta_{*1}}\delta_{\gamma_{*0}\gamma_{*1}}+\delta_{\beta_{*0}\beta_{*1}}[A^2]_{\gamma_{*0}\gamma_{*1}},
\end{cases}
\end{equation}
for $\beta_{*0},~\beta_{*1}\in\mathfrak{I}(\mathbf{d}_1),~\gamma_{*0},~\gamma_{*1}\in\mathfrak{I}(\mathbf{d}_2).$

\vspace{0.2cm}

\noindent \textup{(iii) }(Fusion of coupling strength tensors): 
 Let $\mathfrak{K}^1=\{\kappa_{j_*}^1\}_{j_*\in\{0, 1\}^{m_1}}$ and $\mathfrak{K}^2=\{\kappa_{k_*}^2\}_{k_*\in\{0, 1\}^{m_2}}$ be coupling strength tensors. Then, we define the fusion operation $\star_c$ between two coupling strength tensors as follows: 
\[
\star_c: \bbr^{2^{m_1}}\times \bbr^{2^{m_2}}\to \bbr^{2^{m_1+ m_2}},\quad (\mathfrak{K}^1,\mathfrak{K}^2)\mapsto \mathfrak{K}^1\star_c \mathfrak{K}^2,
\]
where
\begin{align} \label{D-18-1}
\begin{aligned}
(\mathfrak{K}^1\star_c \mathfrak{K}^2)_{(j_*, k_*)} :=
\begin{cases}
\kappa_{j_*}^1\quad& \text{if }j_*\neq \mathbf{1}_{m_1}\text{~and~}k_*= \mathbf{1}_{m_2},\\
\kappa_{k_*}^2\quad& \text{if }j_*= \mathbf{1}_{m_1}\text{~and~}k_*\neq \mathbf{1}_{m_2},\\
\kappa_{\mathbf{1}_{m_1}}^1+\kappa_{\mathbf{1}_{m_2}}^2\quad&\text{if }j_*= \mathbf{1}_{m_1}\text{~and~}k_*= \mathbf{1}_{m_2},\\
0\quad&\text{if }j_*\neq \mathbf{1}_{m_1}\text{~and~}k_*\neq \mathbf{1}_{m_2},\\
\end{cases}
\end{aligned}
\end{align}
for all $j_*\in\{0, 1\}^{m_1}$ and $k_*\in\{0, 1\}^{m_2}$. Moreover, $\mathfrak{K}^1\star_c \mathfrak{K}^2$ can be simply rewritten in a compact form using  Kronecker-delta symbols:
\begin{align*}\label{D-18-2}
(\mathfrak{K}^1\star_c \mathfrak{K}^2)_{(j_*, k_*)}=\kappa^1_{j_*}\delta_{k_*,\mathbf{1}_{m_2}}+\kappa^2_{j_*}\delta_{j_*,\mathbf{1}_{m_1}},
\end{align*}
where $\delta_{(i_1, i_2, \cdots, i_m), (j_1, j_2, \cdots, j_m)}=\delta_{i_1j_1}\delta_{i_2j_2}\cdots\delta_{i_mj_m}$. 

\vspace{0.2cm}

\noindent \textup{(iv) } (Fusion of initial configurations): Let $(T_j^1)^0$ and $(T_j^2)^0$ be two initial configurations. Then, the fusion operation $\star_i$ between initial configurations is defined as the tensor product:
\begin{equation*}
\star_i : \bbc^{\mathbf{d}^1} \times \bbc^{\mathbf{d}^2} \to \bbc^{\mathbf{d}^1 \star_s \mathbf{d}^2}, \quad ( (T_j^1)^0,(T_j^2)^0) \mapsto (T_j^1)^0  \star_i (T_j^2)^0 = (T_j^1)^0\otimes (T_j^2)^0.
\end{equation*}
\end{definition}
\begin{remark}
Below, we give several comments on the fusion operations in Definition \ref{D4.1} for each component:
\begin{enumerate}
\item (Fusion of size vectors): The fusion operation $\star_s$ is merely a juxtaposition of two given size vectors. 
\item (Fusion of natural frequency tensors): In \eqref{D-18-0}, $A^1 \star_\tF  A^2$ is a block diagonal form constructed from $A^1$ and $A^2$. 
\item (Fusion of coupling strength tensors): When we define $(\mathfrak{K}^1 \star_c \mathfrak{K}^2)_{(j_*,k_*)}$, it was crucial whether $j_*$ or $k_*$ is a vector whose components are all one. 
\item(Fusion of initial configuration): The fusion operation $\star_i$ is merely the tensor product of two given initial data. 
\end{enumerate}
\end{remark}

Next,  we are concerned with the characteristic symbol of the double tensor model \eqref{D-18}. By the definition of the fusion operations in Definition \ref{D4.1}, the characteristic symbol for \eqref{D-18} is given by 
\begin{equation*}\label{D-19}
\mathfrak{C}=(\bfd^1\star_s \bfd^2, \mathfrak{K}^1\star_c \mathfrak{K}^2,\{A_j^1\star_\tF  A_j^2\}, \{(T_j^1)^0\otimes(T_j^2)^0\}).
\end{equation*}
In this way, we can define a binary operation of two characteristic symbols $\mathfrak{C}^1$ and $\mathfrak{C}^2$. Recall from Definition \ref{D3.3} that $\mathfrak{D}$ denotes the set of all characteristic symbols. 
\begin{definition} \label{D4.2} (Fusion of characteristic symbols) 
Let $\mathfrak{C}^1=(\bfd^1, \mathfrak{K}^1,\{A^1_j\},\{(T_j^1)^0\}) \in \mathfrak{D}$ and $\mathfrak{C}^2=(\bfd^2, \mathfrak{K}^2,\{A^2_j\},\{(T_j^2)^0\})\in \mathfrak{D}$ be  characteristic symbols of  the Cauchy problems to the LT models, respectively. We define the fusion operation $\star$ between two characteristic symbols:
\begin{align*}
&\star: \mathfrak{D} \times \mathfrak{D} \to \mathfrak{D}, \\
& (\mathfrak{C}^1, \mathfrak{C}^2)\mapsto \mathfrak{C}^1\star \mathfrak{C}^2=(\bfd^1\star_s \bfd^2, \mathfrak{K}^1\star_c \mathfrak{K}^2,\{A_j^1\star_\tF  A_j^2\}, \{(T_j^1)^0\otimes(T_j^2)^0\}). 
\end{align*}
\end{definition}

As a direct consequence of the fusion operation in Definition \ref{D4.2}, we summarize the arguments above in the following theorem. 
\begin{theorem}\label{T4.1}
Let $\mathfrak{C}^1=(\bfd^1, \mathfrak{K}^1,\{A^1_j\},\{(T_j^1)^0\}) \in \mathfrak{D}$ and $\mathfrak{C}^2=(\bfd^2, \mathfrak{K}^2,\{A^2_j\},\{(T_j^2)^0\})\in \mathfrak{D}$ be characteristic symbols of the Cauchy problems to the LT models. If $\{T_j(t)\}_{j=1}^N$ is a solution to system \eqref{D-18} corresponding to the characteristic symbols $\mathfrak{C}_1\star \mathfrak{C}_2$, then $T_j(t)$ can be decomposed as
\[
T_j(t)=(T_j^1\otimes T_j^2)(t), \quad t > 0, \quad j\in [N]. 
\]
\end{theorem}

\subsection{Double tensor model without free flows}
In this subsection, we consider LT models from characteristic symbol $\mathfrak{C}=(\mathbf{d}, \mathfrak{K}, \{A_j\},\{T_j^0\})$ with $A_j\equiv 0$ for all $j\in[N]$. This kind of LT models are studied in \cite{H-P4}. From Remark 4.2 of \cite{H-P4}, if $\{T_i\}$ is a solution to $\mathfrak{C}=(\mathbf{d}, \mathfrak{K}, \{0\}, \{T_j^0\})$ 
we know that
\[
\lim_{t\to\infty}( [T_c]_{\alpha_{*i_*}}[\overline{T}_i]_{\alpha_{*1}}-[T_i]_{\alpha_{*i_*}}[\overline{T}_c]_{\alpha_{*1}})=0,
\]
for all uncontracted indices,  when $\kappa_{i_*}>0$. Now, we assume that $\mathfrak{C}=\mathfrak{C}^1\star \mathfrak{C}^2$, where $\mathfrak{C}^1=(\mathbf{d}^1, \mathfrak{K}^1,\{0\}, \{(T_j^1)^0\})$ and $\mathfrak{C}^2=(\mathbf{d}^2, \mathfrak{K}^2,\{0\}, \{(T_j^2)^0\})$. We also set $m_1$ and $m_2$ be the dimensions of $\mathbf{d}^1$ and $\mathbf{d}^2$. Then, we can assume that the solution $\{T_j\}$ can be decomposed as
\[
T_j=T_j^1\otimes T_j^2, \quad\forall~j\in[N].
\]
Also, we know that $\kappa_{i_*}>0$ if and only if either
\[
\kappa^1_{j_*}>0 \quad\text{and}\quad k_*=\mathbf{1}_{m_2}
\]
or
\[
\kappa^2_{k_*}>0\quad\text{and}\quad j_*=\mathbf{1}_{m_1},
\]
where $i_*=(j_*, k_*)\in\{0, 1\}^{m_1+m_2}$ with $j_*\in\{0, 1\}^{m_1}$ and $k_*\in\{0, 1\}^{m_2}$. If we let
\[
i_*=(j_*, k_*),\quad \alpha_*=(\beta_*, \gamma_*),
\]
then one finds
\begin{align*}
 &[T_c]_{\alpha_{*i_*}}[\overline{T}_i]_{\alpha_{*1}}-[T_i]_{\alpha_{*i_*}}[\overline{T}_c]_{\alpha_{*1}}\\
 &\hspace{1cm} =\frac{1}{N}\sum_{k=1}^N\left([T_k]_{\alpha_{*i_*}}[\overline{T}_i]_{\alpha_{*1}}-[T_i]_{\alpha_{*i_*}}[\overline{T}_k]_{\alpha_{*1}}\right)\\
 &\hspace{1cm}  =\frac{1}{N}\sum_{k=1}^N\left([T_k^1]_{\beta_{*j_*}}[T_k^2]_{\gamma_{*k_*}}[\bar{T}_i^1]_{\beta_{*1}}[\bar{T}_i^2]_{\gamma_{*1}}-[T_i^1]_{\beta_{*j_*}}[T_i^2]_{\gamma_{*k_*}}[\bar{T}_k^1]_{\beta_{*1}}[\bar{T}_k^2]_{\gamma_{*1}}\right).
\end{align*}
If we substitute $k_*=\mathbf{1}_{m_2}$, then we get
\[
[T_c]_{\alpha_{*i_*}}[\bar{T}_i]_{\alpha_{*1}}-[T_i]_{\alpha_{*i_*}}[\bar{T}_c]_{\alpha_{*1}}=\frac{1}{N}\sum_{k=1}^N\left(\langle T_i^2, T_k^2\rangle_\tF  [T_k^1]_{\beta_{*j_*}}[\bar{T}_i^1]_{\beta_{*1}}-\langle T_k^2, T_i^2\rangle_\tF [T_i^1]_{\beta_{*j_*}}[\bar{T}_k^1]_{\beta_{*1}}\right).
\]
Similarly, if we substitue $j_*=\mathbf{1}_{m_1}$, then we get
\[
[T_c]_{\alpha_{*i_*}}[\bar{T}_i]_{\alpha_{*1}}-[T_i]_{\alpha_{*i_*}}[\bar{T}_c]_{\alpha_{*1}}=
\frac{1}{N}\sum_{k=1}^N\left(\langle T_i^1, T_k^1\rangle_\tF  [T_k^2]_{\gamma_{*k_*}}[\bar{T}_i^2]_{\gamma_{*1}}-\langle T_k^1, T_i^1\rangle_\tF [T_i^2]_{\gamma_{*k_*}}[\bar{T}_k^2]_{\gamma_{*1}}\right).
\]
Finally, we have the following results:
\begin{align*}
\begin{cases}
\displaystyle\lim_{t\to\infty}\frac{1}{N}\sum_{k=1}^N\left(\langle T_i^2, T_k^2\rangle_\tF  [T_k^1]_{\beta_{*j_*}}[\bar{T}_i^1]_{\beta_{*1}}-\langle T_k^2, T_i^2\rangle_\tF [T_i^1]_{\beta_{*j_*}}[\bar{T}_k^1]_{\beta_{*1}}\right)=0\quad\text{if }\kappa_{j_*}^1>0,\\
\displaystyle\lim_{t\to\infty}\frac{1}{N}\sum_{k=1}^N\left(\langle T_i^1, T_k^1\rangle_\tF [T_k^2]_{\gamma_{*k_*}}[\bar{T}_i^2]_{\gamma_{*1}}-\langle T_k^1, T_i^1\rangle_\tF[T_i^2]_{\gamma_{*k_*}}[\bar{T}_k^2]_{\gamma_{*1}}\right)=0
\quad\text{if }\kappa_{k_*}^2>0.
\end{cases}
\end{align*}
We can compare these two terms and coupling terms in \eqref{D-6}. Coupling terms of $\dot{T}_j^1$ and $\dot{T}_j^2$ can be expressed as follows:
\[
\left(\frac{\kappa^1_{j_*}}{N}\sum_{k=1}^N\left(\langle T_i^2, T_k^2\rangle_\tF [T_k^1]_{\beta_{*j_*}}[\bar{T}_i^1]_{\beta_{*1}}-\langle T_k^2, T_i^2\rangle_\tF[T_i^1]_{\beta_{*j_*}}[\bar{T}_k^1]_{\beta_{*1}}\right)\right)[T_j^1]_{\beta_{*(1-j_*)}}
\]
and
\[
\left(\frac{\kappa^2_{k_*}}{N}\sum_{k=1}^N\left(\langle T_i^1, T_k^1\rangle_\tF [T_k^2]_{\gamma_{*k_*}}[\bar{T}_i^2]_{\gamma_{*1}}-\langle T_k^1, T_i^1\rangle_\tF[T_i^2]_{\gamma_{*k_*}}[\bar{T}_k^2]_{\gamma_{*1}}\right)\right)[T_j^2]_{\gamma_{*(1-k_*)}},
\]
respectively. From the previous results, we know that each coupling terms in \eqref{D-6} converge to zero as time goes to infinity. We can easily generalize this results to characteristic $\mathfrak{C}=\bs_{i=1}^n\mathfrak{C}^i$ with each natural frequency tensor satisfies 
\[ A^i_j\equiv 0 \quad \mbox{for all $i\in[n]$ and $j\in[N]$}. \]
Since each coupling term in \eqref{E-10} converges to zero, we have
\[
\lim_{t\to\infty}\sum_{k=1}^N \Bigg(\prod_{\substack{p=1\\p\neq \ell}}^n\langle T_j^p, T_k^p\rangle_\tF[T_k^\ell]_{\beta^i_{*j_*}}[\bar{T}_j^\ell]_{\beta^i_{*1}} -
\prod_{\substack{p=1\\p\neq \ell}}^n\langle T_k^p, T_j^p\rangle_\tF[T_j^\ell]_{\beta^\ell_{*j_*}}[\bar{T}_k^\ell]_{\beta^\ell_{*1}}\Bigg)=0,
\]
for all $\ell$ and uncontracted indices when $\kappa^\ell_{j_*}>0$.

\section{Weak coupling of multiple tensor models} \label{sec:5}
\setcounter{equation}{0}
In this section, we extend the weak coupling of two LT models to the arbitrary number of LT models using the fusion operation summarized in Definition \ref{D4.2}.  First, we study algebraic properties of the fusion operation on $\mathfrak{D}$, and it turns out to be a monoid on $\mathfrak{D}$; in other words, it is associative and $\mathfrak{D}$ has the identity element under the fusion operation. 

\subsection{Monoid structure of the fusion operation} \label{sec:5.1} 
In this subsection, we show that the fusion operation $\star$ introduced in Definition \ref{D4.2} gives a monoid structure on $\mathfrak{D}$, i.e., it has the identity and satisfyies associativity relation.  

\subsubsection{Associativity of $\star$}:  We set 
\[
\mathfrak{C}^i : =(\mathbf{d}^i, \mathfrak{K}^i, \{A_j^i\}, \{(T_j^i)^0\}),\quad m_i:=\mathrm{dim}(\mathbf{d}^i), \quad i=1, 2, 3.
\]
Then, note that 
\begin{align*}
\begin{aligned}
& (\mathfrak{C}^1 \star \mathfrak{C}^2)\star \mathfrak{C}^3 \\
& \hspace{0.5cm} =\big( (\mathbf{d}^1\star_s\mathbf{d}^2)\star_s
 \mathbf{d}^3,(\mathfrak{K}^1\star_c \mathfrak{K}^2)\star_c \mathfrak{K}^3,\{(A_j^1\star _\tF  A_j^2)\star_\tF  A_j^3\}, \{((T_j^1)^0\otimes(T_j^2)^0)\otimes(T_j^3)^0\}\big), \\
& \mathfrak{C}^1 \star (\mathfrak{C}^2\star \mathfrak{C}^3) \\
& \hspace{0.5cm} =\big( \mathbf{d}^1\star_s (\mathbf{d}^2\star_s \mathbf{d}^3), \mathfrak{K}^1\star_c (\mathfrak{K}^2\star_c \mathfrak{K}^3),\{A_j^1\star_\tF  ( A_j^2\star_\tF  A_j^3)\}, \{(T_j^1)^0\otimes ((T_j^2)^0\otimes (T_j^3)^0)\}\big).
\end{aligned}
\end{align*}
Thus for the associativity of $\star$, we need to show 
\begin{align} \label{E-0}
\begin{aligned}
&(\mathbf{d}^1\star_s \mathbf{d}^2)\star_s \mathbf{d}^3=\mathbf{d}^1\star_s (\mathbf{d}^2\star_s \mathbf{d}^3),\quad (\mathfrak{K}^1\star_c \mathfrak{K}^2)\star_c \mathfrak{K}^3=\mathfrak{K}^1\star_c (\mathfrak{K}^2\star_c \mathfrak{K}^3),\\
& (A_j^1\star_\tF  A_j^2)\star_\tF  A_j^3=A_j^1\star_\tF  ( A_j^2\star_\tF  A_j^3), \\
& \{((T_j^1)^0\otimes (T_j^2)^0)\otimes(T_j^3)^0\}=\{(T_j^1)^0\otimes((T_j^2)^0\otimes(T_j^3)^0)\}.
\end{aligned}
\end{align}
Since $\star_s$ is a merely juxtaposition of size vectors and $\otimes$ is just tensor product, one can easily verify the first and last relations:
\begin{align*}
\begin{aligned}
& (\mathbf{d}^1\star_s \mathbf{d}^2)\star_s \mathbf{d}^3=\mathbf{d}^1\star_s (\mathbf{d}^2\star_s \mathbf{d}^3),  \\
&   ((T_j^1)^0\otimes(T_j^2)^0)\otimes(T_j^3)^0   = (T_j^1)^0\otimes ((T_j^2)^0\otimes (T_j^3)^0), \quad j\in [N].
\end{aligned}
\end{align*}
In the following two lemmas, we verify the second and third relations in \eqref{E-0} one by one. 
\begin{lemma}\label{L5.1}
The fusion operation $\star_c$   is associative:
\begin{equation} \label{E-0-0}
(\mathfrak{K}^1\star_c \mathfrak{K}^2)\star_c \mathfrak{K}^3=\mathfrak{K}^1\star_c (\mathfrak{K}^2\star_c \mathfrak{K}^3), \quad \mathfrak{K}^i\in \bbr^{2^{m_i}},\quad i=1,2,3.
\end{equation}
\end{lemma}
\begin{proof}
It suffices to show that \eqref{E-0-0} holds componentwisely. We recall from Definition \ref{D4.1} that  for $(i_*, j_*, k_*)\in \{0, 1\}^{m_1}\times \{0, 1\}^{m_2}\times \{0, 1\}^{m_3}$, we use \eqref{D-18-1} to obtain 
\begin{align}
\begin{aligned} \label{E-1}
& ((\mathfrak{K}^1\star_c \mathfrak{K}^2)\star_c \mathfrak{K}^3)_{(i_*,j_*,k_*)} \\
& \hspace{0.5cm} =(\mathfrak{K}^1\star_c \mathfrak{K}^2)_{(i_*, j_*)}\delta_{k_*, \mathbf{1}_{m_3}}+ \mathfrak{K}^3_{k_*}\delta_{(i_*, j_*),\mathbf{1}_{m_1+m_2}}\\
& \hspace{0.5cm} =\mathfrak{K}^1_{i_*}\delta_{j_*, \mathbf{1}_{m_2}}\delta_{k_*, \mathbf{1}_{m_3}}+ \mathfrak{K}^2_{j_*}\delta_{i_*, \mathbf{1}_{m_1}}\delta_{j_*, \mathbf{1}_{m_2}}+ \mathfrak{K}^3_{k_*}\delta_{i_*, \mathbf{1}_{m_1}}\delta_{j_*, \mathbf{1}_{m_2}}.
\end{aligned}
\end{align}
Similarly, one finds 
\begin{align}
\begin{aligned} \label{E-2}
& (\mathfrak{K}^1\star_c (\mathfrak{K}^2\star_c \mathfrak{K}^3))_{(i_*, j_*, k_*)} \\
& \hspace{0.5cm} =\mathfrak{K}^1_{i_*}\delta_{j_*, \mathbf{1}_{m_2}}\delta_{k_*, \mathbf{1}_{m_3}}+ \mathfrak{K}^2_{j_*}\delta_{i_*, \mathbf{1}_{m_1}}\delta_{j_*, \mathbf{1}_{m_2}}+ \mathfrak{K}^3_{k_*}\delta_{i_*, \mathbf{1}_{m_1}}\delta_{j_*, \mathbf{1}_{m_2}}.
\end{aligned}
\end{align}
Thus, it follows from \eqref{E-1} and \eqref{E-2} that 
\[
(\mathfrak{K}^1\star_c \mathfrak{K}^2)\star_c \mathfrak{K}^3= \mathfrak{K}^1\star_c(\mathfrak{K}^2\star_c \mathfrak{K}^3).
\]
\end{proof}

\begin{lemma}\label{L5.2}
The fusion operation $\star_\tF $   is associative:
\begin{equation} \label{E-2-1}
(A^1\star_\tF  A^2)\star_\tF  A^3=A^1\star_\tF  ( A^2\star_\tF  A^3),\quad A^i\in\mathfrak{F}(\bfd^i),\quad i=1,2,3.
\end{equation}
\end{lemma}
\begin{proof}
Similar to Lemma \ref{L5.1}, it suffices to consider each component of both sides in \eqref{E-2-1}. It follows from \eqref{D-18-0} in Definition \ref{D4.1} that 
\begin{align}
\begin{aligned} \label{E-3}
& [(A^1\star A^2)\star A^3]_{\alpha_{*0}\beta_{*0}\gamma_{*0}\alpha_{*1}\beta_{*1}\gamma_{*1}} \\
& \hspace{0.5cm} =[A^1\star A^2]_{\alpha_{*0}\beta_{*0}\alpha_{*1}\beta_{*1}}[A^3]_{\gamma_{*0}\gamma_{*1}} =[A^1]_{\alpha_{*0}\alpha_{*1}}[A^2]_{\beta_{*0}\beta_{*1}}[A^3]_{\gamma_{*0}\gamma_{*1}}.
\end{aligned}
\end{align}
Similarly, we have 
\begin{equation} \label{E-4}
[A^1\star (A_j^2\star A^3)]_{\alpha_{*0}\beta_{*0}\gamma_{*0}\alpha_{*1}\beta_{*1}\gamma_{*1}}=[A^1]_{\alpha_{*0}\alpha_{*1}}[A^2]_{\beta_{*0}\beta_{*1}}[A^3]_{\gamma_{*0}\gamma_{*1}}.
\end{equation}
Finally, we combine \eqref{E-3} and \eqref{E-4} to derive the desired relation.
\end{proof}

Thanks to Lemma \ref{L5.1} and Lemma \ref{L5.2},  we show that $\star$ is associative on $\mathfrak{D}$.

\begin{proposition}\label{P5.1}
The fusion operation $\star$ is associative:
\[
(\mathfrak{C}^1 \star \mathfrak{C}^2) \star \mathfrak{C}^3=\mathfrak{C}^1\star(\mathfrak{C}^2\star \mathfrak{C}^3) ,\quad \mathfrak{C}^j \in \mathfrak{D}, \quad j=1,2,3.
\]
\end{proposition}
\begin{remark}
By the result of Proposition \ref{P5.1}, associativity relation allows us  to perform the following repeated operations for any $n\in \bbn$: 
\[
\mathfrak{C}^1\star \mathfrak{C}^2\star\cdots\star\mathfrak{C}^n.
\]
Thus, we  simply denote the repeated fusion operations as 
\[
\bs_{i=1}^n \mathfrak{C}^i:=\mathfrak{C}^1\star\mathfrak{C}^2\star\cdots\star \mathfrak{C}^n,
\]
and it is well-defined.
\end{remark}

\subsubsection{Identity element of $\star$}: Next, we look for the identity element for the fusion operation $\star$  denoted by $\mathfrak{C}_e\in \mathfrak{D}$ satisfying 
\begin{align}\label{E-5}
\mathfrak{C}_e \star \mathfrak{C}=  \mathfrak{C} \star \mathfrak{C}_e= \mathfrak{C} ,\quad\textup{for all}\quad \mathfrak{C} \in \mathfrak{D}.
\end{align}
To find the explicit representation for  an identity $\mathfrak{C}_e$, we set 
\[
\mathfrak{C}_e:=(\mathbf{d}^e, \mathfrak{K}^e, \{A_j^e\}, \{(T_j^e)^0\}),\quad \mathfrak{C} =(\mathbf{d}, \mathfrak{K}, \{A_j\}, \{ T_j^0\}),\quad m_e:= \mathrm{dim}(\mathbf{d}^e), \quad m= \mathrm{dim}(\mathbf{d}).
\]
Below, we present a possible candidate for  $\mathfrak{C}_e$ one by one. 

\vspace{0.5cm}

\noindent $\bullet$~Case A (Determination of $\mathbf{d}^e$):~By the definition of   $\star_s$, one has 
\[
\mathbf{d}^e\star_s \mathbf{d}=\mathbf{d}\star_s \mathbf{d}^e=\mathbf{d}.
\]
Thus, $\mathbf{d}^e$ should not contain any components; otherwise, $\mathbf{d}$ differs from $\mathbf{d}^e \star \mathbf{d}$. This shows that $\mathbf{d}^e$ is a zero-dimensional vector:
\[  \mathbf{d}^e=\emptyset. \]
Since the rank is defined as the number of components of $\mathbf{d}^e$ (see Section \ref{sec:3.1}), we have 
\[
m_e= \mathrm{dim}(\mathbf{d}^e)=0.
\]
In addition, a solution $T_j^e$ for the Cauchy problem corresponding to the characteristic symbol $\mathfrak{C}^e$ is rank-$m_e$ tensor, i.e., $T_j^e$ is a complex number.  

\vspace{0.2cm}

\noindent $\bullet$~Case B (Determination of $\mathfrak{K}^e$):  Since $m_e=0$,  we have  
\[
\mathfrak{K}^e:\{0, 1\}^0\to \bbr,
\]
and hence  $\mathfrak{K^e}$ can be considered as a real number. Our guess coincides with the fact that $\mathfrak{K}^e$ is a rank-$m_e$ tensor with $2^{m_e}$ elements. 

\vspace{0.2cm}

\noindent $\bullet$~Case C (Determination of $A_j^e$):~since $A_j^e$ is a rank-$2m_e$ tensor and $m_e=0$,  we can regard $A_j^e\in \bbc$ as a complex number with following skew-hermitian property:
\[
A_j^e=-\bar{A}_j^e.
\]
Thus, there exists a real number $\nu_j$ such that
\[
A_j^e=\mathrm{i}\nu_j.
\]
By Case A--Case C, one has a tentative ansatz for $\mathfrak{C}_e$:
\begin{equation} \label{E-7}
\mathfrak{C}_e=(\emptyset, \mathfrak{K}^e, \{\mathrm{i}\nu_j\}, \{(T_j^e)^0\}).
\end{equation}
\vspace{0.2cm}

In the following proposition, we find the explicit representation for $\mathfrak{C}_e$. 
\begin{proposition} \label{P5.3}
The identity $\mathfrak{C}_e$ is given as follows:
\[ \mathfrak{C}_e = (\emptyset, 0, \{0 \}, \{1\}). \]
\end{proposition}
\begin{proof} 
(i) We further identify $\mathfrak{K}^e, \{\mathrm{i}\nu_j\}$ and $(T_j^e)^0$  in \eqref{E-7} as follows. \newline

\noindent $\bullet$~(Representation of $\mathfrak{K}^e$):  since $m_e=0$, we have 
\[
(\mathfrak{K}^e\star_c \mathfrak{K})_{(\emptyset, i_*)}=(\mathfrak{K} \star_c \mathfrak{K}^e)_{(i_*,\emptyset)}=
\begin{cases}
\mathfrak{K}_{i_*}+ \mathfrak{K}^e\quad&\text{if }i_*= \mathbf{1}_m,\\
\mathfrak{K}_{i_*}\quad&\text{otherwise}.
\end{cases}
\]
Since the following relations should hold:
\[
\mathfrak{K}^e\star_c \mathfrak{K} =\mathfrak{K} \star_c \mathfrak{K}^e= \mathfrak{K}, \quad \textup{for all $\mathfrak{K}$}, 
\]
$\mathfrak{K}^e$ satisfies $\mathfrak{K}_{i_*} + \mathfrak{K}^e = \mathfrak{K}_{i_*}$
Thus, we have 
\[
 \mathfrak{K}^e=0.
\]

\noindent $\bullet$~(Representation of  $A_j^e$): We observe that $A_j^e$ should satisfy  
\begin{align*}
& [A_j^e\star_\tF  A_j]_{\emptyset\alpha_{*0}\emptyset\alpha_{*1}} =[A_j^e]_{\emptyset\emptyset}\delta_{\alpha_{*0}\alpha_{*1}}+\delta_{\emptyset\emptyset}[A_j]_{\alpha_{*0}\alpha_{*1}}=A_j^e \delta_{\alpha_{*0}\alpha_{*1}}+[A_j]_{\alpha_{*0}\alpha_{*1}},\\
&  [A_j\star_\tF  A_j^e]_{\alpha_{*0}\emptyset\alpha_{*1}\emptyset} =[A_j]_{\alpha_{*0}\alpha_{*1}}\delta_{\emptyset\emptyset}+\delta_{\alpha_{*0}\alpha_{*1}}[A_j^e]_{\emptyset\emptyset} = [A_j]_{\alpha_{*0}\alpha_{*1}}+A_j^e \delta_{\alpha_{*0}\alpha_{*1}}.
\end{align*}
Since the following relations below holds for all $A_j$:
\[
A_j^e\star_\tF  A_j=A_j\star_\tF  A_j^e=A_j,
\]
one particularly has  
\[
A_j^e \delta_{\alpha_{*0}\alpha_{*1}}+[A_j]_{\alpha_{*0}\alpha_{*1}} = [A_j]_{\alpha_{*0}\alpha_{*1}}.
\]
This yields 
\[
A_j^e=0.
\]

\noindent $\bullet$~(Representation of  $T_j^e$):~since $(T_j^e)^0$ is a complex number whose Frobenius norm equal to 1, the initial data $(T_j^e)^0$ can be considered as a complex number with unit modulus. We set 
\[
c_j:=(T_j^e)^0 \in \bbc, \quad |c_j|=1.
\]
Since $(T_j^e)^0$ should satisfy
\[
(T_j^e)^0\otimes T_j^0= T_j^0\otimes(T_j^e)^0=  T_j^0, \quad \textup{i.e.,}\quad c_j T_j^0 = T_j^0,
\]
we have
\[
c_j = (T_j^e)^0 = 1.
\]
In summary, we have shown that there exists a unique identity for the fusion operation $\star$ on $\mathfrak{D}$, and the unique (left and right) identity of the fusion operation   is given as follows:
\[
\mathfrak{C}_e=(\emptyset, 0, \{0\},\{1\}).
\]
\end{proof}

\begin{remark} 
Since a characteristic symbol uniquely corresponds to a Cauchy problem for the LT model, we are concerned with the Cauchy problem for $\mathfrak{C}^e=(\emptyset, 0, \{0\},\{1\})$:
\begin{align*}
\begin{cases}
\dot{T}_j^e=0, \quad t > 0,\quad j \in [N], \\
T_j^e(0)=(T_j^e)^0=1,
\end{cases}
\end{align*}
which yields a trivial solution $T_j^e(t) \equiv 1.$
\end{remark}

\subsection{Multiple tensor model} \label{sec:5.2}
In this subsection, we construct the multiple tensor model by repeatedly applying fusion operation in Definition \ref{D4.2}.  For any given natural number $n \in \bbn$, consider $n$  LT models whose characteristic symbols are given  by 
\begin{align}\label{E-8}
\mathfrak{C}^\ell=(\mathbf{d}^\ell, \mathfrak{K}^\ell, \{A_j^\ell\}, \{(T_j^\ell)^0\}), \quad m_\ell = \textup{dim}(\mathbf{d}^\ell),\quad 1\leq\ell\leq n.
\end{align}
Since repeated fusion operation is well defined, we naturally consider $\bs_{\ell=1}^n \mathfrak{C}^\ell$ and its corresponding Cauchy problem:

\begin{equation}\label{E-9}
\begin{cases}
 \displaystyle[\dot{T}_j]_{\alpha_{*0}}=\left[{\bigstar_\tF }_{\ell=1}^n A_j^\ell\right]_{\alpha_{*0}\alpha_{*1}}[T_j]_{\alpha_{*1}} +\sum_{i_*\in\{0, 1\}^m}\frac{\left({\bs_c}_{\ell=1}^n \mathfrak{K}^\ell\right)_{i_*}}{N} \\
 \displaystyle \hspace{3cm} \times  \sum_{k=1}^N\left([T_k]_{\alpha_{*i_*}}\overline{[{T}_j]}_{\alpha_{*1}}[T_j]_{\alpha_{*(1-i_*)}}-[T_j]_{\alpha_{*i_*}}\overline{[{T}_k]}_{\alpha_{*1}}[T_j]_{\alpha_{*(1-i_*)}}\right),\quad t>0,\\
 \displaystyle T_j(0)=\bigotimes_{\ell=1}^n(T_j^\ell)^0, \quad \|(T_j^\ell)^0\|_\tF=1,\quad  j \in [N],
\end{cases}
\end{equation}
where $m:=\sum_{l=1}^n m_l$ is the sum of ranks. Similar to the double tensor model in Theorem \ref{T4.1}, we propose the multiple tensor model as follows.

\begin{theorem}\label{T5.1}
Let $\mathfrak{C}^1, \mathfrak{C}^2, \cdots, \mathfrak{C}^n\in \mathfrak{D}$ be the characteristic symbols of the Cauchy problems to the  LT models.  If $\{T_j\}_{j=1}^N$ is a solution to the Cauchy problem associated to  $\bs_{\ell=1}^n\mathfrak{C}^\ell$,   then $T_j(t)$ can be decomposed as 
\[
T_j(t)=\bigotimes_{\ell=1}^n T_j^\ell(t), \quad t>0,
\]
where $\{(T_j^\ell)_{\ell=1}^n\}_{j=1}^N$ is a solution to the following system: for $t>0$, 
\begin{align}\label{E-10}
\begin{cases}
\displaystyle[\dot{T}_j^\ell]_{\beta^\ell_{*0}}=[A_j^\ell]_{\beta^\ell_{*0}\beta^\ell_{*1}}[T_j^\ell]_{\beta^\ell_{*1}} +\sum_{j_*\in\{0, 1\}^{m_i}} \frac{\kappa_{j_*}^\ell}{N}  \\
\displaystyle \times \sum_{k=1}^N \Bigg(\prod_{\substack{p=1\\p\neq \ell}}^n\langle T_j^p, T_k^p\rangle_\tF[T_k^\ell]_{\beta^i_{*j_*}}\overline{[{T}_j^\ell]}_{\beta^i_{*1}}[T_j^\ell]_{\beta^\ell_{*(1-j_*)}} -
\prod_{\substack{p=1\\p\neq \ell}}^n\langle T_k^p, T_j^p\rangle_\tF[T_j^\ell]_{\beta^\ell_{*j_*}}\overline{[{T}_k^\ell]}_{\beta^\ell_{*1}}[T_j^\ell]_{\beta^\ell_{*(1-j_*)}}\Bigg), \\
T_j^\ell(0)=(T_j^\ell)^0,\quad \|(T_j^\ell)^0\|_\tF=1,\quad  \ell \in [n], \quad j \in [N].
\end{cases}
\end{align}
\end{theorem}
\begin{proof}
For a proof, we use a mathematical induction and split the proof into two steps. \\

\noindent$\bullet$  Step A (Initial step):~The case of $n=1$ is trivial, and we have already verified the case of $n=2$ in Theorem \ref{T4.1}.\\

\noindent$\bullet$ Step B (Induction step):~ We assume the statement holds for $n=q\geq2$. Now, we  prove the case when $n=q+1$.\\

Let characteristic symbols $\mathfrak{C}^1, \mathfrak{C}^2, \cdots, \mathfrak{C}^q, \mathfrak{C}^{q+1}$ be given, and we set $\{T_j\}$ to be a solution  to the Cauchy problem whose characteristic symbol is  $\bs_{\ell=1}^{q+1}\mathfrak{C}^\ell$. Below, we show that $\{T_j^\ell\}$ with $j \in [N]$ and $\ell \in [q+ 1]$ is a solution to the following system: for $t>0$, 
\begin{equation}\label{E-11}
\begin{cases}
\displaystyle [\dot{T}_j^\ell]_{\beta^\ell_{*0}}=[A_j^\ell]_{\beta^\ell_{*0}\beta^\ell_{*1}}[T_j^\ell]_{\beta^\ell_{*1}}+\sum_{j_*\in\{0, 1\}^{m_i}}\frac{\kappa_{j_*}^\ell}{N }  \\
\displaystyle \hspace{0.2cm} \times \sum_{k=1}^N \Bigg(\prod_{\substack{p=1\\p\neq \ell}}^{q+1}\langle T_j^p, T_k^p\rangle_\tF[T_k^\ell]_{\beta^i_{*j_*}}\overline{[{T}_j^\ell]}_{\beta^i_{*1}}[T_j^\ell]_{\beta^\ell_{*(1-j_*)}} -
\prod_{\substack{p=1\\p\neq \ell}}^{q+1}\langle T_k^p, T_j^p\rangle_\tF[T_j^\ell]_{\beta^\ell_{*j_*}}\overline{[{T}_k^\ell]}_{\beta^\ell_{*1}}[T_j^\ell]_{\beta^\ell_{*(1-j_*)}}\Bigg), \\
T_j^\ell(0)=(T_j^\ell)^0,\quad \|(T_j^\ell)^0\|_\tF=1,\quad \ell \in [q+1], \quad j \in [N].
\end{cases}
\end{equation}
We set 
\[
S_j=T_j^1\otimes T_j^2, \quad  j \in [N].
\]
Then, it follows from the induction hypothesis on $n=q$ that
\[
T_j=S_j\otimes T_j^3\otimes T_j^4\otimes\cdots T_j^{q+1}.
\]
Moreover, each dynamics of $S_j$ and $T_j^\ell$ with $3\leq \ell \leq q+1$ can be represented in  the following forms.\\

\noindent$\diamond$ (Dynamics of $S_j$): By induction hypothesis, we have
\begin{align}
\begin{aligned}\label{E-12}
&\displaystyle[\dot{S}_j]_{\beta^1_{*0}\beta^2_{*0}} =[A_j^1\star_\tF  A_j^2]_{\beta^1_{*0}\beta^2_{*0}\beta^1_{*1}\beta^2_{*1}}[S_j]_{\beta^1_{*1}\beta^2_{*1}} \\
&\hspace{1cm}+\sum_{\substack{j_*\in\{0, 1\}^{m_1}\\k_*\in\{0, 1\}^{m_2}}} \frac{( \mathfrak{K}^1\star_\tF  \mathfrak{K}^2)_{(j_*, k_*)}}{N} \sum_{p=1}^N   \left(\prod_{l=3}^{q+1}\langle T_j^l, T_p^l\rangle_\tF[S_p]_{\beta^1_{*j_*}\beta^2_{*k_*}}\overline{[{S}_j]}_{\beta^1_{*1}\beta^2_{*1}}[S_j]_{\beta^1_{*(1-j_*)}\beta^2_{*(1-k_*)}} \right. \\
&\hspace{1.5cm} -\left.\prod_{l=3}^{q+1}\langle T_p^l, T_j^l\rangle_\tF[S_j]_{\beta^1_{*j_*}\beta^2_{*k_*}}\overline{[{S}_p]}_{\beta^1_{*1}\beta^2_{*1}}[S_j]_{\beta^1_{*(1-j_*)}\beta^2_{*(1-k_*)}}\right).
\end{aligned}
\end{align}
Let $\{T_j^1,T_j^2\}$ be a solution to  \eqref{E-11}. Then, it directly follows from Theorem \ref{D4.1} that $S_j=T_j^1\otimes T_j^2$ is a solution to \eqref{E-12}.\\

\noindent$\diamond$ (Dynamics of $T_j^\ell$ with $3\leq \ell\leq q+1$):~By induction hypothesis, we have
\begin{align}
\begin{aligned} \label{E-13}
[\dot{T}_j^\ell]_{\beta^i_{*0}} &=[A_j^\ell]_{\beta^i_{*0}\beta^i_{*1}}[T_j^\ell]_{\beta^i_{*1}}  \\
&\hspace{0.2cm}+\sum_{j_*\in\{0, 1\}^{m_i}}\sum_{k=1}^N\frac{\kappa_{j_*}^\ell}{N} \Bigg(\langle S_j, S_k\rangle_\tF\prod_{\substack{p=3\\p\neq \ell}}^{q+1}\langle T_j^p, T_k^p\rangle_\tF[T_k^\ell]_{\beta^\ell_{*j_*}}\overline{[{T}_j^\ell]}_{\beta^\ell_{*1}}[T_j^\ell]_{\beta^\ell_{*(1-j_*)}} \\
&\hspace{1cm}-\langle S_k, S_j\rangle_\tF\prod_{\substack{p=3\\p\neq \ell}}^{q+1}\langle T_k^p, T_j^p\rangle_\tF[T_j^\ell]_{\beta^\ell_{*j_*}}\overline{[{T}_k^\ell]}_{\beta^\ell_{*1}}[T_j^\ell]_{\beta^\ell_{*(1-j_*)}}\Bigg).
\end{aligned}
\end{align}
Since the following relation holds
\[
\langle S_j, S_k\rangle_\tF=\langle T_j^1, T_k^1\rangle_\tF\cdot\langle T_j^2, T_k^2\rangle_\tF,
\]
one can rewrite \eqref{E-13} as \eqref{E-11}. This establishes the case for $n=q+1$. Therefore, by the mathematical induction, we obtain the desired assertion for all $n\in \bbn$. 
\end{proof}

Since systems \eqref{D-18}  and \eqref{E-10} share the same structure, we have the following corollary.
\begin{corollary}
Let $\mathfrak{C}_1, \mathfrak{C}_2, \cdots, \mathfrak{C}_n\in \mathfrak{D}$ be characteristic symbols of LT models given as \eqref{E-8} and $\sigma:\{1, 2, \cdots, n\}\to\{1,2, \cdots, n\}$ be a permutation. If $\{T^\sigma_j\}_{j=1}^N$ is a solution to system  $\bs_{\ell=1}^n\chi_{\sigma(\ell)}$, then $T_j(t)$ can be decomposed as 
\[
T^\sigma_j(t)=\bigotimes_{\ell=1}^n T_j^{\sigma(\ell)}(t),
\]
and $\{(T_j^\ell)_{\ell=1}^n\}_{j=1}^N$ is a solution to \eqref{E-9}.
\end{corollary}

\subsection{Weak coupling of the LT models for  low-rank tensors} \label{sec:5.3}
In this subsection, we discuss two weak couplings of the Kuramoto models and the swarm sphere models by means of the fusion operation.

\subsubsection{Same ranks} First, we review two Kuramoto models  from Section \ref{sec:3.2} whose characteristic symbols are given as 
\[
\mathfrak{C}_1=\left(\emptyset,   \frac{\kappa_1}{2}, \{\mathrm{i}\nu^1_j\}, \{e^{\mathrm{i}(\theta_j^1)^0}\}\right),\quad \mathfrak{C}_2=\left(\emptyset ,   \frac{\kappa_2}{2}, \{\mathrm{i}\nu^2_j\}, \{e^{\mathrm{i}(\theta_j^2)^0}\}\right).
\]
Then, the corresponding Kuramoto models correspond to 
\begin{align*}
\mathfrak{C}_1 ~\Longleftrightarrow~
\begin{cases}
\displaystyle\dot{\theta}^1_j=\nu_j^1+\frac{\kappa_1}{N}\sum_{k=1}^N\sin(\theta^1_k-\theta^1_j),\\
\theta_j^1(0)=(\theta_j^1)^0,
\end{cases}
\hspace{0.5cm}
\mathfrak{C}_2~\Longleftrightarrow~
\begin{cases}
\displaystyle\dot{\theta}^2_j=\nu_j^2+\frac{\kappa_2}{N}\sum_{k=1}^N\sin(\theta^2_k-\theta^2_j),\\
\theta_j^2(0)=(\theta_j^2)^0.
\end{cases}
\end{align*}
Then, the weak coupling of two characteristic symbols $\mathfrak{C}:=\mathfrak{C}_1\star \mathfrak{C}_2$ is explicitly  given  as follows:
\[
\mathfrak{C} =\mathfrak{C}_1\star \mathfrak{C}_2=\left( \emptyset, \frac{\kp_1+\kp_2}{2} , \{\mathrm{i}(\nu^1_j+\nu^2_j)\}, \{e^{\mathrm{i}\left((\theta_j^1)^0+(\theta_j^2)^0\right)}\}\right),
\]
and the corresponding Cauchy problem becomes 
\begin{equation} \label{E-14}
\begin{cases}
\displaystyle\dot{\theta}^1_j+\dot{\theta}^2_j=(\nu_j^1+\nu_j^2)+\frac{\kappa_1+\kappa_2}{N}\sum_{k=1}^N\sin((\theta_k^1+\theta_k^2)-(\theta_j^1+\theta_j^2)),\\
(\theta_j^1+\theta_j^2)(0)=(\theta_j^1)^0+(\theta_j^2)^0.
\end{cases}
\end{equation}
Now, we introduce new variables:
\begin{equation} \label{E-15}
\varphi_j := \theta_j^1 + \theta_j^2, \quad \nu_j := \nu_j ^1 + \nu_j^2,\quad \kp:= \kp_1 + \kp_2.
\end{equation}
Then, system \eqref{E-14} can cast in terms of \eqref{E-15}:
\[
\dot \varphi_j =\nu_j + \frac\kp N \sum_{k=1}^N \sin (\varphi_k - \varphi_j),
\]
which is exactly the Kuramoto model. Thus, the weak coupling of two Kuramoto models is still the Kuramoto model. \newline

Second, we consider the weak coupling  of two swarm sphere models whose characteristic symbols are given by 
\begin{equation*} \label{E-16}
\mathfrak{C}_1=(d_1, (\kappa_{1},0), \{\Omega_j^1\}, \{(x_j^1)^0\}),\quad \mathfrak{C}_2=(d_2, (\kappa_{2}, 0), \{\Omega^2_j\}, \{(x_j^2)^0\}),
\end{equation*} 
with $\Omega_j^1$ being a $d_1\times d_1$ skew-symmetric matrix and $\Omega_j^2$ being a $d_2\times d_2$ skew-symmetric matrix. Furthermore, $(x_j^1)^0$ and $(x_j^2)^0$ are real-valued with unit norms. In particular, we set the second components of coupling strength tensors to be zero for simplicity. Then, the fusion of two characteristic symbols $\mathfrak{C} = \mathfrak{C}_1\star \mathfrak{C}_2$ can be calculated as
\begin{equation} \label{E-16-1}
\mathfrak{C}_1\star \mathfrak{C}_2=\left((d_1, d_2), \begin{pmatrix}
0&\kappa_2\\
\kappa_1&0
\end{pmatrix},
\{\Omega_j^1\star_\tF  \Omega_j^2 \}, \{(x_j^1)^0\otimes (x_j^2)^0\}\right),
\end{equation}
which is exactly the same  as \eqref{C-7-2} except for the natural frequency and the initial data. Thus, the weak coupling of two swarm sphere models is the generalized Lohe matrix model. However, it follows from \cite{H-K-P1} that the weakly coupled model corresponding to the characteristic symbol \eqref{E-16-1} is the double sphere model \eqref{B-1} (see Proposition \ref{P2.2}). \newline

\subsubsection{Different ranks} So far, we have investigated the fusion of two characteristic symbols with the same ranks. Now, we return to the fusion of two characteristic symbols with {\it different} ranks.  We recall the quantum Lohe model introduced in Section \ref{sec:4.1}: 
\begin{align}\label{E-17}
\begin{cases}
\displaystyle\dot{\theta}_j=\nu_j+\frac{\kappa_1}{N}\sum_{k=1}^N \langle x_j, x_k\rangle\sin(\theta_k-\theta_j), \quad t>0, \vspace{0.2cm}\\
\displaystyle\dot{x}_j=\Omega_j x_j+\frac{\kappa_2}{N}\sum_{k=1}^N\cos(\theta_k-\theta_j)(x_k-\langle x_j, x_k\rangle x_j),\\
(\theta_j(0), x_j(0))=(\theta_j^0, x_j^0)\in\bbr\times \bbs^{d-1},\quad j\in [N].
\end{cases}
\end{align}
Recall that model \eqref{E-17} has been derived from the $\mathbf{U}(2)$ parametrization of the Lohe matrix model. Our goal is to verify that model \eqref{E-17} can be obtained from the weak coupling of two models: the Kuramoto model and the swarm sphere model.  First, we rewrite \eqref{E-17} in the form of the double sphere model. Precisely, for given $\theta_j$ in \eqref{E-17}, we introduce new variables in $\bbs^1$:
\[
y_j=(\cos\theta_j, \sin\theta_j)^\top\in\bbs^1,
\]
Then, system \eqref{E-17} can be rewritten as follows:
\begin{align}\label{E-18}
\begin{cases}
\displaystyle\dot{y}_j=\nu_jJy_j+\frac{\kappa_1}{N}\sum_{k=1}^N\left(\langle x_j, x_k\rangle y_ky_j^\top y_j-\langle x_k, x_j\rangle y_jy_k^\top y_j\right),\quad J=\begin{bmatrix}
0&-1\\
1&0\\
\end{bmatrix},
\\
\displaystyle\dot{x}_j=\Omega_j x_j+\frac{\kappa_2}{N}\sum_{k=1}^N\left(\langle y_j, y_k\rangle x_kx_j^\top x_j-\langle y_k, y_j\rangle x_jx_k^\top x_j\right),\\
y_j(0)=(\cos\theta_j^0, \sin\theta_j^0)^\top,\quad x_j(0)=x_j^0.
\end{cases}
\end{align}
If the subsystems of \eqref{E-18} are decoupled, then we have
\begin{align*}\label{E-19}
\begin{cases}
\displaystyle\dot{y}_j=\nu_jJy_j+\frac{\kappa_1}{N}\sum_{k=1}^N\left(y_ky_j^\top y_j- y_jy_k^\top y_j\right),~t > 0,\\
\displaystyle\dot{x}_j=\Omega_j x_j+\frac{\kappa_2}{N}\sum_{k=1}^N\left( x_kx_j^\top x_j- x_jx_k^\top x_j\right),\\
y_j(0)=(\cos\theta_j^0, \sin\theta_j^0)^\top,\quad x_j(0)=x_j^0.
\end{cases}
\end{align*}
Thus, two characteristic symbols of two equations in \eqref{D-5} become 
\[
\mathfrak{C}_1=(2, \kappa_1, \{\nu_j J\}, \{(\cos\theta_j^0, \sin\theta_j^0)^\top\}),\quad \mathfrak{C}_2=(d, \kappa_2,\{\Omega_j\}, \{x_j^0\}).
\]
Finally, we conclude that the characteristic symbol of \eqref{E-18} can be obtained as the fusion of $\mathfrak{C}_1$ and $\mathfrak{C}_2$:
\[
\mathfrak{C}_1\star\mathfrak{C}_2=\left((2, d), \begin{bmatrix}
0&\kappa_1\\
\kappa_2&0
\end{bmatrix}, \{\nu_j J\star \Omega_j\}, \{(\cos\theta_j^0, \sin\theta_j^0)^\top\otimes x_j^0\}
\right).
\]
Note that the characteristic symbol of the quantum Lohe model is a special case of the   double sphere model. So, we can conclude that the quantum Lohe model is the special case of the   double sphere model.

\subsection{Gradient flow formulation} \label{sec:5.4}
In this subsection, we discuss conditions in which the multiple tensor model can be represented as a gradient flow with a suitable analytical potential. Note that multiple sphere and multiple matrix models introduced in Section \ref{sec:2} can be written as gradient flows with total distance functionals as potentials.

 Let $\{T_j\}$ be a solution to the double tensor model \eqref{C-3}. As in \cite{H-K-P2, H-K-P1}, we consider the following potential function:
\begin{equation} \label{E-20} 
\mathcal{V}(\{T_j\})=\frac{\kappa}{2N}\sum_{i, j=1}^N\langle T_i, T_j\rangle_\tF.
\end{equation} 
If there exist two tensors $T_j^1$ and $T_j^2$ such that $T_j = T_j^1 \otimes T_j^2$, then  system \eqref{E-20} becomes
\[
\mathcal{V}(\{T_j\})=\frac{\kappa}{2N}\sum_{i, j=1}^N\langle T_i^1, T_j^1\rangle_\tF\cdot\langle T_i^2, T_j^2\rangle_\tF.
\]
Thus,  if a solution to the multiple tensor model is decomposed by using tensor products, then the   the multiple tensor model becomes a gradient system. 

\begin{theorem} \label{T5.2} 
Let $\mathfrak{C}^1$ and $\mathfrak{C}^2$ be given characteristic symbols, and let $T_j^1$ and $T_j^2$ be the solutions to the LT models \eqref{LT} corresponding to characteristic symbols $\mathfrak{C}_1$ and $\mathfrak{C}_2$, respectively. Suppose that the systems for $\mathfrak{C}_1$ and $\mathfrak{C}_2$ are represented as gradient flows with the following potential functionals:
\begin{align*}
\mathcal{V}_1(\{T_j^1\})=\frac{\kappa}{2N}\sum_{i, j=1}^N\langle T_i^1,T_j^1\rangle_\tF,\quad
\mathcal{V}_2(\{T_j^2\})=\frac{\kappa}{2N}\sum_{i, j=1}^N\langle T_i^2,T_j^2\rangle_\tF,
\end{align*}
on state spaces $M_1$ and $M_2$, respectively. Then, the weakly coupled system corresponding to the characteristic symbol $\mathfrak{C}^1\star \mathfrak{C}^2$ can be expressed as a gradient flow with the following potential:
\[
\mathcal{V}(\{T_j^1\otimes T_j^2\})=\frac{\kappa}{2N}\sum_{i, j=1}^N\langle T_i^1, T_j^1\rangle_\tF\cdot\langle T_i^2, T_j^2\rangle_\tF,\quad T_j^1 \otimes T_j^2 \in M_1 \times M_2.
\]
\end{theorem}
\begin{remark}
It is worthwhile to mention that  Theorem \ref{T5.2} is consistent with the results of \cite{H-K-P2, H-K-P1}.  More precisely, let  $\{T_j\}=\{T_j^1\otimes T_j^2\}$ be a solution to the LT model with characteristic symbol $\mathfrak{C}^1\star \mathfrak{C}^2$. For the case of rank-1 real tensors, then the LT model corresponds to the swarm sphere model, and $\{T_j\}$ lies on $\bbs^{d-1}$. Then, one finds 
\[
\dot{T}_j^1=-\nabla_{T_j^1} \mathcal{V}(\{T_j\}) \Big|_{T_{T_j^1} \bbs^{d-1}}=\frac{\kappa}{N}\sum_{k=1}^N\langle T_j^2, T_k^2\rangle_\tF\Big(
T_k^1-\langle T_k^1, T_j^1\rangle T_j^1\Big),
\]
where $T_{T_j^1} \bbs^{d-1}$ is the tangent space of $\bbs^{d-1}$ at $T_j^1$. This exactly coincides with the double sphere model. On the other hand, if  we consider rank-2 tensors and assume that $T_j$ lies on  $\mathbb{U}(d)$, then the model reduces to  
\[
\dot{T}_j^1= -\nabla_{T_j^1}\mathcal{V}(\{T_j\}) \Big|_{T_{T_j^1} \mathbb{U}(d)} =\frac{\kappa}{2N}\sum_{k=1}^N \Big(\langle T_j^2, T_k^2\rangle_\tF T_k^1-\langle T_k^2,T_j^2\rangle_\tF T_j^1 (T_k^1)^\dagger T_j^1\Big).
\]
Thus, we obtain the double matrix model. 
\end{remark}

By  mathematical induction, Theorem \ref{T5.2} can be extended to the multiple tensor model. Since the proof is straightforward, we omit the proof. 

\begin{corollary}
For $\ell \in [n]$, let  $\mathfrak{C}^\ell$ be the characteristic symbol of the system that can be represented as a gradient flow on a manifold $M_\ell$: 
\[
\mathcal{V}_\ell(\{T_j^\ell\})=\frac{\kappa}{2N}\sum_{i, j=1}^N\langle T_i^\ell, T_j^\ell\rangle_\tF,
\]
on manifold $M_\ell$ for all $\ell\in [n]$. Then, system whose characteristic symbol is given by $\bs_{\ell=1}^n \mathfrak{C}^\ell$ can   be also represented as a gradient flow with the following potential function on the product manifold $M_1\times M_2\times\cdots\times M_n$:
\[
\mathcal{V}\left(\left\{\bigotimes_{\ell=1}^nT_j^\ell\right\}\right)=\frac{\kappa}{2N}\sum_{i, j=1}^N\left(\prod_{\ell=1}^n\langle T_i^\ell, T_j^\ell\rangle_\tF\right).
\]
\end{corollary}

\subsection{Index shuffling operator and induced transforms}
In this section, we further investigate  the structure of $(\mathfrak D, \star)$. Heuristically, suppose that we consider a weakly coupled model from two LT models whose solutions are denoted by $\{T_j^1\}$ and $\{T_j^2\}$. Then, a solution to the weakly coupled model becomes $T_j^1 \otimes T_j^2$. One natural question is whether $T_j^2\otimes T_j^1$ also becomes a solution to the same weakly coupled model. Our goal of this subsection is to provide an affirmative answer. For this purpose, we introduce the index shuffling operator (or a permutation) $\sigma:[m]\to[m]$ on rank-$m$ tensor, which is a one-to-one correspondence, and we set the collection of index shuffling operator on $[m]$ to be 
$\mathfrak{S}_m$, which is the permutation group of $m$ letters.  For simplicity, we use the following handy notation: for $\sigma \in \mathfrak{S}_m$, 
\begin{align*}
&\alpha_{\sigma(*)} =\alpha_{\sigma(1)}\alpha_{\sigma(2)}\cdots\alpha_{\sigma(m)}, \quad 
\alpha_{\sigma(*)0} =\alpha_{\sigma(1)0}\alpha_{\sigma(2)0}\cdots\alpha_{\sigma(m)0},  \\
& \alpha_{\sigma(*)1} =\alpha_{\sigma(1)1}\alpha_{\sigma(2)1}\cdots\alpha_{\sigma(m)1}.
\end{align*} 
The induced transform on a rank $m$-tensor $T$ can be defined as follows:
\[
[T^\sigma]_{\alpha_{\sigma(*)}}:=[T]_{\alpha_*},\quad\forall~\alpha_*=(\alpha_1, \alpha_2,\cdots, \alpha_m)\in \prod_{i=1}^m[d_i].
\]
We can easily check that $T^\sigma\in \bbc^{d_{\sigma(1)}\times d_{\sigma(2)}\times\cdots\times d_{\sigma(m)}}$, i.e.,
\[
\alpha_{\sigma(i)}\in [d_{\sigma(i)}],\quad i\in [m].
\]
\subsubsection{Size vector and rank} Since $\sigma$ is an index shuffling operator, the rank of the tensor is not changed.  However, it changes the size vector of the tensor $\mathbf{d}$ as follows:
\[
\mathbf{d}^\sigma=(d_{\sigma(1)}, d_{\sigma(2)},\cdots, d_{\sigma(m)}).
\]
Here, $\mathbf{d}^\sigma$ can be interpreted as an induced transform of $\mathbf{d}$ by the index shuffling operator $\sigma$.

\subsubsection{Natural frequency tensor} 
Consider a rank-$2m$ natural frequency tensor $A$ satisfying the block skew-Hermitian property:
\[
[A]_{\alpha_{*0}\alpha_{*1}}=-[\bar{A}]_{\alpha_{*1}\alpha_{*0}}.
\]
The induced transform on this rank $2m$-tensor $A$ can be defined as follows:
\[
[A^\sigma]_{\alpha_{\sigma(*)0}\alpha_{\sigma(*)1}}:=[A]_{\alpha_{*0}\alpha_{*1}}.
\]
Then, it is easy to check that $A^\sigma$ satisfies
\[
[A^\sigma]_{\alpha_{\sigma(*)0}\alpha_{\sigma(*)1}}=[A]_{\alpha_{*0}\alpha_{*1}}=-[\bar{A}]_{\alpha_{*1}\alpha_{*0}}=-[\bar{A}^\sigma]_{\alpha_{\sigma(*)1}\alpha_{\sigma(*)0}},
\]
and we conclude that $A^\sigma$ is also a rank-$2m$ natural frequency tensor of size $(d_{\sigma(1)}\times d_{\sigma(2)}\times\cdots\times d_{\sigma(m)})\times(d_{\sigma(1)}\times d_{\sigma(2)}\times\cdots\times d_{\sigma(m)})$. 

\subsubsection{Coupling strength tensor} 
Finally, we introduce the induced transform of $\mathfrak{K}$ with respect to the index shuffling operator $\sigma$ as follows: 
\[
[\mathfrak{K}^\sigma]_{i_{\sigma(*)}}:=[\mathfrak{K}]_{i_*}\quad i_*=(i_1, i_2,\cdots, i_m)\in\{0, 1\}^m,
\]
where $i_{\sigma(*)}=i_{\sigma(1)}i_{\sigma(2)}\cdots i_{\sigma(m)}$.

\vspace{0.5cm}

Below, we provide a brief summary of the induced transform of  four components for the characteristic symbol by the index shuffling operator $\sigma$: 
\begin{itemize}
\item Size vector:
\[
\mathbf{d}^\sigma=(d_{\sigma(1)}, d_{\sigma(2)},\cdots, d_{\sigma(m)}).
\]
\item Rank-$m$ tensors:
\[
[T^\sigma]_{\alpha_{\sigma(*)}}=[T]_{\alpha_*}.
\]
\item Rank-$2m$ block skew-Hermitian tensors:
\[
[A^\sigma]_{\alpha_{\sigma(*)0}\alpha_{\sigma(*)1}}=[A]_{\alpha_{*0}\alpha_{*1}}.
\]

\item Coupling strength tensor:
\[
[\mathfrak{K}^\sigma]_{i_{\sigma(*)}}=[\mathfrak{K}]_{i_*}.
\]
\end{itemize}

Since the characteristic symbol of the LT model $\mathfrak{C}=(\mathbf{d}, \mathfrak{K}, \{A_j\}, \{T_j^0\})$ is determined by the four components, we can also define the induced transform of the characteristic symbol with respect to the index shuffling operator $\sigma$ as follows:
\[
\mathfrak{C}^\sigma:=(\mathbf{d}^\sigma, \mathfrak{K}^\sigma, \{A_j^\sigma\}, \{(T_j^0)^\sigma\}).
\]
Now, we can find the Cauchy problem corresponding to the induced characteristic symbol $\mathfrak{C}^\sigma$. 

\begin{lemma}
If $\{T_j\}$ is a solution to $\mathfrak{C}$, then $\{T_j^\sigma\}$ is a solution to $\mathfrak{C}^\sigma$ for any index shuffling operator $\sigma$.
\end{lemma}
\begin{proof}
First, we fix the index shuffling operator $\sigma \in \mathfrak{S}_m$ and the characteristic symbol: 
\[
\mathfrak{C}=(\mathbf{d}, \mathfrak{K}, \{A_j\}, \{T_j^0\}).
\]
Then, $\{T_j\}$ is a solution to the Cauchy problem for $\mathfrak{C}$ which reads as 
\begin{equation} \label{G-1}
\begin{cases}
[\dot{T}_j]_{\alpha_{*0}}=[A_j]_{\alpha_{*0}\alpha_{*1}}[T_j]_{\alpha_{*1}}\\
+\displaystyle\sum_{i_*\in\{0, 1\}^m}\kappa_{i_*}\left([T_c]_{\alpha_{*i_*}}[\bar{T}_j]_{\alpha_{*1}}[T_j]_{\alpha_{*(1-i_*)}}-[T_j]_{\alpha_{*i_*}}[\bar{T}_c]_{\alpha_{*1}}[T_j]_{\alpha_{*(1-i_*)}}\right),~t > 0, \vspace{0.2cm}\\
T_j(0)=T_j^0\in \bbc^{d_1\times d_2\times\cdots \times d_m},
\end{cases}
\end{equation} 
where $T_c :=\frac{1}{N}\sum_{k=1}^NT_k$. In what follows, we will show that $\{T_j^\sigma\}$ is a solution to the system defined by a characteristic $\mathfrak{C}^\sigma$. \newline 

\noindent $\bullet$ (Initial data): it is easy to check that 
\[
(T_j^\sigma)^0=T_j^\sigma(0)=(T_j^0)^\sigma.
\]

\noindent $\bullet$ (Dynamics):  it directly follows from the dynamics \eqref{G-1} of $T_j$ that 
\begin{align}
\begin{aligned}  \label{G-2}
  [\dot T_j^\sigma]_{\alpha_{\sigma(*)0}}&=[\dot{T}_j]_{\alpha_{*0}}\\
&=[A_j]_{\alpha_{*0}\alpha_{*1}}[T_j]_{\alpha_{*1}}\\
&\hspace{0.5cm}+\displaystyle\sum_{i_*\in\{0, 1\}^m}\kappa_{i_*}\left([T_c]_{\alpha_{*i_*}}[\bar{T}_j]_{\alpha_{*1}}[T_j]_{\alpha_{*(1-i_*)}}-[T_j]_{\alpha_{*i_*}}[\bar{T}_c]_{\alpha_{*1}}[T_j]_{\alpha_{*(1-i_*)}}\right).
\end{aligned}
\end{align} 

$\diamond$ (Free flow): the free flow can be represented as
\begin{equation} \label{G-3}
[A_j]_{\alpha_{*0}\alpha_{*1}}[T_j]_{\alpha_{*1}}=[A_j^\sigma]_{\alpha_{\sigma(*)0}\alpha_{\sigma(*)1}}[T^\sigma_j]_{\alpha_{\sigma(*)1}}.
\end{equation}

$\diamond$ (Coupling flow):  For handy notation, we write 
\begin{align*}
\alpha_{\sigma(*)i_{\sigma(*)}}&=\alpha_{\sigma(1)i_{\sigma(1)}}\alpha_{\sigma(2)i_{\sigma(2)}}\cdots\alpha_{\sigma(m)i_{\sigma(m)}},\\
\alpha_{\sigma(*)(1-i_{\sigma(*)})}&=\alpha_{\sigma(1)(1-i_{\sigma(1)})}\alpha_{\sigma(2)(1-i_{\sigma(2)})}\cdots\alpha_{\sigma(m)(1-i_{\sigma(m)})}.
\end{align*}
Then, the coupling term in the right-hand side of \eqref{G-2} can be represented as 
\begin{align} \label{G-4}
\begin{aligned}
&\sum_{i_*\in\{0, 1\}^m}\kappa_{i_*}\left([T_c]_{\alpha_{*i_*}}[\bar{T}_j]_{\alpha_{*1}}[T_j]_{\alpha_{*(1-i_*)}}-[T_j]_{\alpha_{*i_*}}[\bar{T}_c]_{\alpha_{*1}}[T_j]_{\alpha_{*(1-i_*)}}\right)\\
&\hspace{1cm} =\sum_{i_{\sigma(*)}\in\{0, 1\}^m}\kappa^\sigma_{i_{\sigma(*)}}\Big([T^\sigma_c]_{\alpha_{\sigma(*)i_{\sigma(*)}}}[\bar{T}_j^\sigma]_{\alpha_{\sigma(*)1}}[T_j^\sigma]_{\alpha_{\sigma(*)(1-i_{\sigma(*)})}} \\
&\hspace{4cm}-[T_j^\sigma]_{\alpha_{\sigma(*)i_{\sigma(*)}}}[\bar{T}_c^\sigma]_{\alpha_{\sigma(*)1}}[T_j^\sigma]_{\alpha_{\sigma(*)(1-i_{\sigma(*)})}}\Big).
\end{aligned}
\end{align}
In \eqref{G-2}, we combine \eqref{G-3} and \eqref{G-4} to find dynamical system for $T_j^{\sigma}$:
\[
\begin{cases}
\frac{d}{dt} [T_j^\sigma]_{\alpha_{\sigma(*)0}}=[A_j^\sigma]_{\alpha_{\sigma(*)0}\alpha_{\sigma(*)1}}[T^\sigma_j]_{\alpha_{\sigma(*)1}}\\
\displaystyle \hspace{2cm} +\sum_{i_{\sigma(*)}\in\{0, 1\}^m}\kappa^\sigma_{i_{\sigma(*)}}\Big([T^\sigma_c]_{\alpha_{\sigma(*)i_{\sigma(*)}}}[\bar{T}_j^\sigma]_{\alpha_{\sigma(*)1}}[T_j^\sigma]_{\alpha_{\sigma(*)(1-i_{\sigma(*)})}} \\
\displaystyle\hspace{4cm}-[T_j^\sigma]_{\alpha_{\sigma(*)i_{\sigma(*)}}}[\bar{T}_c^\sigma]_{\alpha_{\sigma(*)1}}[T_j^\sigma]_{\alpha_{\sigma(*)(1-i_{\sigma(*)})}}\Big),\quad t > 0, \\
T_j^\sigma(0)=(T_j^0)^\sigma.
\end{cases}
\]
Finally, we can conclude that $\{T_j^\sigma\}$ is a solution to the Cauchy problem whose  characteristic symbol is $\mathfrak{C}^\sigma$.
\end{proof}

\begin{corollary}
Let $\{T_j\}$ be a solution to the Cauchy problem corresponding to the characteristic symbol $\mathfrak{C}$. Suppose that  $\mathfrak{C}$ has the following symmetry:
\[
\mathfrak{C}=\mathfrak{C}^\sigma
\]
for some permutation $\sigma\in\mathfrak{S}_m$. Then, the solution $\{T_j\}$ also has the same type of symmetry:
\[
T_j(t)=T_j^\sigma(t),\quad t\geq0,\quad j\in[N].
\]
\end{corollary}

\vspace{0.5cm}

Since $\{T_j\}$ and $\{T_j^\sigma\}$ are just rearrangement of each other, two characteristic symbols $\mathfrak{C}$ and $\mathfrak{C}^\sigma$ are essentially identical. To fix the idea,  we define an equivalent relation $\sim$ on $\mathfrak{D}$ as follows:
\[
\mathfrak{C}\sim \mathfrak{C}^\sigma,\quad \mathfrak{C}\in\mathfrak{D}_m,\quad \sigma\in \mathfrak{S}_m,\quad m\in \bbn\cup\{0\}.
\]
 Moreover,  the equivalence class is defined as 
\[
[\mathfrak{C}]_\sim :=\{\mathfrak{C}^\sigma:\sigma\in \mathfrak{S}_m\}.
\]
Now, we define the binary operation, denoted by $\star_\sim$, between two equivalence classes. In fact, such binary operation can be naturally defined from the fusion operation: for $\mathfrak{C}_1 \in \mathfrak{D}_{m_1}$ and $\mathfrak{C}_2 \in \mathfrak{D}_{m_2}$, 
\[
[\mathfrak{C}_1]\star_\sim[\mathfrak{C}_2]:=[\mathfrak{C}_1\star\mathfrak{C}_2].
\]
In the following lemma, we show that the binary operation $\star_\sim$ is well-defined.

\begin{lemma}\label{L3.2}
The binary operation $\star_\sim$ is well-defined.
\end{lemma}

\begin{proof}
Let  $\mathfrak{C}_1,\mathfrak{C}_2$ be two characteristic symbols whose ranks are $m_1$ and $m_2$, respectively. In order to establish the well-definedness of $\star_\sim$, it suffices to show that 
\begin{equation} \label{G-10}
\mathfrak{C}_1\star\mathfrak{C}_2\sim\mathfrak{C}_1^{\sigma_1}\star\mathfrak{C}_2^{\sigma_2},\quad \textup{for all $\sigma_1\in \mathfrak{S}_{m_1}$ and $\sigma_2\in \mathfrak{S}_{m_2}$.}
\end{equation}
To verify the relation \eqref{G-10}, we need to construct an index shuffling operator (or a permutation) on $m_1+m_2$ letters such that
\begin{equation} \label{G-11}
(\mathfrak{C}_1\star\mathfrak{C}_2)^\sigma=\mathfrak{C}_1^{\sigma_1}\star\mathfrak{C}_2^{\sigma_2}.
\end{equation}
In this regard, we define $\sigma :=\sigma_1\otimes \sigma_2\in \mathfrak{S}_{m_1+m_2}$ as
\begin{equation} \label{G-12}
\sigma(i)=\begin{cases}
\sigma_1(i)\quad&\text{for }1\leq i\leq m_1,\\
 m_2+\sigma_2(i-m_1)\quad&\text{for }m_1+1\leq i\leq m_1+m_2.
\end{cases}
\end{equation}
The first $m_1$ components of  $\sigma$ come from $\sigma_1$ and the last $m_2$ terms come from $\sigma_2$. For $\sigma$ defined in \eqref{G-12}, we verify the relation \eqref{G-11} by considering four components of the characteristic symbols one by one. \newline

\noindent (i) (Size vector): by straightforward calculation, one has 
\begin{align*}
(\mathbf{d}^1\star_s \mathbf{d}^2)^\sigma&=(d^1_{\sigma_1(1)}, \cdots, d^1_{\sigma_1(m_1)}, d^2_{\sigma_2(1)}, \cdots, d^2_{\sigma_2(m_2)})
=(\mathbf{d}^1)^{\sigma_1}\star_s(\mathbf{d}^2)^{\sigma_2}.
\end{align*}

\vspace{0.5cm}

\noindent (ii) (Natural frequency tensors): again, we observe  
\begin{align*}
&[(A^1_j\star_\tF  A_j^2)^\sigma]_{\beta_{\sigma_1(*)0}\gamma_{\sigma_2(*)0}\beta_{\sigma_1(*)1}\gamma_{\sigma_2(*)1}} \\
&\hspace{0.5cm}=[A^1_j\star_\tF  A_j^2]_{\beta_{*0}\gamma_{*0}\beta_{*1}\gamma_{*1}} =[A_j^1]_{\beta_{*0}\beta_{*1}}\delta_{\gamma_{*0}\gamma_{*1}}+\delta_{\beta_{*0}\beta_{*1}}[A_j^2]_{\gamma_{*0}\gamma_{*1}}\\
&\hspace{0.5cm}=[(A_j^1)^{\sigma_1}]_{\beta_{\sigma_1(*)0}\beta_{\sigma_1(*)1}}\delta_{\gamma_{\sigma_2(*)0}\gamma_{\sigma_2(*)1}}+\delta_{\beta_{\sigma_1(*)0}\beta_{\sigma_1(*)1}}[(A_j^2)^{\sigma_2}]_{\gamma_{\sigma_2(*)0}\gamma_{\sigma_2(*)1}}\\
&\hspace{0.5cm}=[(A_j^1)^{\sigma_1}\star_\tF (A_j^2)^{\sigma_2}]_{\beta_{\sigma_1(*)0}\gamma_{\sigma_2(*)0}\beta_{\sigma_1(*)1}\gamma_{\sigma_2(*)1}}.
\end{align*}
Hence, we see
\[
 (A_j^1\star_\tF  A_j^2)^\sigma=(A_j^1)^{\sigma_1}\star_\tF (A_j^2)^{\sigma_2}. 
\]

\vspace{0.5cm}

\noindent (iii) (Coupling strength tensors): it suffices to show the desired relation componentwise: 
\begin{equation} \label{G-15}
((\mathfrak{K}^1\star_c\mathfrak{K}^2)^\sigma)_{(j_{\sigma_1(*)}, k_{\sigma_2(*)})}=(\mathfrak{K}^1\star_c\mathfrak{K}^2)_{(j_*, k_*)}.
\end{equation}
By the definition of $\star_c$, we observe 
\begin{align*}
((\mathfrak{K}^1\star_c\mathfrak{K}^2)^\sigma)_{(j_{\sigma_1(*)}, k_{\sigma_2(*)})} &:=\begin{cases}
\kappa_{j_*}^1\quad& \text{if }j_*\neq \mathbf{1}_{m_1}\text{~and~}k_*= \mathbf{1}_{m_2},\\
\kappa_{k_*}^2\quad& \text{if }j_*= \mathbf{1}_{m_1}\text{~and~}k_*\neq \mathbf{1}_{m_2},\\
\kappa_{\mathbf{1}_{m_1}}^1+\kappa_{\mathbf{1}_{m_2}}^2\quad&\text{if }j_*= \mathbf{1}_{m_1}\text{~and~}k_*= \mathbf{1}_{m_2},\\
0\quad&\text{if }j_*\neq \mathbf{1}_{m_1}\text{~and~}k_*\neq \mathbf{1}_{m_2} 
\end{cases} \\
&=\begin{cases}
(\kappa^1)^{\sigma_1}_{j_{\sigma_1(*)}}\quad& \text{if }j_*\neq \mathbf{1}_{m_1}\text{~and~}k_*= \mathbf{1}_{m_2},\\
(\kappa^2)^{\sigma_2}_{k_{\sigma_2(*)}}\quad& \text{if }j_*= \mathbf{1}_{m_1}\text{~and~}k_*\neq \mathbf{1}_{m_2},\\
(\kappa^1)^{\sigma_1}_{\mathbf{1}_{m_1}}+(\kappa^2)^{\sigma_2}_{\mathbf{1}_{m_2}}\quad&\text{if }j_*= \mathbf{1}_{m_1}\text{~and~}k_*= \mathbf{1}_{m_2},\\
0\quad&\text{if }j_*\neq \mathbf{1}_{m_1}\text{~and~}k_*\neq \mathbf{1}_{m_2},
\end{cases}\\
&=((\mathfrak{K}^1)^{\sigma_1}\star_c(\mathfrak{K}^2)^{\sigma_2})_{(j_{\sigma_1(*)}, k_{\sigma_2(*)})}.
\end{align*}
Hence, we verify  \eqref{G-15}. 

\vspace{0.5cm}

\noindent (iv) (Initial configuration): By direct observation, one has 
\begin{align*}
[((T_j^1)^0\star_i (T_j^2)^0)^\sigma]_{\beta_{\sigma_1(*)}\gamma_{\sigma_2(*)}}&=[(T_j^1)^0\star_i(T_j^2)^0]_{\beta_*\gamma_*}=[(T_j^1)^0]_{\beta_*}[(T_j^2)^0]_{\gamma_*}\\
&=[(T_j^1)^0]_{\beta_*}[(T_j^2)^0]_{\gamma_*}=[((T_j^1)^0)^{\sigma_1}]_{\beta_{\sigma_1(*)}}[((T_j^2)^0)^{\sigma_2}]_{\gamma_{\sigma_2(*)}}\\
&=[((T_j^1)^0)^{\sigma_1}\star_i((T_j^2)^0)^{\sigma_2}]_{\beta_{\sigma_1(*)}\gamma_{\sigma_2(*)}}.
\end{align*}
This yields  the desired relation:
\[
((T_j^1)^0\star_i (T_j^2)^0)^\sigma=((T_j^1)^0)^{\sigma_1}\star_i((T_j^2)^0)^{\sigma_2}.
\]
By combining (i), (ii), (iii) and (iv), we obtain the desired relation \eqref{G-11} which gives \eqref{G-10}.
\end{proof}

Since the operator $\star_\sim$ is well-defined, we know that the quotient object $(\mathfrak{D}/\sim, \star_\sim)$  has also a monoid structure inherited from $(\mathfrak{D}, \star)$. Indeed, we show that the quotient object $(\mathfrak{D}/\sim, \star_\sim)$ is a commutative monoid. 

\begin{theorem}
$(\mathfrak{D}/_\sim, \star_\sim)$ is a commutative monoid. 
\end{theorem}

\begin{proof}
The fact that $(\mathfrak{D}/_\sim, \star_\sim)$ becomes a monoid directly follows from the fact that $(\mathfrak D,\star)$ is a monoid. Thus, it suffices to show that for any $\mathfrak C_1$ and $\mathfrak C_2$, 
\begin{equation} \label{G-18}
[\mathfrak{C}_1]\star_\sim[\mathfrak{C}_2]=[\mathfrak{C}_2]\star_\sim[\mathfrak{C}_1].
\end{equation}
For \eqref{G-18}, we claim that
\begin{equation} \label{G-19}
\mathfrak{C}_1\star \mathfrak{C}_2\sim \mathfrak{C}_2\star\mathfrak{C}_1.
\end{equation}
Let $m_1$ and $m_2$ be the ranks of  $\mathfrak C_1$ and $\mathfrak C_2$, respectively. Then, we define the permutation on $m_1+m_2$ letters  as follows:
\[
\sigma(i) :=\begin{cases}
i+m_2,\quad&1\leq i\leq m_1,\\
i-m_1,\quad&m_1+1\leq i\leq m_1+m_2.
\end{cases}
\]
By similar arguments in Lemma \ref{L3.2}, we can see
\[
(\mathfrak{C}_1\star\mathfrak{C}_2)^\sigma=\mathfrak{C}_2\star\mathfrak{C}_1.
\]
This gives \eqref{G-19}. Then, we have 
\[
[\mathfrak{C}_1]\star_\sim[\mathfrak{C}_2]=[\mathfrak{C}_1\star\mathfrak{C}_2]=[\mathfrak{C}_2\star\mathfrak{C}_1]=[\mathfrak{C}_2]\star_\sim[\mathfrak{C}_1].
\]
This shows that $(\mathfrak{D}/_\sim, \star_\sim)$ is commutative. 
\end{proof}

\subsection{Unique factorizability and separability} 

In Proposition \ref{P2.2} and Proposition \ref{P2.5},  we have studied the separability of a solution to the weakly coupled model. In what follows, we provide a relation between the separability of a solution and the unique factorizability of a commutative monoid. 

\begin{definition}
\begin{enumerate}
\item
Let $(M,*)$  be a semigroup. Then, an element $a\in M$ is left-cancellative (or right-cancellative) if for all $b,c\in M$, $a*b=a*c$ (or $b*a = c*a$) always implies that $b=c$. An element is said to be cancellative if it is both left- and right-cancellative. If all elements in $M$ are cancellative, then $M$ is cancellative. 
\vspace{0.2cm}
\item
Let $(M, *)$ be a commutative monoid with the identity element  denoted by $e$. If $a\in M$ has an inverse in $M$, i.e., there exists $b\in M$ such that $ab=ba=e$, then we say that $a$ is a unit of the given monoid.
\vspace{0.2cm}
\item
Let $(M,*)$ be a commutative monoid. For $a,b\in M$, we say that $a$ divides $b$ if there exists $c\in M$ such that $b=a*c$, and we write $a|b$. Then, an equivalence relation $\sim$ naturally arises, and we write $a\sim b$ if $a|b$ and $b|a$. 
\vspace{0.2cm}
\item
An element $a\in M$ is invertible if there exists $b\in M$ such that $a*b=1$.  An element $p\in M$ is a prime if it is not invertible(or it is not a unit) and 
\[
p | a*b \quad \textup{implies}\quad p|a \quad \textup{or}\quad p|b,\quad \forall a,b\in M.
\]
An element $q\in M$ is irreducible if $q$ is not invertible and 
\[
q \sim a*b \quad \textup{implies}\quad q \sim a \quad \textup{or}\quad q \sim b,\quad \forall a,b\in M.
\]
\end{enumerate}
\end{definition}
\begin{remark}
Let $\mathfrak{C}$ be a characteristic symbol of the Kuramoto model \eqref{C-3}, i.e.,
\[
\mathfrak{C}=\left(\emptyset, \frac{\kappa}{2}, \{\mathrm{i}\nu_j\}, \{e^{\mathrm{i}\theta_j^0}\}\right).
\]
If we define $\tilde{\mathfrak{C}}$ as
\[
\tilde{\mathfrak{C}}=\left(\emptyset, -\frac{\kappa}{2}, \{-\mathrm{i}\nu_j\}, \{e^{-\mathrm{i}\theta_j^0}\}\right),
\]
then we can easily check that
\[
\mathfrak{C}\star \tilde{\mathfrak{C}}=\tilde{\mathfrak{C}}\star\mathfrak{C}=\mathfrak{C}_e,
\]
and this yields
\[
[\mathfrak{C}]\star_\sim [\tilde{\mathfrak{C}}]=[\tilde{\mathfrak{C}}]\star_\sim[\mathfrak{C}]=[\mathfrak{C}_e].
\]
Finally, we can conclude that for any $\kappa$, $\{\nu_j\}$, and $\{\theta_j^0\}$, we know that $[\mathfrak{C}]$ has an inverse in $(\mathfrak{D}/\sim, \star_\sim)$. Furthermore, $[\mathfrak{C}]$ is a unit of $(\mathfrak{D}/\sim, \star_\sim)$.

\end{remark}
Note that the set of nonnegative integers under addition is a cancellative monoid and the set of positive integers under multiplication is a cancellative monoid.

\begin{lemma} \label{L5.5}
Let $(M,*)$ be a commutative cancellative monoid. If 
\[
p_1p_2\cdots p_{n_1}=q_1q_2\cdots q_{n_2},\quad \textup{where all $p_i$ and $q_j$ are primes},
\]
then $n_1=n_2=n$, and $p_i=u_iq_{\sigma(i)}$ for some permutation $\sigma\in \mathfrak{S}_n$ and unit elements $\{u_i\}_{i=1}^n$.
\end{lemma} 
\begin{proof}
For a proof, we refer the reader to \cite{J}. 
\end{proof}

The following corollary  directly follows from Lemma \ref{L5.5} in terms of the characteristic symbol.

\begin{corollary}
Let $\mathfrak{C}$ be a characteristic symbol. Then, there exists a unique factorization of $[\mathfrak C]$ up to an order and unit. In other words, there exists characteristic symbols $\mathfrak C_i$ for $i \in [n]$ such that 
\[
[\mathfrak{C}]=[\mathfrak{C}_1]\star_\sim [\mathfrak{C}_2]\star_\sim\cdots \star_\sim [\mathfrak{C}_n]
\]
and at most one $[\mathfrak{C}_i]$ can be a unit.
For the case of $n=1$ and $[\mathfrak{C}_1]$  not being a unit, then $[\mathfrak C]$ is a prime. 
\end{corollary}

\begin{remark}
\begin{enumerate}
\item
It follows from Grothendieck's theory that a commutative monoid can always be extended to an abelian group by introducing suitable inverses. In this case, such an abelian group is called the Grothendieck group. The easiest example of a Grothendieck group is the construction of $\bbz$ from $\bbn$. In this manner, we would  extend $(\mathfrak{D}/_\sim, \star_\sim)$ to an abelian group.

\vspace{0.2cm}

\item
Let $\mathfrak C$ be a characteristic symbol whose size vector is given as  $\mathbf{d}=(d_1, d_2, \cdots, d_m)$, and let $T_j$ be a solution to the Cauchy problem corresponding to $\mathfrak C$. Then, the number of entries for $T_j$ is  $\#(\mathbf{d}):=\prod_{i=1}^m d_i$. Thus, if the number $\#(\mathbf d)$ is larger, then the computational cost for numerical simulation would be higher. Thus, it is crucially required to reduce such numerical cost.  Suppose that the characteristic symbol $\mathfrak{C}$ can be decomposed as follows:
\[
\mathfrak{C}=\bs_{j=1}^n\mathfrak{C}_j,
\]
and let $\mathbf{d}^j$ be the size vector of $\mathfrak{C}_j$. Then, the number of all elements of tensors is $\sum_{j=1}^n\#(\mathbf{d}^j)$. Since  
\[
\#(\mathbf{d})=\prod_{j=1}^n\#(\mathbf{d}^j),
\]
we have 
\[
\sum_{j=1}^n\#(\mathbf{d}^j)\leq 2(n-1)+\frac{\#(\mathbf{d})}{2^{n-1}}.
\]
Thus, we can roughly measure how the computational cost is reduced when the system is decomposed into subsystems. Precisely, the ratio of reduction is at least 
\[
\frac{2(n-1)}{\#(\mathbf{d})}+\frac{1}{2^{n-1}}.
\]
In this way, our separability results would be applied to reduce the computational cost. 
\end{enumerate}
\end{remark}

\section{Emergent dynamics of weakly coupled models for low-rank tensors} \label{sec:6}
\setcounter{equation}{0}
In this section, we study the emergent dynamics of the weakly coupled model for low-rank tensors that can be derived from the double tensor model in a specific case. As introduced in Section \ref{sec:3.2}, the rank-0, 1, 2 would be related  to the Kuramoto model, the swarm sphere model and the Lohe matrix model, respectively. Since coupled systems between two models of the same ranks have been investigated in literature, we study coupled systems between two models of different ranks. We recall that weakly coupled model \eqref{D-2} of the Kuramoto model and the swarm sphere model is introduced in Section \ref{sec:4.1}.  We here focus on the weakly coupled model of the swarm sphere model and the Lohe matrix model. For this purpose, we consider the LT model whose characteristic symbol is given by the fusion operation of two characteristic symbols corresponding to the swarm sphere model and the Lohe matrix model: 
\begin{equation} \label{F-1}
\mathfrak{C} = ( \mathbf{d}, \mathfrak{K}, \{\Omega_j\}, \{T_j^0\})=(d, (\kappa, 0), \{\Omega_j\}, \{ x_j^0\} )\star \left((n, n), 
\begin{bmatrix}
0&\frac\kappa2\\
0&0
\end{bmatrix}
, \{\mi H_j\}, \{U_j^0\}\right),
\end{equation}
where natural frequency tensors $\Omega_j$ and $H_j$ are $d\times d$ skew-symmetric matrix and skew-hermitian matrix, respectively, and initial data $x_j^0$ and $U_j^0$ belong to $\bbs^{d-1}$ and $\mathbf{U}(n)$, respectively. Then, the corresponding Cauchy problem for \eqref{F-1} reads as follows:
%
\begin{align*} \label{F-1-1} 
\begin{cases}
\displaystyle {\dot x}_j= \Omega_j x_j + \frac{\kappa}{N}\sum_{k=1}^N \Big(\langle U_j, U_k \rangle_\tF x_k-\langle U_k, U_j \rangle_\tF \langle x_k, x_j\rangle x_j \Big),\quad t>0,\vspace{0.2cm}\\
\displaystyle {\dot U}_j=-\mi H_j  U_j +  \frac{\kappa}{2N}\sum_{k=1}^N\langle x_j, x_k\rangle \Big(U_k-U_jU_k^\dagger U_j  \Big),\vspace{0.2cm}\\
(x_j,U_j)(0)=(x_j^0,U_j^0) \in\bbs^{d-1}\times \mathbf{U}(n) , \quad j\in[N],
\end{cases}
\end{align*}
where $H_j$ is an $n\times n$ Hermitian matrix. However for the simplicity of calculation, we will consider $\mathbf{SO}(n)$ instead of $\mathbf{U}(n)$. For this, we introduce the real characteristic symbol.

\begin{definition} \label{D3.3}
Let $\mathfrak{C}=(\mathbf{d}, \mathfrak{K}, \{A_j\}, \{T_j^0\})$ be a characteristic symbol. If $A_j$ and $T_j^0$ are real tensors for all $j \in [N]$, then we call $\mathfrak{C}$ as a real characteristic symbol.
\end{definition}
\begin{remark} \label{R3.1} 
It follows from the uniqueness of a solution to the LT model \eqref{LT} that if $A_j$ and $T_j^0$ are real tensors, then a solution $T_j$ is also real-valued tensor. Thus, the following assertion holds: Let $\mathfrak{C}$ be a real characteristic symbol. If $\{T_j(t)\}_{j=1}^N$ is a solution to the Cauchy problem corresponding to $\mathfrak C$, then $T_j(t)$ is real tensor for all $t\geq0$ and $j\in [N]$.
\end{remark}

\vspace{0.5cm}

On the other hand, the following proposition can be directly deduced from the definition of the fusion operation $\star$, Definition \ref{D3.3} and Remark \ref{R3.1}.
\begin{proposition} \label{P5.2}
Let $\{\mathfrak{C}_\ell\}_{\ell=1}^n$ be a set of real characteristic symbols. Then, $\displaystyle \bs_{\ell=1}^n \mathfrak{C}_{\ell}$ is also a real characteristic symbols. 
\end{proposition}
Now, we consider the following Cauchy problem on $\bbs^{d-1} \times \mathbf{SO}(n)$: 
\begin{equation} \label{F-3} 
\begin{cases}
\displaystyle {\dot x}_j= \Omega_j x_j + \frac{\kappa}{N}\sum_{k=1}^N\langle U_j, U_k \rangle_\tF(x_k-\langle x_k, x_j\rangle x_j),\quad t>0, \vspace{0.2cm}\\
\displaystyle {\dot U}_j=A_j U_j +  \frac{\kappa}{2N}\sum_{k=1}^N\langle x_j, x_k\rangle(U_k-U_jU_k^\top U_j ),\vspace{0.2cm}\\
(x_j,U_j)(0)=(x_j^0,U_j^0) \in\bbs^{d-1}\times \mathbf{SO}(n),
\end{cases}
\end{equation}
where $A_j$ is a $n\times n$ skew-symmetric matrix. \newline

It follows from Remark \ref{R3.1} or straightforward calculation that $U_i$ and $x_i$ remain in $\mathbf{SO}(n)$ and $\bbs^{d-1}$, respectively. Before we study the emergent dynamics of the whole ensemble, we recall several definitions in relation to the  collective dynamics of \eqref{F-3}.
\begin{definition} \label{D6.1}
\begin{enumerate}
\item System \eqref{F-3} exhibits complete aggregation if the following estimate holds:
\[
\lim_{t\to\infty} \max_{1\leq i,j\leq N}\Big( |x_i(t) - x_j(t) | + \|U_i(t) - U_j(t)\|_\tF\Big) = 0.
\]
\item System \eqref{F-3} exhibits practical aggregation if the following estimate holds:
\[
\lim_{\kp\to\infty} \limsup_{t\to\infty} \max_{1\leq i,j\leq N}\Big( |x_i(t) - x_j(t) | + \|U_i(t) - U_j(t)\|_\tF\Big) = 0.
\]
\item The state $\{(x_j,U_j)\}$ tends to a partially locked state if the following relation holds:
\[
\lim_{t\to\infty} \max_{1\leq i,j\leq N}\Big( |x_i(t) - x_j(t) |\Big) =0\quad \textup{and} \quad \exists~~\lim_{t\to\infty} U_i(t) U_j^\top (t). 
\]
\end{enumerate}
\end{definition}

\vspace{0.2cm}

In the following subsections, we study the emergent dynamics of \eqref{F-3} depending on the dissimilarity of natural frequency tensors:
\begin{align*} \label{F-4}
\begin{aligned}
&\textup{(i) Homogeneous ensemble:} \quad \Omega_j \equiv \Omega,\quad A_j \equiv A,\quad j \in [N]. \\
&\textup{(ii) Heterogeneous ensemble:} \quad \Omega_i \neq \Omega_j,\quad A_i \neq A_j,\quad \textup{for some $i\neq j$.} \\
&\textup{(iii) Partially homogeneous ensemble:} \quad \Omega_j = \Omega,\quad A_i \neq A_j,\quad \textup{for some $i\neq j$.} 
\end{aligned}
\end{align*}

\subsection{Homogeneous ensemble} \label{sec:6.1}
In this subsection, we consider a homogeneous ensemble:
\begin{equation*} \label{F-4-1}
\Omega_i = O,\quad A_i  = O, \quad i \in [N].
\end{equation*}
In this case, system \eqref{F-3} with zero frequency tensors reads as 
\begin{align}\label{F-5}
\begin{cases}
\displaystyle {\dot x}_j= \frac{\kappa}{N}\sum_{k=1}^N\langle U_j, U_k \rangle_\tF(x_k-\langle x_k, x_j\rangle x_j),\quad t>0, \vspace{0.2cm}\\
\displaystyle {\dot U}_j=\frac{\kappa}{2N}\sum_{k=1}^N\langle x_j, x_k\rangle(U_k-U_jU_k^\top U_j ),\vspace{0.2cm}\\
x_j(0)=x_j^0\in\bbs^{d-1},\quad U_j(0)=U_j^0\in\SOn,\quad j\in [N].
\end{cases}
\end{align}
For aggregation estimate, we set 
\begin{equation} \label{F-6}
h_{ij}:= \langle x_i,x_j\rangle,\quad G_{ij}:= U_iU_j^\top,\quad g_{ij} := \langle U_i,U_j\rangle_\tF=\textup{tr}(G_{ij}),
\end{equation}
and define the functionals:
\[ \mathcal A(\mathcal X):=\min_{1\leq i\neq j\leq N}\langle x_i,x_j\rangle,\quad \mathcal D(\mathcal H):=\max_{1\leq i,j\leq N} |1-h_{ij}|, \quad \mathcal D(\mathcal U) := \max_{1\leq i,j\leq N } \|U_i - U_j \|_\tF.
\]
Below, we provide  a positively invariant region for \eqref{F-5}. 
\begin{lemma} \label{L6.1}
Suppose initial data $\{(x_i^0, U_i^0))\}_{i=1}^N$ satisfy
\begin{equation} \label{F-7}
\mathcal A(\mathcal X^0)>0,\quad \mathcal D(\mathcal U^0)<\sqrt2,
\end{equation}
and let $(\mathcal X,\mathcal U)$ be a solution to \eqref{F-5}. Then, one has
\begin{equation*}
\mathcal A(\mathcal X(t))>0 \quad \mathcal D(\mathcal U(t))<\sqrt2,   \quad t>0.
\end{equation*}
\end{lemma}
\begin{proof}
It follows from \eqref{F-5} and \eqref{F-6} that
\begin{align} \label{F-8}
\begin{aligned}
&{\dot h}_{ij} = \frac\kp N \sum_{k=1}^N \Big( g_{ik}(h_{kj} - h_{ik}h_{ij}) + g_{jk} (h_{ik} - h_{kj}h_{ij}) \Big), \\
&{\dot G}_{ij} = \frac\kp N \sum_{k=1}^N \Big( h_{ik}( G_{kj} - G_{ik} G_{ij}) + h_{jk}(G_{ik} - G_{ij}G_{kj}) \Big).
\end{aligned}
\end{align}
Since a solution to \eqref{F-5} is analytic and we assume \eqref{F-7}, there exists a finite time $T_1>0$ such that
\begin{equation*} \label{F-9}
\mathcal A(\mathcal X(t))>0,\quad \mathcal D(\mathcal U(t))<\sqrt2,\quad t\in [0,T_1).
\end{equation*}
Moreover, there also exists a finite time $T_2>0$ such that
\begin{equation*} \label{F-10}
\mathcal A(\mathcal X(t)) = h_{i_t j_t}(t),\quad t\in [0,T_2).
\end{equation*}
We set 
\[  T_*:= \min \{ T_1,T_2\}. \]
Since the following relation holds:
\begin{equation*} \label{F-11}
\|U_i - U_j\|_\tF^2 = 2( d - g_{ij}) <2, \quad  \textup{or equivalently,} \quad  g_{ij}>d-1,\quad t\in [0,T_*),
\end{equation*}
one has 
\begin{align*} 
\begin{aligned}
\frac\d\dt \mathcal A(\mathcal X) &\geq \frac\kp N \sum_{k=1}^N\Big( g_{i_t k} (h_{kj_t} - h_{i_tk} \mathcal A(\mathcal X)) + g_{j_t k} (h_{i_tk} - h_{kj_t} \mathcal A( \mathcal X))  \Big) \\
&> \frac{\kp (n-1)}{N}\mathcal A(\mathcal X)  \sum_{k=1}^N (1-h_{i_tk} + 1-h_{kj_t}) >0, \quad t\in [0,T_*).
\end{aligned}
\end{align*}
On the other hand, we rewrite $\eqref{E-5}_2$ as
\begin{align*}
{\dot g}_{ij} & =\frac\kp N \sum_{k=1}^N h_{ik}\Big( \|I_n-G_{ij}\|_\tF^2 - \|I_n-G_{kj}\|_\tF^2) + h_{jk} ( \|I_n-G_{ij}\|_\tF^2 - \|I_n-G_{kj}\|_\tF^2\Big)  \\
& \hspace{0.5cm}+ (2- \|I_n - G_{ij}\|_\tF^2) \cdot \frac{\kp}{2N} \sum_{k=1}^N \Big( h_{ik} \| I_n - G_{ik}\|_\tF^2 + h_{jk} \|I_n - G_{jk}\|_\tF^2\Big) .
\end{align*}
By the same argument, one has 
\begin{align*} 
\begin{aligned}
\frac\d\dt \mathcal D(\mathcal U)^2 &= -\frac\kp N \sum_{k=1}^N \Big( h_{i_tk} ( \mathcal D(\mathcal U)^2 - \|I_n - G_{kj_t}\|_\tF^2)  + h_{j_tk} (\mathcal D(\mathcal U)^2 - \|I_n - G_{ki_t}\|_\tF^2) \Big) \\
&\hspace{0.5cm} -(2- \mathcal D(\mathcal U)^2) \cdot \frac{\kp}{2N} \sum_{k=1}^N \Big(h_{i_tk} \|I_n - G_{ki_t} \|_\tF^2 + h_{j_tk} \|I_n - G_{kj_t}\|_\tF^2\Big)\\
&<0,\quad t\in [0,T_*).
\end{aligned}
\end{align*}
Hence, if we continue the process above successively, then we obtain the desired assertion.
\end{proof}

Now, we are ready to provide a sufficient condition leading to complete aggregation for \eqref{E-1}.

\begin{theorem} \label{T6.1}
Suppose the initial data  $(\mathcal X^0,\mathcal U^0)$ satisfy 
\[
 \mathcal A(\mathcal X^0)>0,\quad \mathcal D(\mathcal U^0)<\sqrt 2,
\]
and let $(\mathcal X,\mathcal U)$ be a solution to \eqref{F-5}. Then, system \eqref{F-5} exhibits complete aggregation:
\[
\lim_{t\to\infty}\Big( \mathcal D(\mathcal X(t)) + \mathcal D(\mathcal U(t)) \Big)=0.
\]
Moreover, the convergence rate is exponential. 
\end{theorem}

\begin{proof}
Since we assume \eqref{E-20}, we can apply Lemma \ref{L5.1}:
\begin{equation*}
h_{ij}(t) >0,\quad g_{ij}(t) >d-1,\quad t>0,\quad i\neq j.
\end{equation*}
Note that 
\[
\mathcal D(\mathcal X(t))^2 = \max_{1\leq i\neq j \leq N } |x_i(t) - x_j(t)|^2 = 2 (1- \mathcal A(\mathcal X(t))).
\]
Now we choose extremal indices $(i_t,j_t)$ such that 
\[  \mathcal A(\mathcal X(t)) = \langle x_{i_t},x_{j_t}\rangle. \]
Then, one has 
\begin{align*}
1-h_{i_tk} + 1-h_{j_tk} &= \frac12(|x_{i_t} - x_k|^2 + |x_{j_t} - x_k|^2 ) >\frac14(|x_{i_t} - x_k| + |x_{j_t} - x_k|)^2 \\
& >\frac14 |x_{i_t} - x_{j_t} |^2  = \frac14\mathcal D(\mathcal X)^2 = \frac12 ( 1- \mathcal A(\mathcal X)).
\end{align*}
It follows from \eqref{E-5} and \eqref{E-10} that
\[
\frac\d\dt \mathcal A(\mathcal X) > \frac{\kp(n-1)}{N} \mathcal A(\mathcal X) \sum_{k=1}^N (1- h_{i_tk} + 1-h_{j_tk}) > \frac{\kp(n-1)}{2} \mathcal A(\mathcal X) ( 1- \mathcal A( \mathcal X)).
\]
Thus, we obtain the desired exponential convergence of $\mathcal D(\mathcal X)$. On the other hand for $\mathcal D(\mathcal U)$, we use \eqref{E-12} to find
\begin{align*}
\frac\d\dt \mathcal D(\mathcal U)^2 < -\frac{\kp}{4 } \mathcal A(\mathcal X^0)(2-\mathcal D(\mathcal U)^2) \mathcal D(\mathcal U)^2.
\end{align*}
This shows the desired assertion. 
\end{proof}
\begin{remark}
The initial conditions in Theorem \ref{T6.1} leading to complete aggregation appear in \cite{C-H} and \cite{H-R} for decoupled systems, respectively. To be more specific, by letting $U_i \equiv U$ in $\eqref{E-1}_1$, the model reduces to the model on $\bbs^{d-1}$:
\begin{equation} \label{F-14}
\dot x_j = \frac\kp N \sum_{k=1}^N (x_k - \langle x_i,x_k\rangle x_i).
\end{equation}
Then, the authors in \cite{C-H} showed that if the initial data satisfy $\mathcal A(\mathcal X^0)> 0$, then  system \eqref{F-14} displays complete aggregation. Similarly by setting $x_j \equiv x$ in $\eqref{E-1}_2$, it reduces to the model on $\SOd$:
\begin{equation} \label{F-15}
\dot U_j = \frac{\kp}{2N} \sum_{k=1}^N ( U_k - U_jU_k^\top U_j).
\end{equation}
Although the model in \cite{H-R} is defined on the unitary group $\mathbf{U}(d)$, a similar calculation leads to the same result, i.e., if the initial data satisfy $\mathcal D(\mathcal U^0)<\sqrt2$, then \eqref{F-15} exhibits complete synchronization. 
\end{remark}
\subsection{Heterogeneous ensemble} \label{sec:6.2}

In this subsection, we consider a heterogeneous ensemble:
\begin{equation} \label{F-16}
\Omega_i \neq \Omega_j,\quad A_i \neq A_j \quad \textup{for some $i\neq j \in [N]$.}
\end{equation}
Due to the dissimilarity between natural frequencies, we might not expect the emergence of complete aggregation. Instead, our goal is to provide a sufficient framework leading to practical aggregation.

\begin{theorem} \label{T6.2} 

Suppose initial data and system parameters satisfy 
\begin{equation} \label{F-16-1}
\mathcal A(\mathcal X^0) > \frac12, \quad  \mathcal D(\mathcal U^0)^2 < 1+ \sqrt{ 1-\frac{4 \mathcal D(\mathcal A)}{\kp \mathcal A(\mathcal X^0) }  } ,\quad  \kp >\min\left\{  \frac{4 \mathcal D(\mathcal A)}{\mathcal A(\mathcal X^0)},\frac{8\mathcal D(\Omega) }{\mathcal A(\mathcal U^0)}\right\} ,
\end{equation}
and let $(\mathcal X, \mathcal U)$ be a solution to \eqref{F-3} with \eqref{F-16}. Then, practical aggregation emerges asymptotically: 
\[
\limsup_{t\to\infty} \mathcal D(\mathcal X(t)) \leq 1 -\sqrt{1-\frac{8\mathcal D(\Omega)}{\kp \mathcal A(\mathcal U^0)}}  \quad\textup{and}\quad  \limsup_{t\to\infty} \mathcal D(\mathcal U(t)) \leq 1- \sqrt{ 1-\frac{4 \mathcal D(\mathcal A)}{\kp \mathcal A(\mathcal X^0) }  }. 
\]
\end{theorem}
\begin{proof}
 By slightly modifying \eqref{F-8}, one has
 \begin{align} \label{F-17}
 \begin{aligned}
&{\dot  h}_{ij} = \langle x_i,(\Omega_j -\Omega_i)x_j\rangle + \frac\kp N \sum_{k=1}^N \Big( g_{ik}(h_{kj} - h_{ik}h_{ij}) + g_{jk} (h_{ik} - h_{kj}h_{ij}) \Big), \\
&{\dot G}_{ij} = A_i G_{ij} - G_{ij}A_j + \frac\kp N \sum_{k=1}^N \Big( h_{ik}( G_{kj} - G_{ik} G_{ij}) + h_{jk}(G_{ik} - G_{ij}G_{kj}) \Big).
 \end{aligned}
 \end{align}
 Now, we use the conditions on initial data:
 \[
 \mathcal A(\mathcal X^0)>\frac12, \quad \mathcal A(\mathcal U^0) > d-\frac12\left( 1+ \sqrt{ 1-\frac{4 \mathcal D(\mathcal A)}{\kp \mathcal A(\mathcal X^0) }  }\right)>0,
 \]
and the continuity of solutions to see that  there exists a finite time $T_*>0$ such that
\[
\mathcal A(\mathcal X(t))>0, \quad \mathcal A(\mathcal U(t))>0, \quad t\in [0,T_*). 
\]
It follows from \eqref{E-10} and \eqref{F-17} that
 \begin{align*}
 \frac\d\dt \mathcal A(\mathcal X) >\frac{\kp}{2} \mathcal A(\mathcal U)  \mathcal A(\mathcal X)(1- \mathcal A(\mathcal X)) - \mathcal D(\Omega), \quad t\in [0,T_*).
 \end{align*}
At $t=0$, since $\mathcal A(\mathcal X^0)>\frac12$, one has 
\[
 \frac\d\dt \mathcal A(\mathcal X)\biggl|_{t=0} >\frac{\kp}{2} \mathcal A(\mathcal U)  \mathcal A(\mathcal X)(1- \mathcal A(\mathcal X)) - \mathcal D(\Omega)\biggl|_{t=0} = \frac{\kp \mathcal A(\mathcal U^0)}{2} \mathcal A(\mathcal X^0) (1 - \mathcal A(\mathcal X^0))-\mathcal D( \Omega)  >0.
\]
Thus, we see that $\mathcal A(\mathcal X)$ begins to increase 
\[
\mathcal A(\mathcal X(t))>\mathcal A(\mathcal X^0),\quad t\in [0,T_{**}).
\]
For $t\in [0,T_{**})$, the diameter $\mathcal D(\mathcal U)$ satisfies  
 \begin{align*} 
\begin{aligned}\label{F-18}
&\frac\d\dt \mathcal D(\mathcal U)^2 = -\frac\kp N \sum_{k=1}^N \Big( h_{i_tk} ( \mathcal D(\mathcal U)^2 - \|I_n - G_{kj_t}\|_\tF^2)  + h_{j_tk} (\mathcal D(\mathcal U)^2 - \|I_n - G_{ki_t}\|_\tF^2) \Big) \\
& -(2- \mathcal D(\mathcal U)^2) \cdot \frac{\kp}{2N} \sum_{k=1}^N h_{i_tk} \|I_n - G_{ki_t} \|_\tF^2 + h_{j_tk} \|I_n - G_{kj_t}\|_\tF^2) + \textup{tr} (G_{i_tj_t}(A_i - A_j)) \\
&< -\mathcal A(\mathcal X^0) (2- \mathcal D(\mathcal U)^2) \cdot \frac{\kp}{2N } \sum_{k=1}^N \Big(  \|I_n - G_{ki_t} \|_\tF^2 + \|I_n - G_{kj_t}\|_\tF^2 \Big) + \mathcal D(\mathcal A) \\
&< -\frac{\kp \mathcal A(\mathcal X^0)}{2}  (2- \mathcal D(\mathcal U)^2)D(\mathcal U)^2 + \mathcal D(\mathcal A) .
\end{aligned}
\end{align*}
Since $\mathcal D({\mathcal U}^0)$ satisfies \eqref{F-16-1}, then we have
\[
\frac\d\dt \mathcal D(\mathcal U)^2 \biggl|_{t=0} <0 ,
\]
which implies that $\mathcal D(\mathcal U)$ starts to increase and $\mathcal A(\mathcal U)$ start to decrease. Thus, if we continue this process, then we conclude 
\[
\mathcal A(\mathcal X(t))>\frac12, \quad  \mathcal D(\mathcal U(t))^2 < 1+ \sqrt{ 1-\frac{4 \mathcal D(\mathcal A)}{\kp \mathcal A(\mathcal X^0) }  },\quad t>0. 
\]
In addition, since $\mathcal A(\mathcal X)$ and $\mathcal D(\mathcal U)$ satisfy for $t>0$, 
\begin{align*}
& \frac\d\dt \mathcal A(\mathcal X) >\frac{\kp}{2} \mathcal A(\mathcal U^0)  \mathcal A(\mathcal X)(1- \mathcal A(\mathcal X)) - \mathcal D(\Omega), \\
 &\frac\d\dt \mathcal D(\mathcal U)^2 < -\frac{\kp \mathcal A(\mathcal X^0)}{2}  (2- \mathcal D(\mathcal U)^2)D(\mathcal U)^2 + \mathcal D(\mathcal A).
\end{align*}
By dynamical systems theory, there exists a finite entrance time $T_1>0$ such that
\begin{align*}
&\mathcal A(\mathcal X(t))> \frac12\left(  1 + \sqrt{1-\frac{8\mathcal D(\Omega)}{\kp \mathcal A(\mathcal U^0)}}\right) \quad\textup{or}\quad  \mathcal D(\mathcal X(t))^2 < 1 -\sqrt{1-\frac{8\mathcal D(\Omega)}{\kp \mathcal A(\mathcal U^0)}}, \\
&\mathcal D(\mathcal U(t))^2  < 1- \sqrt{ 1-\frac{4 \mathcal D(\mathcal A)}{\kp \mathcal A(\mathcal X^0) }  } ,\quad t>T_1. 
\end{align*}
This shows the desired assertion. 
 \end{proof}

\subsection{Partially homogeneous ensemble} \label{sec:6.3}
In this subsection, we consider a partially homogeneous ensemble:
\begin{equation} \label{F-19}
\Omega_\ell \equiv \Omega,\quad \ell=1,\cdots,N, \quad A_i \neq A_j \quad \textup{for some $i\neq j$.}
\end{equation}
Although the natural frequencies $\{\Omega_j\}$ are the same, complete aggregation for $\{x_j\}$ can emerge. However, since $\{A_j\}$ are different in general, $\{U_j\}$ tends to locked state. Thus, our goal in this subsection is to find a sufficient framework such that the following relation holds:
\[
\lim_{t\to\infty}  |x_i (t) - x_j(t)|=0 \quad \textup{and}\quad \exists~~\lim_{t\to\infty} U_i(t) U_j^\top (t). 
\] 
First, we show that relative distances for $\{x_j\}$ tend to zero and the maximal diameter $\mathcal D(\mathcal U)$ can be controlled in terms of the coupling strength $\kp$.

\begin{lemma} \label{L6.2} 
Suppose initial data and system parameters satisfy 
\begin{align*}
 \mathcal A(\mathcal X^0)>0,\quad  \mathcal D(\mathcal U^0)^2 < 1+ \sqrt{ 1-\frac{4 \mathcal D(\mathcal A)}{\kp \mathcal A(\mathcal X^0) }  },\quad \kp > \frac{4 \mathcal D(\mathcal A)}{\mathcal A(\mathcal X^0)}, 
\end{align*}
and   let $(\mathcal X, \mathcal U)$ be a solution to \eqref{F-3} with \eqref{F-19}. Then, there exists $T_* > 0$ such that 
\[
\lim_{t\to\infty} \mathcal D( \mathcal X(t))= 0, \quad \mathcal D(\mathcal U(t)^2< 1-\sqrt{ 1-\frac{4 \mathcal D(\mathcal A)}{\kp \mathcal A(\mathcal X^0) }  },\quad t>T_*. 
\]
\end{lemma}

\begin{proof}
We closely follow Theorem \ref{T6.3} to conclude that
\begin{align} \label{F-20}
\begin{aligned}
& \frac\d\dt \mathcal A(\mathcal X) > \frac\kp2 \mathcal A(\mathcal U^0) \mathcal A(\mathcal X)(1-\mathcal A(\mathcal X)), \\
& \frac\d\dt \mathcal D(\mathcal U)^2 <  -\frac{\kp \mathcal A(\mathcal X^0)}{2}  (2- \mathcal D(\mathcal U)^2)D(\mathcal U)^2 + \mathcal D(\mathcal A) .
\end{aligned}
\end{align}
Thus, inequalities \eqref{F-20} are verified. For the convergence of $\mathcal D(\mathcal X)$, we use $\eqref{F-20}_1$ to show that $\mathcal A(\mathcal X)$ tends to 1 and recall the relation $\mathcal D(\mathcal X)^2 = 2(1-\mathcal A(\mathcal X))$. 
\end{proof}
In order to establish the desired convergence of $G_{ij} = U_iU_j^\top$, we adopt the strategy in \cite{H-R}. Recall that $G_{ij}$ satisfies
\begin{equation} \label{F-21}
\frac\d\dt  G_{ij} = A_i G_{ij} - G_{ij}A_j + \frac\kp N \sum_{k=1}^N \Big( h_{ik}( G_{kj} - G_{ik} G_{ij}) + h_{jk}(G_{ik} - G_{ij}G_{kj}) \Big),
\end{equation} 
and it follows from Lemma \ref{L6.2} that all $h_{ij}$ converges to 1. Thus, we rewrite \eqref{F-21} as 
\begin{equation*}  \label{F-22}
\frac\d\dt  G_{ij} = A_i G_{ij} - G_{ij}A_j + \frac\kp N \sum_{k=1}^N \Big(  ( G_{kj} - G_{ik} G_{ij}) +  (G_{ik} - G_{ij}G_{kj}) \Big) +\kp \varepsilon_{ij} ,
\end{equation*}
where the auxiliary functional $\varepsilon_{ij}$ is defined as
\begin{equation*}
\varepsilon_{ij} := \frac1N \sum_{k=1}^N \Big( (h_{ik}-1)(G_{kj} - G_{ik}G_{ij} )+ ( h_{jk}-1)(G_{ik} - G_{ij} G_{kj}) \Big) .
\end{equation*} 
Let $\{(\tilde x_j, \tilde U_j)\}$ be any solution to \eqref{F-3} with \eqref{F-19}. Then, we write
\begin{align*}
&\tilde G_{ij} := \tilde U_i \tilde U_j^\top,\quad \tilde h_{ij} := \langle \tilde x_i,\tilde x_j\rangle,\\
&\tilde \varepsilon_{ij} := \frac1N \sum_{k=1}^N \Big( ( \tilde h_{ik}-1)(\tilde G_{kj} -\tilde  G_{ik}\tilde G_{ij} )+ ( \tilde h_{jk}-1)(\tilde G_{ik} - \tilde G_{ij} \tilde G_{kj}) \Big) . 
\end{align*} 
Then, we observe 
\begin{align*}
(G_{ij} - \tilde G_{ij})' &= A_i ( G_{ij} - \tilde G_{ij}) - (G_{ij} - \tilde G_{ij})A_j  + \veps_{ij} - \tilde \veps_{ij} \\
&\hspace{0.5cm}+ \frac\kp N \sum_{k=1}^N \Big(  ( G_{kj} - G_{ik} G_{ij}) -( \tilde G_{kj} - \tilde G_{ik} \tilde G_{ij}) \\
&\hspace{3cm}+  (G_{ik} - G_{ij}G_{kj})-(\tilde G_{ik} - \tilde G_{ij}\tilde G_{kj}) \Big) .
\end{align*}
Define the diameter between $G_{ij}$:
\[
d(\mathcal U,\tilde {\mathcal U}) := \max_{1\leq i, j\leq N} \| G_{ij} - \tilde G_{ij} \|_\tF = \max_{1\leq i,j\leq N} \|U_iU_j^\top - \tilde U_i \tilde U_j^\top\|_\tF. 
\]
Then, after algebraic manipulation, one finds a differential inequality for $d(\mathcal U,\tilde {\mathcal U})$: 
\[
\frac\d\dt d(\mathcal U,\tilde {\mathcal U})  \leq -2\kp (1-2\mathcal D(\mathcal U))d(\mathcal U,\tilde {\mathcal U})  +  \|\veps_{ij}- \tilde\veps_{ij}\|_\tF.
\]
Below, we recall an elementary lemma whose proof can be found  in Lemma A.1 of \cite{C-H-H-J-K}. 
 \begin{lemma}   
Suppose that $f = f(t) \in (L^1\cap L^\infty)(\bbr_+)$ is a positive and bounded function whose argument tends to zero, as $t \to \infty$, and let $y=y(t)$ be a nonnegative $\mathcal C^1$-function satisfying 
\[
\begin{cases}
\dot y    \leq -ay+ f, \quad t>0, \\
y(0) = y_0,
\end{cases}
\]
where $a$ is a positive constant. Then, $y=y(t)$ tends to zero with an integrable convergence rate. 
\begin{equation*}
\lim_{t\to\infty} y(t) = 0.
\end{equation*}
\end{lemma}

\vspace{0.2cm}

\noindent Now, we are ready to state the main theorem. 
\begin{theorem} \label{T6.3} 
Suppose initial data and system parameters satisfy 
\begin{equation*}
 \mathcal A(\mathcal X^0)>0,\quad  \mathcal D(\mathcal U^0)^2 < 1+ \sqrt{ 1-\frac{4 \mathcal D(\mathcal A)}{\kp \mathcal A(\mathcal X^0) }  },\quad \kp > \frac{4 \mathcal D(\mathcal A)}{\mathcal A(\mathcal X^0)}, 
\end{equation*} 
and   let $(\mathcal X, \mathcal U)$ be a solution to \eqref{F-3} with \eqref{F-19}.  Then, a solution to system \eqref{F-3} tends to the partially locked state. 
\end{theorem}

\begin{proof}
It follows from Lemma \ref{L6.2} that relative distances for $\mathcal X$ tend to zero. Thus, it suffices to consider the dynamics of $\mathcal U$. From the arguments in the above, we show that $d(\mathcal U,\tilde {\mathcal U})$ tends to zero. As a next step, since our system is autonomous, for any $T>0$, we choose $\tilde U_j$ and $\tilde V_j$ as 
\begin{equation*}
\tilde U_j(t) = U_j(t+T). 
\end{equation*} 
By discretizing the time $t\in \bbr_+$ as $n\in \bbz_+$ and setting $T=m \in \bbz_+$, we conclude that $\{U_i(n)U_j^\top(n)\}_{n\in \bbz_+}$ becomes a Cauchy sequence in the complete spaces $ \mathbf{SO}(n)$. Hence, there exists constant unitary matrices $U_{ij}^\infty \in  \mathbf{SO}(n)$ such that
\begin{equation*}
\lim_{t\to\infty} \| U_i(t)U_j^\top(t) - U_{ij}^\infty \|_\tF = 0.
\end{equation*}
This establishes the convergence of $U_iU_j^\top$. 
\end{proof} 

\section{Conclusion}\label{sec:7}
\setcounter{equation}{0}
In this paper, we have provided a systematic algebraic approach for the weak coupling of the Lohe tensor models with possibly different ranks and sizes. The Lohe tensor model incorporates well-studied first-order aggregation models such as the Kuramoto model, the swarm sphere model, the Lohe matrix model that  have been extensively studied in previous literature. For the systematic modeling of the  weak coupling of two models, we introduced a characteristic symbol corresponding to a  given Cauchy problem to the Lohe tensor model, and it consists of four components, namely a size vector, a coupling strength tensor, set of natural frequency tensors and initial configuration, and then we also introduce a fusion operation $\star$ which provides a binary operation on the space of all characteristic symbols. It is worthwhile to mention that such fusion operation can form a monoid structure. In this way, we can derive a weakly coupled model for multiple Lohe tensor models. As a concrete example, we presented  a weakly coupled model for the swarm sphere model and the Lohe matrix model. For this weakly coupled model, we also provided several sufficient frameworks leading to collective emergent dynamics. Of course, there are many untouched issues in this work. For example, we have not studied the emergent dynamics of the multiple tensor model and did not investigate the solution structure of the proposed model. These interesting issues will be left for   future work.

\end{document}